\documentclass[a4paper, 11pt]{article}
\usepackage[a4paper, total={6in, 8in}]{geometry}
\usepackage[english]{babel}
\usepackage[utf8]{inputenc}
\usepackage{amsmath}
\usepackage{setspace}
\usepackage{slashed} 
\usepackage{enumitem}
\usepackage{changepage}
\usepackage{amsfonts}
\usepackage{amssymb}
\usepackage{tikz-cd}
\usepackage{adjustbox}
\usepackage{mathrsfs}
\usepackage{amsbsy}
\usepackage{xcolor}
\usepackage{stmaryrd}
\usepackage{extarrows}
\usepackage{varwidth}
\usepackage{array}
\usepackage{tikz}
\usepackage{tikz-cd}
\usetikzlibrary{matrix}
\usepackage{gensymb}
\usepackage{bm}
\usepackage{graphicx}
\usepackage{tensor}
\usepackage{hyperref}
\hypersetup{%
    pdfborder = {0 0 0}
}
\usepackage{mathtools}
\usepackage{braket}
\usepackage{bbm} 
\usepackage[symbol]{footmisc}

\usepackage[backend=biber,
            style=numeric,
            bibstyle=numeric,
            citestyle=numeric,
            sorting=nyt,
            hyperref=true,
            backref=true,
            giveninits=true,
            doi=false,
            isbn=false,
            url=true,
            ,maxbibnames=99]{biblatex}
 \renewbibmacro{in:}{}
\usepackage{amsthm}
\theoremstyle{plain}
\newtheorem{theorem}{Theorem}[section]
\newtheorem{lemma}[theorem]{Lemma}
\newtheorem{proposition}[theorem]{Proposition}
\newtheorem{corollary}[theorem]{Corollary}

\theoremstyle{definition}
\newtheorem{definition}[theorem]{Definition}
\newtheorem{example}[theorem]{Example}
\newtheorem{remark}[theorem]{Remark}

\newtheorem{prop/def}[theorem]{Proposition/Definition}
\newtheorem{assumption}[theorem]{Assumption}

\numberwithin{figure}{section}
\numberwithin{table}{section}
\numberwithin{equation}{section}

\makeatletter
\newcommand*{\transpose}{%
  {\mathpalette\@transpose{}}%
}
\newcommand*{\@transpose}[2]{%
  \raisebox{\depth}{$\m@th#1\intercal$}%
}
\makeatother

\newcommand*\cs{\textnormal{cs }}
\newcommand*\Z{\mathbb{Z}}

\newcommand*\R{\mathbb{R}}
\newcommand*\C{\mathbb{C}}
\newcommand*\mC{\mathcal{C}}

\newcommand*\mV{\mathcal{V}}
\newcommand*\cla{\textnormal{cla}}

\newcommand*\DNC{D}

\DeclareMathOperator{\an}{an}
\DeclareMathOperator{\Sq}{Sq}

\newcommand{\ov}[1]{\overline{#1}}
\newcommand{\un}[1]{\underline{#1}}

\newcommand\numberthis{\addtocounter{equation}{1}\tag{\theequation}}
\newcommand\SmallMatrix[1]{{%
  \tiny\arraycolsep=0.3\arraycolsep\ensuremath{\begin{pmatrix}#1\end{pmatrix}}}}

\title{Orientations for DT invariants on quasi-projective Calabi--Yau 4-folds}
\date{}
\author{Arkadij Bojko\thanks{Address: Professur für Mathematik ETH Zürich, HG J 16.2 Rämistrasse 101, Zürich 8092, Switzerland , E-mail: arkadij.bojko@math.ethz.ch
}}

\begin{document}

\maketitle
\begin{abstract}
        For a Calabi--Yau 4-fold $(X,\omega)$, where $X$ is quasi-projective and $\omega$ is a nowhere vanishing section of its canonical bundle $K_X$, the (derived) moduli stack of compactly supported perfect complexes $\mathcal{M}_X$ is $-2$-shifted symplectic and thus has an orientation bundle $O^\omega\to \mathcal{M}_X$ in the sense of Borisov--Joyce \autocite{BJ}, necessary for defining Donaldson--Thomas type invariants of $X$. We extend first the orientability result of \cite{CGJ,JU} to projective spin 4-folds. Then for any smooth projective compactification $Y$, such that $D=Y\backslash X$ is strictly normal crossing, we define orientation bundles on the stack $\mathcal{M}_{Y}\times_{\mathcal{M}_D}\mathcal{M}_{Y}$ and express these as pullbacks of orientation bundles from gauge theory, constructed using positive Dirac operators on the double of $X$. As a result, we relate the orientation bundle $O^\omega\to \mathcal{M}_X$ to a gauge-theoretic orientation on the classifying space of compactly supported K-theory. Using orientability of the latter under the compactly supported version of the assumption in \cite{JU}, we obtain orientability of $\mathcal{M}_X$. We also prove orientability of moduli spaces of stable pairs and Hilbert schemes of proper subschemes. Finally, we consider the compatibility of orientations under direct sums. 
    \end{abstract}

\section{Introduction}
\subsection{Background and results}
Donaldson-Thomas type theory for Calabi--Yau 4-folds has been developed by Borisov--Joyce \autocite{BJ}, Cao--Leung \autocite{CL} and Oh--Thomas \autocite{OT}\,. The construction relies on having Serre duality 
$$
\textnormal{Ext}^2(E^\bullet,E^\bullet)\cong \textnormal{Ext}^2(E^\bullet,E^\bullet)^*\,.
$$
for a perfect complex $E^\bullet$.
This gives a real structure on $\textnormal{Ext}^2(E^\bullet,E^\bullet)$, and one can take its real subspace $\textnormal{Ext}^2(E^\bullet,E^\bullet)_{\mathbb{R}}$. Borisov--Joyce \autocite[Definition 3.25]{BJ} define orientations
as a square root of the natural isomorphism 
$$
\textnormal{det}(\mathbb{L}_{\boldsymbol{X}})^2\cong \mathcal{O}_{\boldsymbol{X}}
$$
for any $-2$-shifted symplectic derived scheme $\boldsymbol{X}$. On moduli stacks of sheaves, this is equivalent to the orientation of $\textnormal{Ext}^2(E^\bullet,E^\bullet)_{\mathbb{R}}$ or in the language of \autocite{OT} to a choice of a class of positive isotropic subspaces.  \\

For the Borisov--Joyce class or Oh--Thomas class to be defined in singular homology, respectively Chow homology with $\Z[2^{-1}]$ coefficients, one needs to show that these choices can be made continuously. This was proved for compact Calabi--Yau 4-folds in Cao--Gross--Joyce \autocite{JU,CGJ}\footnote{Guage-theoretic orientations are constructed in \cite{JU} after an error was found in \cite{CGJ}. However, the argument comparing them to orientations on perfect complexes remains the same as in \cite{CGJ}.}. 
However, many conjectural formulas have been written down in \autocite{CK1, CK2, CK3, CKM, CKM2, CMT, CT1} and \autocite{CT2}  for DT4 invariants of non-compact Calabi--Yau 4-folds, where additionally the moduli spaces can be non-compact and one has to use localization (see Oh--Thomas \autocite[§6]{OT} and Cao--Leung \autocite[§8]{CL}). The construction of virtual structure sheaves of Oh--Thomas \autocite{OT} for non-compact Calabi--Yau 4-folds and its localization also depends on existence of global orientation. Moreover, in string theory the DT4-invariants over Hilbert (respectively~Quot-) schemes of points were studied by Nekrasov \autocite{Nekrasov1} and Nekrasov--Piazzalunga \autocite{Nekrasov2}. The dependence on signs (orientations) is explained in \autocite[§2.4.2]{Nekrasov2}. 

The main goal of this work is to justify these results geometrically by proving orientability of the moduli stack $\mathcal{M}_X$ of compactly supported perfect complexes for any quasi-projective Calabi--Yau 4-fold\textit{ satisfying Assumption \ref{ass:astcs} below -- see the discussion below}. We also prove orientability of moduli spaces of stable pairs (as in \autocite{CMT2, CKM2})  and of Hilbert schemes of proper subschemes. Our result is also used by Oh--Thomas \cite{OT} to extend their theory to quasi-projective Calabi--Yau 4-folds and allows for the construction of global virtual structure sheaves even when $X$ is not compact. \textit{The original version of this work came out before the correction of \cite{CGJ} that appeared in \cite{JU}, hence the old statement was too strong. While I was writing an update which imposed the condition \eqref{eq:H^3-vanishing-intro}, I learnt from Ivan Karpov that the stronger formulation from Assumption \ref{ass:astcs}, which was announced to appear in \cite{KT26}, was possible. The current update shows that the methods in the proof of Theorem \ref{mainnctheorem} can be naturally extended to address this level of generality.}

Of particular importance are toric Calabi--Yau 4-folds satisfying Assumption \ref{ass:astcs} and, due to work of Nekrasov and his collaborators, especially $\C^4$ which was considered from a mathematical point of view in Cao--Kool \cite{CK1}. Recently, Kool--Rennemo \cite{KRdraft} described the orientations on $\operatorname{Hilb}^n(\C^4)$ in terms of a CY4 quiver. The existence of orientations for a general CY4 quiver were proved in \cite[Corollary 3.27]{Bo25}.

  A Calabi--Yau $n$-fold $(X,\Omega)$ here is a smooth quasi-projective variety $X$ of dimension $n$ over $\C$, such that $\Omega$ is a nowhere vanishing (algebraic) section of its canonical bundle $K_X$. We will additionally always impose the following assumption due to \cite[($\ast$)]{JU}.
  \begin{assumption}
  \label{ass:astcs}
For any $v\in H^3_{\cs}(X,\Z)$, denote its projection to $H^3_{\cs}(X,\Z_2)$ by $\ov{v}$. Then we require that
$$
\int_X \ov{v}\cup \Sq^2_{\cs}(\ov{v}) = 0
$$
where $\Sq^2_{\cs}:H^3_{\cs}(X,\Z_2)\to H^5_{\cs}(X,\Z_2)$ is the compactly supported Steenrod-square operation.
  \end{assumption}
Ivan Karpov informed me that \cite[§3.16]{KT26} shows that this condition is not an empty one. An important feature of our results is that we are able to express our orientations on $\mathcal{M}_X$ in terms of differential geometric ones.  This implies that we only need to make a $\Z_2$-choice for each class in $K^0_{\cs}(X)$ -- the compactly supported K-theory -- after fixing a normal-crossing divisor at infinity. We can also describe the compatibility of orientations under direct sums of classes in $K^0_{\cs}(X)$. This is important for constructing vertex algebras and Lie algebras on the homology of $\mathcal{M}_X$  following Joyce \cite{Jvertex}. It is also used for expressing wall-crossing formulae as in \cite{Bo25} (conjectured in \cite{GJT}).  \\



\subsection{Content of sections and main theorems} 
In §\ref{bigsection duality and spin}  we introduce briefly the language of derived algebraic geometry and recall the results of Pantev--Toën--Vaquié--Vezzosi \autocite{PTVV} and Brav--Dyckerhoff \autocite{BD} about existence of shifted symplectic structures. For any smooth (quasi-)projective variety $X$, we introduce \textit{twisted virtual canonical bundles} $\Lambda_{L}$ for any coherent sheaf $L$ on the moduli stack of (compactly supported) perfect complexes $\mathcal{M}_X$. If $(X,\Theta)$ is a spin manifold as in Definition \ref{definition spin variety}, then the virtual canonical bundle  is defined by $K_{\mathcal{M}_X} = \Lambda_{\Theta}$, where $\Theta$ is the corresponding theta characteristic. In Definition \ref{definition spin real structure}, we define the orientation bundle $O^{\mathcal{S}}\to\mathcal{M}_X$ over the moduli stack of perfect complexes on a projective spin 4-fold $X$.
\begin{theorem}[see Theorem \ref{spinprop}]
\label{spinpropintro}
Let $(X,\Theta)$ be smooth, projective and spin of complex dimension $4$. Consider
\begin{equation}
\label{OSintro}
O^\mathcal{S}\to \mathcal{M}_X
\end{equation}
 the orientation bundle from Definition \ref{definition spin real structure} and $\Gamma_X : (M_X)^{\textnormal{top}}\to \mathcal{C}_X$ be as in Definition \ref{definition before theorems}. Apply the topological realization functor $(-)^{\textnormal{top}}$ from Blanc \autocite[Definition 3.1]{Blanc}  to \ref{OSintro} to obtain a $\Z_2$-bundle 
 $
 (O^\mathcal{S})^{\textnormal{top}}\to (\mathcal{M}_X)^{\textnormal{top}}\,.
 $
 There is an isomorphism of $\Z_2$-bundles
 \begin{equation*}
     (O^{\mathcal{S}})^{\textnormal{top}}\cong \Gamma_{X}^*(O^{\slashed{D}_+}_{\mathcal{C}})\,,
 \end{equation*}
 where $O^{\slashed{D}_+}_{\mathcal{C}}\to \mathcal{C}_X$ is the $\Z_2$-bundle from Joyce--Tanaka--Upmeier \autocite[Definition 2.22]{JTU} applied to the positive Dirac operator $\slashed{D}_+: S_+\to S_-$ as in Cao--Gross--Joyce \autocite[Theorem 1.11]{CGJ}. When Assumption \ref{ass:astcs} is satisfied by $X$, the isomorphism becomes canonical and the bundle $O^{\mathcal{S}}\to \mathcal{M}_X$ is trivializable. 
\end{theorem}

By Corollary \ref{sccor}, this can be used to construct orientations on non-compact Calabi--Yau 4-folds which admit a spin-compactification satisfying Assumption \ref{ass:astcs}. Instead of trying to address the existence of such spin compactifications, we define in Definition \ref{definitionbowtie} for any non-compact Calabi--Yau 4-fold $X$ and its smooth projective compactification $Y$ the moduli stack $\mathcal{M}_{Y}\times_{\mathcal{M}^{\DNC}}\mathcal{M}_{Y}$ of perfect complexes identified at the divisor $D = Y\backslash X$ which we require to be strictly normal crossing (see Definition \ref{defSNC}). We define an orientation bundle $O^{\bowtie}\to \mathcal{M}_{Y}\times_{\mathcal{M}^{\DNC}}\mathcal{M}_{Y}$ depending on extension data $\bowtie$, which contains the information about a choice of order $\mathfrak{ord}$ of the smooth irreducible divisors in $D = \bigcup_{i=1}^N D_i$. Under the inclusion into the first component
$$
\zeta: \mathcal{M}_X\to \mathcal{M}_{Y}\times_{\mathcal{M}^{\DNC}}\mathcal{M}_{Y}
$$
the bundle $O^{\bowtie}$ pulls back to  $O^{\omega}\to \mathcal{M}_X$\,. This construction mimics the excision techniques from gauge theory as in Donaldson--Kronheimer \autocite{DK}, Donaldson \cite{Donaldson} and Upmeier \autocite{Markus} and is more natural. The analog in topology would be considering vector bundles $E,F\to X^+$ identified on $+$, where $(X^+,+)$ is the one point compactification of $X$ by the point $+$. This is known to generate $K^0_{\cs}(X)$ which is the natural topological analog of compactly supported perfect complexes on $X$. We collect here the orientability results stated in §\ref{bigsection duality and spin} (see Theorem \ref{maintheorem} and Theorem \ref{mainnctheorem}):
\begin{theorem}[Theorem \ref{maintheorem}]
\label{maintheoremintro}
Let $(X,\Omega)$ be a smooth Calabi--Yau 4-fold and $Y$ its smooth projective compactification by a strictly normal crossing divisor $D$. Choose $\mathfrak{ord}$ and the extension data $\bowtie$ as in Definition \ref{extension data}. If Assumption \ref{ass:astcs} is satisfied by $\tilde{Y}$ constructed in Definition \ref{definition double} then the $\mathbb{Z}_2$-bundle 
\begin{equation*}
    O^{\bowtie}\to\mathcal{M}_{Y}\times_{\mathcal{M}^{\DNC}}\mathcal{M}_{Y}
\end{equation*}
is trivializable.
 Let $\mathcal{C}_{D}, \mathcal{C}_{Y}$ be the topological spaces of maps from Definition \ref{definition before theorems}.  Let $D^{\mathcal{C}}_O\to \mathcal{C}_{Y}\times_{\mathcal{C}_{D}}\mathcal{C}_{Y}$ be the $\Z_2$-bundle from \eqref{DO}. If 
 $$\Gamma: (\mathcal{M}_{Y,D})^{\textnormal{top}}\to \mathcal{C}_{Y}\times_{\mathcal{C}_{D}}\mathcal{C}_{Y}$$
 is the natural map from Definition \ref{definition Gamma}, then there exists a canonical isomorphism 
\begin{equation*}
    \mathfrak{I}^{\bowtie}: \Gamma^*(D^{\mathcal{C}}_O)\cong (O^{\bowtie})^{\textnormal{top}}\,.
\end{equation*}
\end{theorem}
When making the statement, we assumed Assumption \ref{ass:astcs} holds for $\tilde{Y}$. By Lemma \ref{lem:assumptionfortildeY} this holds whenever
\begin{equation}
\label{eq:H^3-vanishing-intro}
H^3(X,\Z_2)\cong 0 \cong H^3_{\cs}(X,\Z_2)\,.
\end{equation}
After a more careful treatment of orientability in §\ref{sec:restriction}, we prove:
\begin{theorem}[Theorem \ref{mainnctheorem}]
\label{mainnctheoremintro}
Let $(X,\Omega)$ be  a Calabi--Yau 4-fold satisfying Assumption \ref{ass:astcs}. Then the $\mathbb{Z}_2$-bundle 
\begin{equation}
\label{Oomegaintro}
    O^\omega\to \mathcal{M}_X
\end{equation}
from Definition \ref{orientation-def} is trivializable. 

Moreover, fix a smooth projective compactification $Y$, such that $D=Y\backslash X$ is a strictly normal crossing divisor. Let $(O^{\omega})^{\textnormal{top}}\to (\mathcal{M}_X)^{\textnormal{top}}$ be the $\Z_2$-bundle obtained by applying the topological realization functor $(-)^{\textnormal{top}}$ to \eqref{Oomegaintro}, then there is an isomorphism
$$
\mathfrak{I}_{Y}:(\Gamma^{\cs}_X)^*(O^{\cs})\cong (O^{\omega})^{\textnormal{top}}\,,
$$
where $O^{\cs}\to \mathcal{C}^{\cs}_X$ is the $\Z_2$-bundle from Definition \ref{definition beforencmain}, $$\mathcal{C}^{\cs}_X = \textnormal{Map}_{C^0}\big((X^+,+),(BU\times\mathbb{Z},0)\big)$$ is the topological space classifying $K^0_{\cs}(X)$ and $\Gamma^{\cs}: (\mathcal{M}_X)^{\textnormal{top}}\to C^{\cs}_X$ is the natural map from Definition \ref{definition beforencmain}. Whenever $\tilde{Y}$ additionally satisfies Assumption \ref{ass:astcs}, this isomorhism is canonical. 
\end{theorem}
We argue in Remark \ref{remark joyce suggestion}, that with some extra care, the isomorphism $\mathfrak{I}_{Y}$ could be shown to be independent of $Y$, but leave proof of this for once applications arise. 

If $X$ is quasi-projective and $\mathcal{M}$ is a moduli stack of pairs of the form $\mathcal{O}_X\to F$, where $F$ is compactly supported, we need a modification, because the complex is not compactly supported. Let $\mathcal{E}\to \mathcal{M}$ be the universal family, then for fixed compactification $X\subset Y$ there is a natural isomorphism from Definition \ref{definition orientation stablepairs}:
$$
\textnormal{det}\Big(\big(\underline{\textnormal{Hom}}_{\mathcal{M}}(\mathcal{E},\mathcal{E})\big)_0\Big)\cong \textnormal{det}^*\Big(\big(\underline{\textnormal{Hom}}_{\mathcal{M}}(\mathcal{E},\mathcal{E})\big)_0\Big)\,,
$$
where $(-)_0$ denotes the trace-less part and we use the notation $\underline{\textnormal{Hom}}_Z(E,F)$ for two perfect complexes $E,F$ on $X\times Z$ to denote $\pi_{2\,*}(E^\vee\otimes F)$, where $\pi_2: X\times Z\to Z$ is the projection. To this duality, we can again associate a $\Z_2$-bundle $O^0\to \mathcal{M}$. We can now state the result.
\begin{theorem}[Thm. \ref{theorem stable pairs}]
\label{thm:stable-pair-intro}
If $X$ satisfies Assumption \ref{ass:astcs}, then $O^0$ is trivializable. Moreover, let $i:X\to Y$ be a compactification with $Y\backslash X=D$ strictly normal crossing. Let $\eta: \mathcal{M} \to \mathcal{M}_{Y}\times_{\mathcal{M}_{D}}\mathcal{M}_{Y}$ be given by $[E]\to ,[\bar{E},\mathcal{O}_{Y}]$, where $\bar{E}$ is the extension by a structure sheaf to the divisor, then there is a canonical isomorphism
$$
\eta^*(O^{\bowtie}) \cong O^0\,.
$$
\end{theorem}
Consequentially, if Assumption \ref{ass:astcs} is satisfied by $\tilde{Y}$ (e.g., when \eqref{eq:H^3-vanishing-intro} holds) and if all stable pairs parameterized by $\mathcal{M}$ are of a fixed class $\llbracket \mathcal{O}_X\rrbracket +\alpha$, where $\alpha\in K^0_{\textnormal{cs}}(X)$, this determines unique orientations up to a global sign for a fixed compactification $Y$ of $X$.
 \begin{remark}
   In \cite{KT26}, the authors discuss two new approaches to proving Theorem \ref{mainnctheoremintro} and Theorem \ref{thm:stable-pair-intro}. Note that the comparison of the orientations above, a priori, depends on the choice of a compactification. The proofs of \cite[Theorem 3.1, Theorem A.1]{KT26} state that the constructions of \cite{KT26} are independent of choices of compactifications.
 \end{remark}


In §\ref{bigsection technical proof}, we recall some background from Cao--Gross--Joyce \autocite{CGJ}. We then develop some technical tools for transport along complex determinant line bundles of complex pseudo-differential operators generalizing the work of work of Upmeier \cite{Markus}, Donaldson \cite{Donaldson}, \cite{DK} and Atiyah--Singer \cite{AS1}. We require such generalizations, because in allowing $Y$ to be any compactification of $X$, we lose the ability to work with real structures and real line bundles, as there is no analog for the real Dirac operator $\slashed{D}_+$ on $Y$. We can still do excisions natural up to contractible choices of isotopies and use these to restrict everything back into $X$ where the operator $\slashed{D}_+$ exists.
 
 In §\ref{section relative framing} we construct the doubled spin manifold $\tilde{Y}$ from $X$, and its moduli space (topological stack) $\mathcal{B}_{\tilde{Y},\tilde{T}}$ of connections on pairs of principal bundle $P,Q\to \tilde{Y}$ identified on $\tilde{T}$. Assuming that $\tilde{Y}$ satisfies Assumption \ref{ass:astcs},its has a product orientation as defined in Definition \ref{definition relative framing orienation}, which corresponds to the product of the orientation $\Z_2$-torsors at each pair of connections. The resulting orientation bundle is denoted $D_O(\tilde{Y})$ and its pullback to pairs of vector bundles on $Y$ identified on $\DNC$ is denoted by $D_O$.
 
 The bundles $D_O$ and $O^{\bowtie}$ restricted to holomorphic vector bundles generated by global sections are related in §\ref{section comon resolution} as weak H-principal $\Z_2$-bundles (see Definition \ref{definition weakstrong Hprincipal}) by constructing isotopies (natural up to contractible choices) between the real structures used to obtain them. The isomorphisms can be enriched to one of strong H-principal $\Z_2$-bundles when they are trivializable. This tells us that $O^{\bowtie}$ is trivializable whenever $\tilde{Y}$ satisfies Assumption \ref{ass:astcs} and allows us to express compatibility under direct sums in §\ref{bigsection signs}. 
 
 
 Finally, in §\ref{bigsection signs}, we discuss the relations of orientations under direct sums in K-theory. For given extension data $\bowtie$, we obtain the result in Proposition \ref{proposition sign}, which expresses choices of trivialization of $O^{\bowtie}\to \mathcal{M}_{Y}\times_{ \mathcal{M}^{\DNC}}\mathcal{M}_{Y}$ in terms of orientations on the K-theoretic space $\mathcal{C}_{Y}\times_{\mathcal{C}_{\DNC}}\mathcal{C}_{Y}$ with the property $\pi_0(\mathcal{C}_{Y}\times_{\mathcal{C}_{\DNC}}\mathcal{C}_{Y}) = K^0(Y\cup_{\DNC}Y)$. The signs comparing the orientations are more complicated, but restrict to the expected result under the inclusions $\zeta:\mathcal{M}_X\to \mathcal{M}_{Y}\times_{\mathcal{M}^{\DNC}}\mathcal{M}_{Y}$ and $\kappa^{\cs}:\mathcal{C}^{\cs}_X\to \mathcal{C}_{Y}\times_{\mathcal{C}_{\DNC}}\mathcal{C}_{Y}$. The statement in Theorem \ref{signscomparison} is slightly more general.
\begin{theorem}[Theorem \ref{signscomparison}]
\label{theorem rules under addition intro}
Assume that $X$ satisfies \eqref{eq:H^3-vanishing-intro}. Let $\mathcal{C}^{\textnormal{cs}}_{\alpha}$ denote the connected component of $\mathcal{C}_X^{\textnormal{cs}}$ corresponding to $\alpha\in K^0_{\textnormal{cs}}(X) =\pi_0(\mathcal{C}^{\textnormal{cs}}_X)$\, and $O^{\cs}_\alpha = O^{\cs}|_{\mathcal{C}^{\cs}_\alpha}$. 
There is a canonical isomorphism
$\phi^\omega:O^\omega\boxtimes O^\omega\to \mu^*(O^\omega)$ such that for fixed choices of trivialization $o^{\textnormal{cs}}_\alpha$ of $O^{\textnormal{cs}}_\alpha$, we have
$$
\phi^\omega\Big(\mathfrak{I}\big((\Gamma^{\textnormal{cs}}_X)^*o^{\textnormal{cs}}_\alpha\big)\boxtimes \mathfrak{I}\big((\Gamma^{\textnormal{cs}}_X)^*o^{\textnormal{cs}}_{\beta}\big)\Big) \cong \epsilon_{\alpha,\beta}\mathfrak{I}\big((\Gamma^{\textnormal{cs}}_X)^*o^{\textnormal{cs}}_{\alpha+\beta}\big)\,.
$$
where the $\epsilon_{\alpha,\beta}\in \{\pm1 \}$ satisfy
$\epsilon_{\beta,\alpha} =(-1)^{\bar{\chi}(\alpha,\alpha)\bar{\chi}(\beta,\beta) + \bar{\chi}(\alpha,\beta)} \epsilon_{\alpha,\beta}$ and $\epsilon_{\alpha,\beta}\epsilon_{\alpha+\beta,\gamma}=\epsilon_{\beta,\gamma}\epsilon_{\alpha,\beta+\gamma}$ for all $\alpha,\beta,\gamma\in K^0_{\textnormal{cs}}(X)$.
\end{theorem}
 
 Using this, we formulate a version of \autocite[Theorem 2.27] {JTU} for the compactly supported orientation group $\Omega^{\cs}_X$ adapted to the compactly supported case. 
 \section*{Acknowledgments} 
 The author wishes to express his gratitude to Dominic Joyce for his supervision and many helpful suggestions and ideas. While working on \cite{KT26}, Ivan Karpov was very kind to share with me the formulation using Assumption \ref{ass:astcs} which I used to replace the weaker version relying on \eqref{eq:H^3-vanishing-intro}.
 
 We are indebted to Markus Upmeier for patient explanation of his work. We also thank Chris Brav, Yalong Cao, Simon Donaldson, Mathieu Florence,  Michel van Garrel, Jacob Gross, Maxim Kontsevich, Martijn Kool, Edwin Kutas,  Sven Meinhardt, Hector Papoulias, Yuuji Tanaka, Richard Thomas, Bertrand Toën and Andrzej Weber. Finally, we would like to thank the referee for invaluable remarks.
 
 The author was supported by a Clarendon Fund Scholarship at the University of Oxford. 
 \section*{Notation}
 If $E\to X$ is a complex vector bundle, we denote $\textnormal{det}(E) = \Lambda^{\textnormal{dim}_{\C}}E$ its complex determinant line bundle.  We write $\textnormal{det}^*(E)$ to denote its dual. If $E$ is real we write $\textnormal{det}_{\R}(E) = \Lambda^{\textnormal{dim}_{\R}}E$. 
\\
 If $L_1,L_2\to Y$ are two line bundles we will sometimes use the notation $L_1L_2 = L_1\otimes L_2$.\\
 If $P:H_0\to H_1$ is an operator between Hilbert spaces, we write $P^*$ to denote its adjoint. \\
 If $D\subset X$ is a divisor, we write $E(D) = E\otimes_{\mathcal{O}_X}\mathcal{O}_D$.\\
 If $s:V\to W$ is a map of vector bundles $V,W\to X$ and $Q\to X$ another vector bundle, then we will write
 $$
 s=s\otimes\textnormal{id}_Q: V\otimes Q\to W\otimes Q\,.
 $$
We often do not distinguish between a scheme $X$ over $\C$ and its analytification $X^{\textnormal{an}}$ .
 
\section{Orientation on non-compact Calabi--Yau 4-folds}
\label{bigsection duality and spin}
In this section, we review the definition of orientation on the moduli stack of perfect complexes of a Calabi--Yau 4-fold which uses the existence of $-2$-shifted symplectic structure as defined by Pantev--Toën--Vaquié--Vezzosi in \autocite{PTVV}. We introduce some new notions of orientation in more general cases which we use together with compactifications of non-compact Calabi--Yau 4-folds to prove that the moduli stack of compactly supported perfect complexes is orientable. 

\subsection{Moduli stacks of perfect complexes and shifted symplectic structures}
\label{section 3.1}

Here we recall the construction of moduli stacks of perfect complexes that we will be working with and the shifted symplectic structures on them. \\

Let $X$ be a smooth algebraic variety. Its category $D^b\big(\textnormal{Coh}(X)\big)$ of complexes of sheaves with coherent cohomologies does not have a moduli stack of its objects in the setting of standard algebraic stacks. Instead, one needs to rely on the methods provided by the theory of higher stacks and derived stacks. For a thorough discussion of these two terms, one can look at Toën and Vezzosi \autocite{TVHAG1, TVHAG2, TVHAGDAG} in the setting of model categories (see Hovey \autocite{Hovey} and Hirschhorn \autocite{Hirschhorn} ) and Lurie \autocite{DAG} in the setting of $\infty$-categories (see Lurie \autocite{LurieHT} or \autocite{LurieSHT}). Both higher stacks and derived stacks over $\mathbb{C}$ form the $\infty$-categories \textbf{HSta}$_\mathbb{C}$ and $\textbf{DSta}_\mathbb{C}$. There exist an inclusion $\infty$-functor $i:\textbf{HSta}_\mathbb{C}\to\textbf{DSta}_\mathbb{C} $ and its adjoint truncation $\infty$-functor $t_0:\textbf{DSta}_\mathbb{C}\to\textbf{HSta}_\mathbb{C} $, which relate the two categories. Note that the $\mathbb{C}$-points of the stacks are left invariant under these functors.

Let $\mathcal{T}$ be a dg-category (for background on dg-categories see Keller \autocite{Keller} and Toën \autocite{To1}). Toën \autocite{To1} introduces homotopy theory of dg-categories which is then used in  Toën and Vaquié \autocite{TVaq} to define for $\mathcal{T}$ its associated moduli stack as a derived stack $\boldsymbol{\mathcal{M}}_\mathcal{T}$ which classifies the pseudo-perfect dg-modules of $\mathcal{T}^{\textnormal{op}}$. This defines a functor from the homotopy category of dg-categories to the homotopy category of derived stacks
$
\boldsymbol{\mathcal{M}}_{(-)}:\textnormal{Ho}(\textbf{dg - Cat})^{\textnormal{op}}\to \textnormal{Ho}(\textbf{DSta}).
$
One can compose this with the truncation functor 
$t_0: \textnormal{Ho}(\textbf{DSta})\to \textnormal{Ho}(\textbf{HSta})$
mapping to the homotopy category of higher (infinity) stacks. We denote the composition by $\mathcal{M}_{(-)} = t_0\circ\boldsymbol{\mathcal{M}}_{(-)}$\,.

\begin{definition}
\label{definition stack of perfect complexes}
Let $X$ be quasi-projective variety over $\mathbb{C}$, then we use the notation $\boldsymbol{\mathcal{M}}_X$, $\mathcal{M}_X$ for $\boldsymbol{\mathcal{M}}_{L_{\textnormal{pe}}(X)},$ respectively $\mathcal{M}_{L_{\textnormal{pe}}(X)}$, where $L_{\textnormal{pe}}(X)$ is the dg-category of perfect complexes on $X$. When $X$ is smooth and projective, it is shown by Toën--Vaquié in \autocite[Corollary 3.29]{TVaq} that $\boldsymbol{\mathcal{M}}_X$ is a locally geometric derived stack which is locally of finite type. A locally geometric derived stack corresponds to a union of open geometric sub-stacks, and thus it has a well defined cotangent complex $\mathbb{L}_{\boldsymbol{\mathcal{M}}_X}$ in the dg-category of perfect modules on $\boldsymbol{\mathcal{M}}_X$ denoted by $L_{\textnormal{pe}}(\boldsymbol{\mathcal{M}}_X)$.\\
\end{definition}

For $X$ smooth and projective, both $\boldsymbol{\mathcal{M}}_X$ and $\mathcal{M}_X$ can be expressed as mapping stacks in the homotopy categories of the  $\infty$-categories $\textbf{DSta}_\mathbb{C}$ and $\textbf{HSta}_\mathbb{C}$. Let $\textbf{Perf}_\mathbb{C}$ be the derived stack of perfect dg-modules/complexes over $\mathbb{C}$ as defined in \autocite[Definition 1.3.7.5]{TVHAG2} when applied to complexes of vector spaces over $\mathbb{C}$ as the chosen homotopy algebraic geometric context. Taking its truncation $\textnormal{Perf}_\mathbb{C} = t_0(\textbf{Perf}_\mathbb{C})$, we have
\begin{equation}
\label{smooth projective mapping}
  \boldsymbol{\mathcal{M}}_X=\textbf{Map}(X,\textbf{Perf}_\mathbb{C})\,\qquad \mathcal{M}_X=\textnormal{Map}(X,\textnormal{Perf}_\mathbb{C})\,.
\end{equation}

\begin{definition}If $X$ is just a quasi-projective variety that is not smooth or is not projective, then we will denote the mapping stacks using a superscript:

\begin{equation}
\label{mapping stack}
    \boldsymbol{\mathcal{M}}^X=\textbf{Map}(X,\textbf{Perf}_\mathbb{C})\,\qquad \mathcal{M}^X=\textnormal{Map}(X,\textnormal{Perf}_\mathbb{C})\,.
\end{equation}
\end{definition}

\begin{remark}
	One can show that if $X$ is not smooth, then $\mathcal{M}_X$ will not be equal to $\mathcal{M}^X$ in general, as the former can classify objects of $\textbf{QCoh}(X)$ which are not perfect. When $X$ is a quasi-projective smooth variety over $\mathbb{C}$, then $L_{\textnormal{pe}}(X)$ is no longer a finite type dg-category and its moduli stack $\boldsymbol{\mathcal{M}}_X$  behaves differently. For any affine scheme $\textnormal{Spec}(A)\in \textbf{Aff}_{\mathbb{C}}$, the $\textnormal{Spec}(A)$-points of $\mathcal{M}_X$ are families of perfect complexes on $X\times \textnormal{Spec}(A)$ with proper support in $X$. 
\end{remark}

If $X$ is proper, we can use the description of $\boldsymbol{\mathcal{M}}_X$ as a mapping stack to construct a universal complex on $X\times\boldsymbol{\mathcal{M}}_X$:
If
\begin{equation}
\label{mapping map}
    \boldsymbol{u}: X\times \boldsymbol{\mathcal{M}}_X\to \textbf{Perf}_\mathbb{C}\,
\end{equation}
is the canonical morphisms for $\boldsymbol{\mathcal{M}}_X$ as a mapping stack, and  $\mathcal{U}_0$ is the universal complex on $\textbf{Perf}_\mathbb{C}$ used by Pantev--Toën--Vaquié--Vezzosi in \autocite{PTVV}, then one defines the universal complex 
$
\mathcal{U}_X = \boldsymbol{u}^*(\mathcal{U}_0)
$
on  $X\times\boldsymbol{\mathcal{M}}_X$. 
When $X$ is quasi-projective and not necessarily proper, we need a different construction of the universal complex.
\begin{definition}
\label{def unicom}
Let $i_X: X \hookrightarrow Y$ be a smooth compactification of $X$ and $\mathcal{U}_Y\in L_{\textnormal{pe}}\big(Y\times \boldsymbol{\mathcal{M}}_Y\big)$ the universal complex. Let \begin{equation}
\label{xiY}
\xi_Y=\boldsymbol{\mathcal{M}}_{(i_X^*)}: \boldsymbol{\mathcal{M}}_X\to \boldsymbol{\mathcal{M}}_Y\end{equation}
be the image of the pullback $i_X^*: L_{\textnormal{pe}}(Y)\to L_{\textnormal{pe}}(Y)$, then it acts on $\textnormal{Spec}(A)$-points by the right adjoint of $\big(i_{X}\times \textnormal{id}_{\textnormal{Spec}(A)}\big)^*$ as follows from its construction in Toën--Vaquié \cite[§3.1]{TVaq} and therefore by the pushforward $ (i_X\times\textnormal{id}_{\textnormal{Spec}(A)})_*$ of families of compactly supported perfect complexes on $X$. We define $\mathcal{U}_X\to X\times \boldsymbol{\mathcal{M}}_X$ by
$$
\mathcal{U}_X = \xi^*_Y(\mathcal{U}_Y)\,.
$$
It is independent of the choice of a compactification\footnote{Simply choose a common compactification $Y\leftarrow Y''\to Y'$ and compare the resulting universal complexes.}.
\end{definition}

 When $Z=\prod_{i\in I} Z_i$, we will use $\pi_{I'}:Z\to \prod_{i\in I'}Z_i$ for $I'\subset I$ to denote the projection to $I'$ components of the product. We use this also for general fiber products.   Let $Y$ now be any projective smooth four-fold and  $L$ a coherent sheaf on $Y$ and $\mathcal{U}_{Y}\in \textnormal{L}_{\textnormal{pe}}(Y\times \mathcal{M}_{Y})$ its universal complex. We define
\begin{equation}
\label{extl}
    \mathcal{E}\textnormal{xt}_{L} = \pi_{2,3\,*} (\pi_{1,2}^*\,\mathcal{U}^\vee_{Y}\otimes \pi_{1,3}^*\,\mathcal{U}_{Y}\otimes\pi_1^*L)\,,\quad 
 \mathbbm{P}_{ L}=\Delta_{\mathcal{M}_Y}^*\mathcal{E}\textnormal{xt}_{ L}\,.
\end{equation}
As pushforward along $\pi_{2,3}:Y\times \boldsymbol{\mathcal{M}}_{Y}\times \boldsymbol{\mathcal{M}}_{Y}\to \boldsymbol{\mathcal{M}}_{Y}\times \boldsymbol{\mathcal{M}}_{Y}$ maps (compactly supported) perfect complexes in $Y$ to perfect complexes, it has a right adjoint $\pi_{2,3}^!$ by Lurie's adjoint functor theorem Gaitsgory--Rozenblyum \cite[Thm. 2.5.4, §1.1.2]{Gait}, Lurie \cite[Corollary 5.5.2.9]{LurieHT}. Moreover, $\pi^!_{2,3} = \pi_{2,3}^*(-)\otimes K_Y[4]$, which gives us the usual Serre duality in families
\begin{equation}
\label{serreduality}
    \mathcal{E}\textnormal{xt}_{ L}\cong {\sigma}^*(\mathcal{E}\textnormal{xt}_{ (K_Y\otimes L^\vee)}^\vee)[-4]\,,
\end{equation}
where ${\sigma}: \boldsymbol{\mathcal{M}}_Y\times\boldsymbol{\mathcal{M}}_Y\to\boldsymbol{\mathcal{M}}_Y\times\boldsymbol{\mathcal{M}}_Y $ is the map interchanging the factors. 
\begin{definition}
\label{Plambda}
Let $Y$ be smooth and and $L$ a coherent sheaf on $Y$, then 
as \eqref{extl} are perfect, we construct the \textit{$L$-twisted virtual canonical bundle}
$$
\Sigma_L =\textnormal{det}(\mathcal{E}\textnormal{xt}_{L})\,,\quad \Lambda_{ L} = \textnormal{det}(\mathbb{P}_{ L})\,. 
$$
Moreover, $\Sigma_L, \Lambda_L$ are \textit{$\Z_2$-graded} with \textit{degree} given by a map $\textnormal{deg}(\Lambda_L):\boldsymbol{\mathcal{M}}_X\to \Z_2$, such that $$ \textnormal{deg}(\Sigma_L)|_{\boldsymbol{\mathcal{M}}_\alpha\times\boldsymbol{\mathcal{M}}_\beta} \equiv \chi(\alpha,\beta\cdot L)\quad (\textnormal{mod }\, 2)\,,\quad  \textnormal{deg}(\Lambda_L)|_{\boldsymbol{{\mathcal{M}}_\alpha}} \equiv \chi(\alpha,\alpha\cdot L)\quad (\textnormal{mod }\, 2)\,.$$ where $\alpha,\beta\in K^0(X)$ and $\boldsymbol{\mathcal{M}}_\alpha$ is the stack of complexes with class $\llbracket E\rrbracket = \alpha$. See Definition \ref{def grading real} and \eqref{compactly supported chi} for more details.
\end{definition}
 From the duality \eqref{serreduality}, we obtain the isomorphisms 
\begin{align}
    \label{i_L}
  \Sigma_L\cong\sigma^*(\Sigma_{K_X\otimes L^\vee})^*\,, \quad   \theta_{ L}:\mathbb{P}_{ L}\longrightarrow \mathbb{P}^\vee_{ (K_X\otimes L^\vee)}[-4]\,,\quad
    i_{ L}: \Lambda_{ L}\longrightarrow \Lambda_{ (K_X\otimes L^\vee)}^ *\,.
\end{align}
It follows from \autocite[Proposition 3.3]{BD} that the tangent complex of $\boldsymbol{\mathcal{M}}_X$ can be expressed as \begin{equation}
    \mathbbm{T}(\boldsymbol{\mathcal{M}}_X) = \mathbb{P}_{\mathcal{O}_X}[1]
\end{equation}
We use the term Calabi--Yau manifold in the following sense.
\begin{definition}
\label{definition CY}
A Calabi--Yau $n$-fold is a pair $(X,\Omega)$, where $X$ is a smooth quasi-projective variety of dimension $n$ over $\C$ and $\Omega$ is a global non-vanishing algebraic section of the canonical bundle $K_X$ of $X$.
 In this case, we use the notation $K_{\boldsymbol{\mathcal{M}}_X}= \Lambda_{ \mathcal{O}_X}$. Then applying the above to $L =\mathcal{O}_X$, we obtain the isomorphisms
\begin{align}
\label{theta}
    &\theta^\omega: \mathbb{T}(\boldsymbol{\mathcal{M}}_X)\to \mathbb{L}(\boldsymbol{\mathcal{M}}_X)[-2]\,,\\
    \label{i}
    &i^\omega: K_{\boldsymbol{\mathcal{M}}_X}\to K_{\boldsymbol{\mathcal{M}}_X}^*\,.
\end{align}
Brav and Dyckerhoff prove in \autocite[Proposition 5.3]{BD} and \autocite[Theorem 5.5 (1)]{BD} that the isomorphism $\theta^\omega$ comes from a $-2$-shifted symplectic form $\omega$ on $\boldsymbol{\mathcal{M}}_X$.
\end{definition}

 This extra condition to have a $-2$-shifted symplectic stack $(\boldsymbol{\mathcal{M}}_X,\omega)$ is necessary for constructing Borisov--Joyce fundamental classes as in \autocite{BJ} and Oh--Thomas classes from \autocite{OT}. \\

\subsection{Orientation bundles on moduli stacks of perfect complexes}
\label{section alggeom orientation}

In this subsection, we review the known results of orientation on compact Calabi--Yau 4-folds, and we define new orientation bundles which we will use in later sections to prove the orientability for the non-compact case. \\

There is a different but equivalent approach to constructing a $\mathbb{Z}_2$-bundle other than taking a real line bundle (or a complex line bundle with a real structure) and its associated orientation bundle.
\begin{definition}
\label{definition Z_2 bundles}
Let $L\to X$ be a complex line bundle over some variety, stack or a topological space. Additionally, let us assume, that there is an isomorphism $\mathcal{I}: L\to L^*$. Then  consider the adjoint isomorphism $\mathcal{J}$ which is the composition of
$$
L\otimes L\xrightarrow{\textnormal{id}_L\otimes \mathcal{I}}L\otimes L^*\to \underline{\mathbb{C}}\,,
$$
where the second isomorphism is the canonical one. Then one can define the \textit{square root $\mathbb{Z}_2$-bundle associated with $\mathcal{I}$}  denoted by $O^{\mathcal{I}}$. This bundle is given by the sheaf of its sections (which we denote the same way):
$$
O^{\mathcal{I}}(U) = \{o:L|_U\xrightarrow{\sim}\underline{\mathbb{C}}_U : o\otimes o = \mathcal{J}\} \,.
$$
\end{definition}

As this will be important for the proof of Theorem \ref{maintheorem}, we mention here how this construction is related to the one using real structures.

\begin{remark}
\label{realstruc}
Let $\mathcal{I}: L \to L^*$ be an isomorphism and $\langle-,-\rangle$  a metric on $L$. There is the induced isomorphism $\mu': L\to \bar{L}$ such that $\langle-,\mu'(l) \rangle = \mathcal{I}(l)$ for all $l\in L$. As an anti-linear endomorphism, we take its second power $(\mu')^2$ which is given by a multiplication by a strictly positive real function $s$ on $X$. Then $\mu = \frac{\mu'}{\sqrt{s}}$ is a well defined real structure on $L$. Let $L_{\mu}$ denote the real line bundle of fixed points and $\textnormal{or}(L_{\mu})$ its associated orientation bundle, then it is easy to see that $\textnormal{or}(L_{\mu})$  is canonically isomorphic to $O^{\mathcal{I}}$. For a nice exposition of orientations on $O(n,\mathbb{C})$ bundles see Oh-Thomas \autocite{OT}.
 \end{remark}

We recall the definition  of orientations on $-2$-shifted derived stacks (see Borisov--Joyce \autocite[Definition 2.12]{BJ}), which can now be applied also to non-compact Calabi--Yau 4-folds.

\begin{definition}
\label{orientation-def}
Let $(\boldsymbol{S},\omega)$ be a $-2$-shifted symplectic derived stack. Let 
$
\theta^\omega: \mathbb{T}(\boldsymbol{S})\to \mathbb{L}(\boldsymbol{S})[-2]
$
be the isomorphism associated to the $-2$ shifted symplectic form. Taking the determinants and inverting, one obtains the isomorphism
$
i^\omega: \textnormal{det}(\mathbb{L}_{\boldsymbol{S}})\to \textnormal{det}(\mathbb{L}_{\boldsymbol{S}})^*.
$
The \textit{orientation bundle}  $O^\omega\to \boldsymbol{S}$ is the square root $\mathbb{Z}_2$ bundle associated to $i^\omega$. \\
Let $(X,\omega)$ be a Calabi--Yau 4-fold as in Definition \ref{definition CY}, then we have $K_{\boldsymbol{\mathcal{M}}_X} = \textnormal{det}(\mathbb{L}_{\boldsymbol{\mathcal{M}}_X})$. The isomorphisms $\theta^\omega$ and $i^\omega$ from \eqref{theta} and \eqref{i} are associated to the $-2$-shifted symplectic derived stack from Definitions \ref{definition stack of perfect complexes} and \ref{definition CY}.
We denote in this case the orientation bundle by $O^{\omega}\to \boldsymbol{\mathcal{M}_X}$.
\end{definition}

From now now on we will restrict ourselves to working only in higher stacks $\textbf{HSta}_\mathbb{C}$. By this we mean that we take truncations $S = t_0(\boldsymbol{S})$ everywhere and restrict bundles and their isomorphism constructed on $\boldsymbol{S}$ to $S$ by the canonical inclusion $S\hookrightarrow \boldsymbol{S}$.

When $X$ is a compact Calabi--Yau 4-fold, Cao--Gross--Joyce \autocite[Theorem 1.15]{CGJ} combined with \cite[Theorem 13.5]{JU} prove that $O^\omega\to \mathcal{M}_X$ is trivializable. One could generalize their result by replacing the requirement of $X$ being Calabi--Yau by a weaker one. 
\begin{definition}
\label{definition spin variety}
Let $X$ be a smooth projective variety and $K_X$ its canonical divisor class. A divisor class $\Theta$, such that $2\Theta= K_{X}$ is called a \textit{theta characteristic}. We say that $(X,\Theta)$ for a given choice of a theta characteristic $\Theta$ is \textit{spin}. 
\end{definition}

Using this, we construct orientation for the case when $X$ is smooth projective and spin. 

\begin{definition}
\label{definition spin real structure}
When $X$ is spin and a choice of $\Theta$ is made, we will also use the notation $K_{\mathcal{M}_X} = \Lambda_{\Theta}$.  By \eqref{i_L} we obtain a natural isomorphism $i_{\Theta}:K_{\mathcal{M}_X}\xrightarrow{\sim}K^*_{\mathcal{M}_X}$ which induces the orientation bundle $O^{\mathcal{S}}\to \mathcal{M}_X$ via Definition \ref{definition Z_2 bundles}.
\end{definition}

\begin{remark}
 A choice of $\Theta$ is equivalent to a choice of spin structure on $X^{\textnormal{an}}$ (see Atiyah \autocite[Proposition 3.2]{atiyahriemann}). 
\end{remark}

Before we state the generalization of Cao--Gross--Joyce \autocite[Theorem 1.15]{CGJ} to projective spin varieties, we recall some terminology used to formulate it. 

Blanc \autocite{Blanc} and Simpson \autocite{Simpson}, define the topological realization $\infty$-functor:
$
(-)^{\textnormal{top}}: \textbf{HSta}_{\mathbb{C}}\to \textbf{Top}\,
$
as the simplically enriched left Kan extension of the functor $
(-)^{\textnormal{an}}: \textbf{Aff}_{\mathbb{C}}\to \textbf{Top}\,,$
which maps every finite type affine scheme over $\mathbb{C}$ to its analytification.\\

\begin{definition}
\label{definition before theorems}
Let $Z$ be a projective variety over $\C$. Let $\mathcal{M}^Z$, be the mapping stack from \eqref{mapping stack}. Let 
$
u_Z: Z\times\mathcal{M}^Z\to \textnormal{Perf}_{\C}
$
be the canonical map. Applying $(-)^{\textnormal{top}}$ and using Blanc \autocite[§4.2]{Blanc}, we obtain
$
(u_Z)^{\textnormal{top}}: Z^{\textnormal{an}}\times (\mathcal{M}^Z)^{\textnormal{top}}\to BU\times \mathbb{Z}\,,
$
where $BU = \varinjlim_{n\to \infty}BU(n)$.
This gives us 
$$
\Gamma_Z: (M^Z)^{\textnormal{top}}\to \textnormal{Map}_{C^0}(Z^{\textnormal{an}}, BU\times \Z)\,,
$$
by the universal property of $\textnormal{Map}_{C^0}(-,-)$, where $\textnormal{Map}_{C^0}(-,-)$ denotes the mapping space bifunctor in \textbf{Top}. For any topological space $T$ we use the notation 
$
\mathcal{C}_T = \textnormal{Map}_{C^0}(T, BU\times \Z)\,.
$
\end{definition}

\begin{theorem}
\label{spinprop}
Let $(X,\Theta)$ be a spin manifold satisfying Assumption \ref{ass:astcs} and
$
\label{spinorientationbundle}
O^\mathcal{S}\to \mathcal{M}_X
$ its orientation bundle from Definition \ref{definition spin real structure}.
Let $\Gamma_X : (\mathcal{M}_X)^{\textnormal{top}}\to \mathcal{C}_X$ be as in Definition \ref{definition before theorems} and apply $(-)^{\textnormal{top}}$ to obtain a $\Z_2$-bundle 
 $
 (O^\mathcal{S})^{\textnormal{top}}\to (\mathcal{M}_X)^{\textnormal{top}}.
 $
 There is a canonical isomorphism of $\Z_2$-bundles
 \begin{equation*}
     (O^{\mathcal{S}})^{\textnormal{top}}\cong \Gamma_{X}^*(O^{\slashed{D}_+}_{\mathcal{C}})\,,
 \end{equation*}
 where $O^{\slashed{D}_+}_{\mathcal{C}}\to \mathcal{C}_X$ is the $\Z_2$-bundle from Joyce--Tanaka--Upmeier \autocite[Definition 2.22]{JTU} applied to the positive Dirac operator $\slashed{D}_+: S_+\to S_-$ as in Cao--Gross--Joyce \autocite[Theorem 1.11]{CGJ}, Joyce--Upmeier, \cite[Theorem 13.5]{JU}. In particular, $O^{\mathcal{S}}\to \mathcal{M}_X$ is trivializable by \cite[Theorem 12.6]{JU}. .
\end{theorem}
\begin{proof}
This is a simple generalization of the proof of \autocite[Theorem 13.5]{JU} relying on the fact that Theorem \ref{Cao} requires $X^{\textnormal{an}}$ to be a spin manifold to trivialize the orientation bundle on $\mathcal{B}_X$. We mention it, as we hope that it will be useful in the future for generalization of DT$_4$ invariants. 
We only discuss the corresponding real structure on the differential geometric side replacing \cite[Definition 3.24]{CGJ}. We have the pairing
$
    \wedge^{S}: (\mathcal{A}^{0,k}\otimes \Theta)\otimes(\mathcal{A}^{0,4-k}\otimes \Theta)\to \mathcal{A}^ {4,4}\,,
$
and the corresponding \textit{spin Hodge star} $\star^S: \mathcal{A}^{0,q}\otimes \Theta\to \mathcal{A}^{0,n-q}\otimes \Theta$ 
$$
(\beta\otimes t)\wedge^S \star^S_k(\alpha\otimes s) = \langle\beta\otimes s,\alpha\otimes t \rangle\underline{\Omega} \qquad \alpha,\beta\in A^{0,k}, \, s,t \in \Gamma^{\infty}(\underline{\Omega})\,,
$$
where $\underline{\Omega}\in \mathcal{A}^{4,4}$ is the volume form.  
As a result, we have the real structures:
$
\#^S_1: \mathcal{A}^{0,\textnormal{even}}\otimes \Theta\to\mathcal{A}^{0,\textnormal{even}} \otimes \Theta$ and $
\#^S_2:\mathcal{A}^{0,\textnormal{odd}}\otimes \Theta\to \mathcal{A}^{0,\textnormal{odd}}\otimes \Theta$, where again $\#^S_1|_{\mathcal{A}^{0,2q}\otimes \Theta} = (-1)^q\star^S$ and $\#^S_2|_{\mathcal{A}^{0,2q+1}\otimes \Theta} =(-1)^{q+1}\star^S$. The Dolbeault operator commutes with these
$D_{\Theta}\circ\#^S_1=\#^S_2\circ D_{ \Theta}$ and its real part is the positive Dirac operator $\slashed{D}: S_+\to S_-$ by Friedrich \autocite[§3.4]{Fri}. As twisting by connections only corresponds to tensoring symbols of operators by identity, this extends also to real structures on $\textnormal{det}(D^{\nabla_{\textnormal{End}(E)}})$. 
\end{proof}

\begin{remark}
\label{spinsigns}
Note that one can also state the equivalent of Cao--Gross--Joyce \autocite[Theorem 1.15 (c)]{CGJ}, expressing the comparison of orientations under direct sums on $\mathcal{M}_X$ in terms of the comparison on $\mathcal{C}_X$. 
\end{remark}

 Suppose, that $X$ is Calabi--Yau and that there exists $Y$ smooth with an open embedding $i_Y: X\hookrightarrow Y$, where $Y$ is spin. Recall that we have the map $\xi_Y: \mathcal{M}_X\to \mathcal{M}_Y$ from Definition \ref{def unicom}. We say that $Y$ is a \textit{spin compactification} of $X$. We now state the weaker result about orientability for non-compact Calabi--Yau 4-folds.

\begin{corollary}
\label{sccor}
Let $X$ be a Calabi--Yau 4-fold, and let  $Y$ be a spin compactification of $X$ with a choice of $\Theta$ and an isomorphism $\phi:\mathcal{O}_X\xrightarrow{\sim}\Theta|_{X}$, then there exists an induced isomorphism of $\mathbb{Z}_2$ bundles on $\mathcal{M}_X$:
$$
O^\omega \cong \xi_Y^*(O^\Theta)\,.
$$
In particular, $\mathcal{M}_X$ is orientable if $Y$ satisfies Assumption \ref{ass:astcs}.
\end{corollary}
\begin{proof}
Let $E$ be a perfect complex on $X$ with a proper support, then the $\mathbb{Z}_2$ torsors over a $\textnormal{Spec}(A)$-point $[E]$ of both of the above $\mathbb{Z}_2$-bundles are given by 
$$\{o_{E}: \textnormal{det}\big(\underline{\textnormal{Hom}}(E,E)\big)\xlongrightarrow{\sim}\mathbb{C}\textnormal{   s.t. }o_{E}\otimes o_{E} = \textnormal{ad}(i^\omega)|_{[E]}\}\,$$
where $i^\omega$ is the Serre duality, and we used the isomorphism $E\otimes \Theta \cong E$ induced by $\phi$. Thus we have a natural identification of both $\mathbb{Z}_2$-bundles. The last statement follows from Theorem \ref{spinprop}.
\end{proof}

\begin{remark}
\label{remark comparison of spin}
If $Y$ is a spin compactification of $X$ and $Y\backslash X = D$ is a divisor. Let $D=\cup_{i=1}^{N}D_i$ be its decomposition into  irreducible components. If we can write the canonical divisor of $Y$ as  $K_{Y} = \sum_{i=1}^Na_i D_i$, where $a_i \equiv 0 (\textnormal{mod } 2)$, then one can take $$\Theta = \sum_{i=1}^N\frac{a_i}{2} D_i$$
as the square root. After choosing a meromorphic section $\bar{\Omega}^{\frac{1}{2}}$ of $\Theta$ with poles and zeros on $D$, one obtains an isomorphism $\phi:\mathcal{O}_X\xrightarrow{\sim} \Theta|_X $. Then the condition of Corollary \ref{sccor} is satisfied. 
\end{remark}
 
\begin{example}
The simplest example is $\mathbb{C}^4$. While its natural compactification $\mathbb{P}^4$ is not spin, one can choose to compactify it as $\mathbb{P}^1\times\mathbb{P}^3$ or $(\mathbb{P}^1)^{\times 4}$ which are both spin, both of which satisfy the property in Remark \ref{remark comparison of spin}.
\end{example}
\begin{example}
\label{vector bundle example}
Let $S$ be a smooth projective variety $0\leq\textnormal{dim}_\C(S) = k\leq 4$ and let $E\to S$ be a vector bundle, s.t. $\textnormal{det}(E) = K_S$. Then $X = \textnormal{Tot}(E\to S)$ is Calabi--Yau. Taking its smooth compactification $Y=\mathbb{P}(E\oplus \mathcal{O}_S)$ with the divisor at infinity $D = \mathbb{P}(E)\subset \mathbb{P}(E\oplus \mathcal{O}_S)$, one can show that $K_{Y} = -(\textnormal{rk}(E)+1)D$. If $\textnormal{rk}(E)\in 2\Z + 1$, we see that we can choose $\Theta =\mathcal{O}_{Y}(-\frac{(\textnormal{rk}(E)+1)}{2}D)$ which satisfies the property of Remark \ref{remark comparison of spin}. 
If $X = \textnormal{Tot}(L_1\oplus L_2\to S)$ for a smooth projective surface $S$ and its line bundles $L_1,L_2$, s.t. $L_1L_2=K_S$, then the spin compactification can be obtained as $\mathbb{P}(L_1\oplus \mathcal{O}_S)\times_{S}\mathbb{P}(L_2\oplus \mathcal{O}_S)$\,. 
\end{example}
\begin{example}
\label{example toricspin}
Suppose we have a toric variety $X$ (see Fulton \autocite{Fultontoric}, Cox \autocite{Cox}) given by a fan in the lattice $\mathbb{Z}^n\subset \mathbb{R}^n$. Suppose it is smooth and it contains the natural cone spanned by $(e_i)_{i=1}^n$. Define the hyperplanes 
$$
H_i = \{(x_1,\ldots, x_n)\in \R^n : \sum_{j=1}^nx_j = i \}\,.
$$
Then $X$ is Calabi--Yau if and only if all the primitive vectors of rays of the fan lie in $H_1$ and all the cones are spanned by a basis. A simple generalization of this well known statement shows that $X$ is spin if and only if all the primitive vectors lie in $H_{\textnormal{odd}} = \bigcup_{i\in 2\Z +1} H_i$\,. Starting from a toric Calabi--Yau $X$,  one can compactify $X$ to a projective smooth toric variety $Y$ by adding divisors corresponding to primitive vectors. In general, we will not have spin compactifications: 
 Consider the fan in $\mathbb{R}^2$ with more than 3 primitive vectors in $H_1$, then any compactification will be consecutive blow ups of a Hirzebruch surface at points, with at least one blow up. 

A common way of constructing Calabi--Yau manifolds is by removing anti-canonical divisors from a Fano manifold. To further illustrate the scarceness of spin-compactifica\-tions in even dimension, we study the classification of toric projective Fano fourfolds by Batyrev \cite{batyrev}. Using the condition described above, we can show that there exist only 4 smooth toric Fano fourfolds with a spin structure. These are $\mathbb{P}_{\mathbb{P}^3}(\mathcal{O}\oplus\mathcal{O}(2))$, $\mathbb{P}^1\times\mathbb{P}^3$, $\mathbb{P}^1\times \mathbb{P}_{\mathbb{P}^2}(\mathcal{O}\oplus\mathcal{O}(1))$ and $4\mathbb{P}^1$ corresponding to the polytopes $B_2,B_4, D_{12}$ and $L_8$ respectively. Note that there are 123 smooth projective toric Fano fourfolds in total. 
\end{example}

To avoid having to answer the question of existence of spin-compactifications, we develop a different general approach in the next section.

\subsection{Orientations for a non-compact Calabi--Yau via algebraic excision principle}

Let $(X,\Omega)$ be a Calabi--Yau fourfold, then we fix a compactification $Y$ with $D= Y\backslash X$ a \textit{strictly normal crossing divisor}, where we are using the following convention.

\begin{definition}
\label{defSNC}
A strictly normal crossing divisor is a normal crossing divisor, which is a union of smooth divisors with transversal intersections (any $k$-fold intersection is in particular smooth).
\end{definition}

Let $(X,\Omega)$ be a Calabi--Yau fourfold, then we fix a compactification $Y$ with $Y\backslash D$ a strictly normal crossing divisor. By Hironaka  \autocite[Main Theorem 1]{Hironaka}, Bierstone--Milman \cite{bierstone} there exists such a compactification by embedding into a projective space and taking resolutions. Consider triples complexes $(E,F,\phi)$, where $E,F\in L_{\textnormal{pe}}(Y)$ and  
$
\phi: E|_{D}\xrightarrow{\sim}  F|_{D}\,.
$
We will take the difference of the determinants 
$\textnormal{det}\big(\underline{\textnormal{Hom}}(E,E)\big)$ and $\textnormal{det}\big(\underline{\textnormal{Hom}}((F,F)\big)$ and cancel the contributions which live purely on the divisor. One could think of this as an algebraic version of the  excision principle defined for complex operators in §\ref{section pseudo}. Let us now make the described method more rigorous.

Let $X$, $Y$ and $D$ be as in the paragraph above, then we can write $D$ as the union
\begin{equation}
\label{divisordecomposition}
D=\bigcup_{i=1}^N D_i\,
\end{equation}
where each $D_i$ is a smooth divisor. We require $\Omega$ to be algebraic, then there exists a unique meromorphic section $\bar{\Omega}$ of $K_{Y}$, s.t. $\bar{\Omega}|_{X} =\Omega$. The poles and zeroes of $\bar{\Omega}$ express $K_{Y}$ uniquely in the following form
$
K_{Y} = \sum_{i=1}^N a_iD_i,$ where $a_i\in \Z.
$
We may write for the canonical line bundle:
\begin{equation}
\label{canbun}
  K_{Y} = \bigotimes_{i=1}^N\mathcal{O}(k_iD_i) =  \bigotimes_{i=1}^N\mathcal{O}(\textnormal{sgn}(k_i)D_i)^{\otimes |k_i|}\,.  
\end{equation}
Let $N_D$ be the free lattice spanned by the divisors $D_i$ which we from now on denote by the elements $e_i\in N_D$. For a line bundle $L = \otimes_{i=1}^N\mathcal{O}(a_iD_i)$ we write $L_{\underline{a}}$, where $\underline{a} = (a_1,\ldots,a_N)$. We will also use the notation $L_{\underline{k}} = K_Y$. Then for a non-zero global section $s_i$ of $\mathcal{O}(D_i)$ one has the usual exact sequence
$$
\begin{tikzcd}
  0\arrow[r]& L_{\underline{a}}\arrow[r,"\cdot s_i"]&L_{\underline{a}+e_i}\arrow[r]& L_{\underline{a}+e_i}\otimes_{\mathcal{O}_{Y}}\mathcal{O}_{D_i}\arrow[r]& 0\,. 
\end{tikzcd}
$$
As all the operations used to define 
$\mathcal{E}\textnormal{xt}_{\underline{a}} = \mathcal{E}\textnormal{xt}_{L_{\underline{a}}}$ and $\mathbb{P}_{\underline{a}}=\mathbb{P}_{L_{\underline{a}}}$
in Definition \ref{Plambda} are derived, we obtain distinguished triangles
\begin{equation}
\label{disttr}
    \begin{tikzcd}
        \mathcal{E}\textnormal{xt}_{\underline{a}}\arrow{r} & \mathcal{E}\textnormal{xt}_{\underline{a}+e_i}\arrow{r} & \mathcal{E}\textnormal{xt}_{L_{\underline{a}+e_i}\otimes_{\mathcal{O}_{Y}}\mathcal{O}_{D_i}} \arrow{r}{[1]} &  \mathcal{E}\textnormal{xt}_{\underline{a}}[1]\,,\\
        \mathbb{P}_{\underline{a}}\arrow{r} & \mathbb{P}_{\underline{a}+e_i}\arrow{r} & \mathbb{P}_{L_{\underline{a}+e_i}\otimes_{\mathcal{O}_{Y}}\mathcal{O}_{D_i}} \arrow{r}{[1]} &  \mathbb{P}_{\underline{a}}[1]\,. 
    \end{tikzcd}
\end{equation}
\sloppy By \eqref{mapping stack} both $\mathcal{M}_{Y}$ and $\mathcal{M}_{D_i}$ can be expressed as mapping stacks
$\textnormal{Map}\big(Y,\textnormal{Perf}_{\mathbb{C}}\big)$ and $\textnormal{Map}\big(D_i,\textnormal{Perf}_{\mathbb{C}}\big)$, respectively. Let $\textnormal{inc}_{D_i}:D_i\to Y$ be the inclusion, then we denote by $\rho_i:\mathcal{M}_{Y}\to \mathcal{M}_{D_i}$ the morphisms induced by the pullback $(\textnormal{inc}_{D_i})^* :L_{\textnormal{pe}}(Y)\to L_{\textnormal{pe}}(D_i)$, where we are using that both stacks are mapping stacks. For each divisor $D_i$ we set $L_{\underline{a}}|_{D_i} = L_{\underline{a},i}$ and
$$\mathcal{E}\textnormal{xt}_{\underline{a},i}=\pi_{2,3\,*}(\pi_{1,2}^*\,\mathcal{U}^\vee_{D_i}\otimes\pi_{1,3}^*\,\mathcal{U}_{D_i}\otimes\pi_1^*L_{\underline{a},i})\,,\qquad  \mathbb{P}_{\underline{a},i} = \Delta^*\mathcal{E}\textnormal{xt}_{\underline{a},i}\,,\\
$$
\begin{lemma}
We have the isomorphism 
\begin{align*}
\label{restriction}
    \mathcal{E}\textnormal{xt}_{L_{\underline{a}+e_i}\otimes_{\mathcal{O}_{Y}}\mathcal{O}_{D_i}} &\cong (\rho_i\times \rho_i)^*(\mathcal{E}\textnormal{xt}_{\underline{a}+e_i,i})\,,\\
    \mathbb{P}_{L_{\underline{a}+e_i}\otimes_{\mathcal{O}_{Y}}\mathcal{O}_{D_i}} &\cong \rho_i^*(\mathbb{P}_{\underline{a}+e_i,i})\,.
\end{align*}
where we use the same notation for the complexes $\mathbb{P}$ on $\mathcal{M}_{Y}$ and $\mathcal{M}_{D_i}$.
\end{lemma}
\begin{proof}
For universal complexes $\mathcal{U}_Y$, $\mathcal{U}_{D_i}$ on $Y\times \mathcal{M}_{Y}$, $D_i\times \mathcal{M}_{D_i}$, we have 
$
\mathcal{U}_Y|_{D_i\times \mathcal{M}_{Y}} = (\textnormal{id}_{D_i}\times\rho_i)^*\mathcal{U}_{D_i}
$ as follows from the commutative diagram
\begin{equation*}
    \begin{tikzcd}[column sep=4em, row sep=3em]
      Y\times\mathcal{M}_{Y}\arrow[r]&\textnormal{Perf}_{\mathbb{C}}\\
      D_i\times\mathcal{M}_{Y}\arrow[u, "\textnormal{inc}_{D_i}\times\textnormal{id}_{\mathcal{M}_{Y}}"]\arrow[r,"\textnormal{id}_{D_i}\times\rho_i"]&D_i\times\mathcal{M}_{D_i}\arrow[u]\,.
    \end{tikzcd}
\end{equation*}
For the dual, we also have $\mathcal{U}^\vee_Y|_{D_i\times \mathcal{M}_{Y}} = (\textnormal{id}_{D_i}\times\rho_i)^*\mathcal{U}^\vee_{D_i}.$ Thus we have the following equivalences
\begin{align*}
    &\mathcal{E}\textnormal{xt}_{L_{\underline{a}+e_i}\otimes_{\mathcal{O}_{Y}}\mathcal{O}_{D_i}}\\ =& \pi_{2,3\,*}\big(\pi_{1,2}^*(\mathcal{U}^\vee)\otimes\pi_{1,3}^*(\mathcal{U})\otimes\pi_1^*(L_{\underline{a}+e_i}\otimes_{\mathcal{O}_{Y}}\mathcal{O}_{D_i})\big)\\ 
    \cong &\pi_{2,3\,*}\Big((\textnormal{inc}_{D_i}\times \textnormal{id}_{\mathcal{M}_{Y}\times \mathcal{M}_{Y}} )_*\big(\mathcal{U}^\vee\otimes\mathcal{U}\otimes\pi_1^*L_{\underline{a}+e_i}\big)|_{D_i\times\mathcal{M}_{Y}\times \mathcal{M}_Y}\Big)\\
    \cong &\pi_{2,3\,*}\Big((\textnormal{inc}_{D_i}\times \textnormal{id}_{\mathcal{M}_{Y}\times \mathcal{M}_{Y}})_*\circ (\textnormal{id}_{D_i}\times\rho_i\times \rho_i)^*(\pi^*_{1,2}(\mathcal{U}^\vee_{D_i})\otimes\pi_{1,3}^*(\mathcal{U}_{D_i})\otimes\pi_1^*L_{\underline{a}+e_i,i})\Big)
    \\\cong &\pi_{2,3\,*}\Big((\textnormal{id}_{D_i}\times\rho_i\times\rho_i)^*(\pi_{1,2}^*(\mathcal{U}_{D_i}^\vee)\otimes\pi_{1,3}^*(\mathcal{U}_{D_i}))\otimes\pi_1^*(L_{\underline{a}+e_i,i})\Big) 
    \\ \cong & (\rho_i\times \rho_i)^*(\mathcal{E}\textnormal{xt}_{\underline{a}+e_i,i})\,,
\end{align*}
the first isomorphism is the projection formula \cite[Lem. 3.2.4]{Gaitsgorylecture1.3} and the last step follows from the base change isomorphism $\pi_{2,3\,*}\circ (\textnormal{id}_{D_i}\times \rho_i)^*\cong (\rho_i\times \rho_i)^*\circ \pi_{2,3\, *}$ using  Gaitsgory \cite[Prop. 2.2.2]{Gaitsgorylecture1.3}\footnote{These references are stated for derived stacks. So we should work with derived stacks until we construct the isomorphisms in Definition \ref{definitionbowtie}, which we can then restrict to its truncation.} and that the diagram
$$
\begin{tikzcd}
D\times \mathcal{M}_X\times \mathcal{M}_X\arrow[d]\arrow[r]&D\times M_D\times M_D\arrow[d]\arrow[r]&D\arrow[d]\\
M_X\times M_X\arrow[r]&M_D\times M_D\arrow[r]& *
\end{tikzcd}
$$
 consists of Cartesian diagrams by the pasting law in $\infty$-categories.
The second formula follows using $(\rho_i\times\rho_i)\circ \Delta_{\mathcal{M}_{Y}} = \Delta_{\mathcal{M}_{D_i}}\circ\rho_i:\mathcal{M}_X\to \mathcal{M}_{D_i}\times \mathcal{M}_{D_i}$.
\end{proof}
 After taking determinants of \eqref{disttr}, we obtain the isomorphisms
\begin{align*}
\label{divcon}
   \Sigma_{\underline{a}+e_i} \cong  \Sigma_{\underline{a}}\otimes \rho_i^*\Sigma_{\underline{a}+e_i,i}\,,\quad  \Lambda_{\underline{a}+e_i}\cong \Lambda_{\underline{a}}\otimes \big(\rho_i^*\Lambda_{\underline{a}+e_i,i})\,,
  \numberthis 
\end{align*}
where we omit writing $L$. We have the maps 
$
 i_{D_i}: D_i\to Y\,,i_{D}: D\to Y
$
inducing
\begin{align*}
    \rho_i:& \mathcal{M}_{Y} \longrightarrow \mathcal{M}_{D_i}\,,\qquad \rho_{D}: \mathcal{M}_{Y}\longrightarrow \mathcal{M}^{D}\,,\\
    \mathcal{M}^{\textnormal{sp}}_D &:=\prod_{i=1}^N \mathcal{M}_{D_i}\quad\textnormal{    and    }\quad \rho=\prod_{i=1}^N\rho_i:\mathcal{M}_{Y}\longrightarrow \mathcal{M}^{\textnormal{sp}}_D
\end{align*}

Note that we have the obvious map $\mathcal{M}^{D}\to \mathcal{M}_{D_i}$ induced by the inclusion $D_i\hookrightarrow D$. This gives
\begin{equation}
\label{sp}
\textnormal{sp}: \mathcal{M}_{Y}\times_{\mathcal{M}^D}\mathcal{M}_Y=\mathcal{M}_{Y,D}\longrightarrow \mathcal{M}_{Y}\times_{\mathcal{M}^{\textnormal{sp}}_{D}}\mathcal{M}_{Y}=\mathcal{M}^{\textnormal{sp}}_{Y,D}\,.
\end{equation}

\begin{definition}
\label{extension data}
For given $X$, $Y$ as above let $\bar{\Omega}$ be a meromorphic section of $K_{Y}$ restricting to $\Omega$. Let $\mathfrak{ord}$ denote the decomposition of $D$ into irreducible components as in \eqref{divisordecomposition}, which also specifies their order, such that there exist $0\leq N_1\leq N_2\leq N$, such that $a_i=0$ for $0<i\leq N_1$, $a_i>0$ for $N_1<i\leq N_2$ and $a_i<0$ for $N_2<i\leq N$, where $a_i$ are the coefficients from \eqref{canbun}. For the construction, we may assume $N_1=0$. We define \textit{extension data} as the following ordered collection of sections
$$
\bowtie=
\Big((s_{i,k})_{\begin{subarray}a i\in\{1,\ldots N_2\}\\
1\leq k\leq a_i 
\end{subarray}},(t_{j,l})_{\begin{subarray}a j\in\{N_2+1,\ldots N\}\\
1\leq l\leq -a_j 
\end{subarray}}\Big)\,, \quad s_{i,k}:\mathcal{O}_{Y}\to \mathcal{O}_{Y}(D_i)\,,\quad t_{j,l}:\mathcal{O}_{Y}\to \mathcal{O}_{Y}(D_j)\,.
$$ such that $\prod_{{\begin{subarray}a i\in\{1,\ldots N_2\}\\
1\leq k\leq a_i\end{subarray}}}s_{i,k}\prod_{\begin{subarray}a j\in \{N_2+1,\ldots, N\}\\
1\leq l\leq -a_j 
\end{subarray}}(t_{j,l})^{-1} = \Omega$ and $s_{i,k}$, $t_{j,k}$ are holomorphic with zeros only on $D_i$, resp. $D_j$.
\end{definition}
This leads to a definition of a new $\Z_2$-bundle:
\begin{definition}
\label{definitionbowtie}
On $\mathcal{M}^{\textnormal{sp}}_{Y,D}$ we have the line bundle 
\begin{equation}
\mathcal{L}_{Y,D} = \pi_1^*\Lambda_{\underline{0}}\otimes(\pi_2^*\Lambda_{\underline{0}})^*\,,
\end{equation}
where $\mathcal{M}_{Y}\xleftarrow{\pi_1} \mathcal{M}^{\textnormal{sp}}_{Y,D}\xrightarrow{\pi_2}\mathcal{M}_{Y}$ are the natural projections.

For a fixed choice $\bowtie$, there is a natural isomorphism
\begin{align*}
\label{taubowtie}
    \vartheta^{\textnormal{sp}}_{\bowtie}:\mathcal{L}_{Y,D}&\cong\pi_1^*\Lambda_{\underline{0}} \otimes \pi_1^*\circ\rho^*(\Lambda_{D})\otimes \pi_1^*\circ\rho^*(\Lambda_{D})^* \otimes(\pi_2^*\Lambda_{\underline{0}})^*\\
    &\cong \pi_1^*\Lambda_{\underline{k}}\otimes \pi_2^*(\Lambda_{\underline{k}})^*
   \cong  \pi_1^*(\Lambda_{\underline{0}})^*\otimes \pi_2^*(\Lambda_{\underline{0}})
  \cong \mathcal{L}^*_{Y,D}\,,
    \numberthis
\end{align*}
Here $\Lambda_{D}\to \mathcal{M}^{\textnormal{sp}}_{D}$ are line bundles, and we used the commutativity of 
\begin{equation*}
    \begin{tikzcd}
     \arrow[d,"\pi_1"] \mathcal{M}^{\textnormal{sp}}_{Y,D}\arrow[r,"\pi_2"]&\mathcal{M}_{Y}\arrow[d,"\rho"] \\
     \mathcal{M}_{Y}\arrow[r,"\rho"]&\mathcal{M}^{\textnormal{sp}}_{D}
    \end{tikzcd}\,
\end{equation*}
in the first step. The bundles $\Lambda_D =\Lambda^*_{D,-}\otimes\Lambda_{D,+}$ appear as the result of using chosen $s_{i,k} $ to construct isomorphism \eqref{divcon} for the first $N_2$ divisors, then $\Lambda_{D,-}$ is obtained from using $t_{j,k}^{-1}$ and \eqref{divcon}. Thus we will have the expressions:
\begin{align*}
 \Lambda_{D,+}=&\Lambda_{\underline{k}-\sum_{i=1}^{N_2-1}a_ie_i+(a_{N_2}-1)e_{N_2},N_2}\ldots\otimes\Lambda_{(\underline{k}-\sum_{i=1}^{N_2-1}a_ie_i),N_2}\otimes\ldots\\
 &\otimes\Lambda_{\underline{k}-(a-1)e_1,1}\otimes\ldots\otimes\Lambda_{\underline{k},1}\\
 \Lambda_{D,-}=&\Lambda_{\underline{k}-\sum^{N_2}_{i=1}a_ie_i,N_2+1}\otimes \ldots \ldots \otimes \Lambda_{-e_N,N}\,.
\end{align*}
The second to last step uses \eqref{i_L}.
We define the $\Z_2$-bundles by using Definition \ref{definition Z_2 bundles}:
\begin{align*}
\label{prodbunmor}
    \vartheta^{\textnormal{sp}}_{\bowtie} : \mathcal{L}_{Y,D}\to& (\mathcal{L}_{Y,D})^*\,,\qquad O^{\bowtie}_{\textnormal{sp}}\to \mathcal{M}^{\textnormal{sp}}_{Y,D}\,,\\
    O^{\bowtie} &= \textnormal{sp}^*\big(O^{\bowtie}_{\textnormal{sp}}\big) 
    \numberthis
\end{align*}
where $O^{\bowtie}_{\textnormal{sp}}$ associated to $\vartheta_{\bowtie}^{\textnormal{sp}}$.
\end{definition}
The important property of the $\mathbb{Z}_2$-bundle $O^{\bowtie}$ is that it is going to allow us to use index theoretic excision on the side of gauge theory to prove its triviality. One should think of the triples  $[E,F,\phi]$ which are the points in $\mathcal{M}_{Y,D}$ as similar objects to the relative pairs in \autocite[Definition 2.5]{Markus} with identification given in some neighborhood of the divisor $D$. The $\mathbb{Z}_2$-bundle $O^{\bowtie}$ only cares about the behavior of the complexes in $X$. 
\begin{definition}
 \label{definition Gamma}
 Recall that from  Definition \ref{definition before theorems} we have the maps  $\Gamma_{Y}:(\mathcal{M}_{Y})^{\textnormal{top}}\to \mathcal{C}_{Y}$ and $\Gamma_{D}:(\mathcal{M}^{D})^{\textnormal{top}}\to \mathcal{C}_{D}$, We define $\Gamma$ as the composition 
\begin{align*}
\label{Gamma}
  &(\mathcal{M}_{Y,D})^{\textnormal{top}}\longrightarrow  (\mathcal{M}_{Y})^{\textnormal{top}}\times^h_{(\mathcal{M}^{D})^{\textnormal{top}}}(\mathcal{M}_{Y})^{\textnormal{top}}
  \\
  &\longrightarrow \mathcal{C}_{Y}\times^h_{\mathcal{C}_{D}}\mathcal{C}_{Y}\simeq\mathcal{C}_{Y}\times_{\mathcal{C}_D}\mathcal{C}_Y=\mathcal{C}_{Y,D} \,.
  \numberthis
\end{align*}
The first map is induced by the homotopy commutative diagram obtained from applying $(-)^{\textnormal{top}}$ to the Cartesian diagram 
\begin{equation*}
   \begin{tikzcd}
\arrow[d]\mathcal{M}_{Y,D}\arrow[r]&\mathcal{M}_{Y}\arrow[d]\\
\mathcal{M}_{Y}\arrow[r]& \mathcal{M}^{D}
 \end{tikzcd} \,.
\end{equation*}
The second map uses homotopy commutativity of 
\begin{equation*}
\begin{tikzcd}
  (\mathcal{M}_{Y})^{\textnormal{top}} \arrow[d,"\Gamma_{Y}"]\arrow[r]& (\mathcal{M}^{D})^{\textnormal{top}}\arrow[d,"\Gamma_{D}"]&\arrow[l](\mathcal{M}_{Y})^{\textnormal{top}}\arrow[d,"\Gamma_{Y}"]\\
 \mathcal{C}_{Y}\arrow[r]&\mathcal{C}_{D}&
\arrow[l]\mathcal{C}_{Y}
\end{tikzcd}\,.
\end{equation*}
The final homotopy equivalence is the result of the map $(\textnormal{inc}_{D})^{\textnormal{an}}: (D)^{\textnormal{an}}\to Y^{\textnormal{an}}$ being a cofibration for the standard model structure on \textbf{Top}. The map $\mathcal{C}_{Y}\to \mathcal{C}_{D}$ is a fibration so the homotopy fiber-product is given by the strict fiber-product up to homotopy equivalences.

 \end{definition}
 
 We now state the theorem which follows from Proposition \ref{proposition main} below and is the main tool in proving orientability of $\mathcal{M}_X$.

\begin{theorem}
\label{maintheorem}
For $X,Y$ and $D$, fix $\mathfrak{ord}$ and the extension data $\bowtie$ as in Definition \ref{definitionbowtie}. For the $\mathbb{Z}_2$-bundles 
\begin{equation}
\label{Ovartheta}
O^{\bowtie}\to\mathcal{M}_{Y,D}\qquad\textnormal{and}\qquad D^{\mathcal{C}}_{O}\to \mathcal{C}_{Y,D}
\end{equation}
where the latter was defined in \eqref{DO bundle}, there exist an isomorphism
\begin{equation}
\label{isomorphism DCO Ovartheta}
    \mathfrak{I}^{\bowtie}: \Gamma^*(D^{\mathcal{C}}_O)\cong (O^{\bowtie})^{\textnormal{top}}\,.
\end{equation}
If $\tilde{Y}$ from Definition \ref{definition double} additionally satisfies Assumption \ref{ass:astcs}, then this isomorphism is canonical and $O^{\bowtie}$ is trivializable. By Lemma \ref{lem:assumptionfortildeY}, this holds when $H^3(X,\Z_2)=0=H^3_{\cs}(X,\Z_2)$.
\end{theorem}
We now re-interpret this result to apply it to the orientation bundle of interest $O^{\omega}\to \mathcal{M}_X$. 
\begin{definition}
\label{definition beforencmain}
Let $\zeta:\mathcal{M}_X\to \mathcal{M}_{Y,D}$ induced by the map $[0]\times \xi_Y: \textnormal{Spec}(\mathbb{C})\times \mathcal{M}_X\to \mathcal{M}_{Y,D}$. We have the commutative diagram
\begin{equation}
\label{factoring}
\begin{tikzcd}
   \arrow[d,"\Gamma^{\textnormal{cs}}"] \mathcal{M}^{\textnormal{top}}_X\arrow[r,"\zeta^{\textnormal{top}}"]&\mathcal{M}_{Y,D}\arrow[d,"\Gamma"]\\
   C^{\textnormal{cs}}_X\arrow[r,"i_{\textnormal{cs}}"]&\mathcal{C}_{Y,D}\,,
\end{tikzcd}
\end{equation}
where 
\begin{equation}
\label{eq:kappacs}
\kappa^{\textnormal{cs}}: \mathcal{C}^{\textnormal{cs}}_X =\{0\}\times_{\mathcal{C}_{D}}\mathcal{C}_{Y}\hookrightarrow \mathcal{C}_{Y,D},\end{equation}
and
$
  \{0\}\times_{\mathcal{C}_{D}} \mathcal{C}_{Y}= \textnormal{Map}_{C^0}\big((X^+,+),(BU\times\mathbb{Z},0)\big)\,.
$
 The space 
$C^{\textnormal{cs}}_X=\textnormal{Map}_{C^0}\big((X^+,+),(BU\times\mathbb{Z},0)\big)$
is the classifying space of compactly supported K-theory on $X$ (see Spanier \autocite{Spanier}, Ranicki--Roe \autocite[§2]{RaRo}, May \autocite[Chapter 21]{MayPonto}):
$
\pi_0(C^{\textnormal{cs}}_X) := K^0_{\textnormal{cs}}(X)\,.
$
We define 
\begin{equation}
\label{Ocs}
 O^{\textnormal{cs}}:=(\kappa^{\cs})^*(D^{\mathcal{C}}_O)\,.   
\end{equation}
\end{definition}
\begin{remark}
    For the purpose of the first orientability result in Theorem \ref{mainnctheorem}, $C^{\cs}_X$ mapping to the second factor of $C_{Y,D}$ is necessary due to the gluing set up in \eqref{gluing equation} and the corrections made in §\ref{sec:restriction} due to Assumption \ref{ass:astcs}. The author is being lazy here to avoid rewriting the entire section. Otherwise, it would be more natural to map to the first factor. If $\tilde{Y}$ from \ref{maintheorem} satisfies Assumption \ref{ass:astcs}, then the map should go to the first factor due to the signs observed in Definition \ref{definitiondual}. We will use the mapping to the first factor in §\ref{bigsection signs}.
\end{remark}
The following is the first important consequence of Theorem \ref{maintheorem} and leads to the construction of virtual fundamental classes using Borisov--Joyce \cite{BJ} or Oh--Thomas \cite{OT} using the $-2$-shifted symplectic structures. It also gives preferred choices of orientations at fixed points up to a global sign when defining/computing invariants using localization as in \cite{OT, CL} for moduli spaces of compactly supported sheaves $M_{\alpha}$ with a fixed $K$-theory class $\alpha\in K^0_{\textnormal{cs}}$ and for a given compactification $Y$.
\begin{theorem}
\label{mainnctheorem}
Let $(X,\Omega)$ be a Calabi--Yau fourfold satisfying Assumption \ref{ass:astcs}. Then the $\mathbb{Z}_2$-bundle 
$O^\omega\to \mathcal{M}_X$ is trivializable. Moreover, for a fixed choice of an embedding $X\to Y$, with $D=Y\backslash X$ strictly normal crossing, there exists an isomorphism
$$
\mathfrak{I}:(\Gamma^{\cs}_X)^*(O^{\cs})\cong (O^{\omega})^{\textnormal{top}}
$$
which is canonical if $\tilde{Y}$ from Definition \ref{definition double} satisfies Assumption \ref{ass:astcs}. By Lemma \ref{lem:assumptionfortildeY}, this last statements holds whenever $H^3(X,\Z_2) = 0 = H^3_{\cs}(X,\Z_2)$.
\end{theorem}
\begin{proof}
To understand how Assumption \ref{ass:astcs} comes in, see the discussions in §\ref{sec:restriction}. With this in mind, we prove the statement in 3 steps:

1. There is an natural isomorphism $\zeta^*(O^{\bowtie}) \cong O^\omega$: Consider a $\textnormal{Spec}(A)$-point $[E]\in \mathcal{M}_{X}$, then at the corresponding point $([\tilde{E}],0)\in \mathcal{M}_{Y,D}$, the isomorphism $\vartheta_{\bowtie}$ is given by the dual of
$$
\mathbb{C}\otimes\textnormal{det}(\underline{\textnormal{Hom}}(\tilde{E},\tilde{E})) \cong \textnormal{det}^*(\underline{\textnormal{Hom}}(\tilde{E},\tilde{E}\otimes K_{Y}))\cong \textnormal{det}^*(\underline{\textnormal{Hom}}(\tilde{E},\tilde{E}))\,,
$$
where we use that $\Lambda_{L}|_{[0]}\cong \mathbb{C}$, the first isomorphism is Serre duality, and the second one is the composition of isomorphisms induced by $E\xrightarrow{s_{i,k}}E(D_i)$ and $E(-D_j)\xrightarrow{t_{j,l}}E$. As $E$ is compactly supported in $X$, these isomorphisms compose into $E\xrightarrow{\Omega}E\otimes K_{Y}$ by the assumption on $\bowtie$. Therefore $\vartheta_{\bowtie}|_{([\tilde{E}],0)}$ coincides with $i^\omega|_{[E]}$ and their associated $\mathbb{Z}_2$-bundles are identified. 

2. By Lemma \ref{lemma independent}, we know $O^{\textnormal{cs}}$, $O^\omega$ are independent of the choice of $\bowtie$. We define a family of $\mathfrak{I}(\bowtie)$ for fixed $\mathfrak{ord}$:
$$
\mathfrak{I}(\bowtie):(\Gamma^{\cs}_{Y})^*(O^{\cs})\stackrel{\eqref{factoring}}{\cong} (\zeta^{\textnormal{top}})^*\circ \Gamma^*(D^{\mathcal{C}}_O)\cong(\zeta^{\textnormal{top}})^*\circ (O^{\bowtie})^{\textnormal{top}}\stackrel{1.}{\cong} (O^\omega)^{\textnormal{top}}\,.
$$
 Any two choices of $s_{i,k}$ differ by $\C^*$ (and same holds for $t_{j,l}$). Therefore the set of $\bowtie$ corresponds to $(\C^*)^{\sum_{i}^N|a_i|-1}$ which is connected and $\mathfrak{I}(\bowtie)$ does not depend on $\bowtie$.
 
3. \sloppy For simplicity, let us assume we only have two different divisors $D_1,D_2$. We then have the isomorphism obtained from applying \eqref{divcon} twice
\begin{align}
\label{D_2pi1}
\pi_1^*\big(\Lambda_{e_1+e_2}\big)\cong& \pi_1^*(\Lambda_{\underline{0}})(\rho_2\circ\pi_1)^*\Lambda_{e_2,2}(\rho_1\circ\pi_1)^*\Lambda_{e_1+e_2,1}\,,\\
\label{D_1pi1}
 \pi_1^*\big(\Lambda_{e_1+e_2}\big)\cong& \pi_1^*(\Lambda_{\underline{0}})(\rho_1\circ\pi_1)^*\Lambda_{e_1,1}(\rho_2\circ\pi_1)^*\Lambda_{e_1+e_2,2}\,,\\
 \label{D_2pi2}
\pi_2^*\big(\Lambda_{e_1+e_2}\big)\cong& \pi_2^*(\Lambda_{\underline{0}})(\rho_2\circ\pi_2)^*\Lambda_{e_2,2}(\rho_1\circ\pi_2)^*\Lambda_{e_1+e_2,1}\,,\\
\label{D_1pi2}
 \pi_2^*\big(\Lambda_{e_1+e_2}\big)\cong& \pi_2^*(\Lambda_{\underline{0}})(\rho_1\circ\pi_2)^*\Lambda_{e_1,1}(\rho_2\circ\pi_2)^*\Lambda_{e_1+e_2,2}\,.
\end{align}
To show that there is no difference between the chosen two orders, we use the commutative diagram, where all rows and columns fit into distinguished triangles:
\begin{equation}
\label{3x3diag}
    \begin{tikzcd}
   \mathbb{P}\arrow[d]\arrow[r]&\mathbb{P}_{e_1}\arrow[d]\arrow[r]&\arrow[d]\rho_1^*\big(\mathbb{P}_{e_1,1}\big)\\
      \mathbb{P}_{e_2}\arrow[d]\arrow[r]&\mathbb{P}_{e_1+e_2}\arrow[d]\arrow[r]&\arrow[d]\rho_1^*\big(\mathbb{P}_{e_1+e_2,1}\big)\\
          \rho_2^*\big(\mathbb{P}_{e_2,2}\big)\arrow[r]&\rho_2^*(\mathbb{P}_{e_1+e_2,2})\arrow[r]&\rho^*_{1,2}\big(\mathbb{P}_{e_1+e_2,1,2}\big)
         \end{tikzcd}
         \end{equation}
We used in the bottom right corner term the restriction $\rho_{1,2}:\mathcal{M}_Y\to \mathcal{M}_{D_1\cap D_2}$ and $\mathbb{P}_{e_1+e_2,1,2} = \mathbb{P}_{\mathcal{O}(D_1+D_2)|_{D_1\cap D_2}}$.
Taking determinants  of all four corner terms of the diagram (see Knudsen--Mumford \cite[Prop. 1]{knumum}) and pulling back by $\pi_1$, we get both \eqref{D_2pi1} and \eqref{D_1pi1} where the latter comes with $(-1)^{\textnormal{deg}\big((\rho_{D_2\circ \pi_1})^*\Lambda_{\mathcal{O}(D_1)|_{D_1}}\big)\textnormal{deg}\big((\rho_1\circ \pi_1)^*\Lambda_{\mathcal{O}(D_2)|_{D_2}}\big)}$.
This holds also for  \eqref{D_2pi2}, \eqref{D_1pi2}. By commutativity of the diagram, we see that choosing the steps \eqref{D_2pi1}, \eqref{D_2pi2} or \eqref{D_1pi1}, \eqref{D_1pi2} we obtain the same as the signs cancel.  As this permutes any two divisors, we obtain independence in the general case.  From Lemma \ref{lemma independent}, Proposition \ref{proposition Ovartheta Hprincipal} and Proposition \ref{proposition main}, $D^{\mathcal{C}}_O$ and $\mathfrak{I}$ are independent of the order.  Note that this should be all considered under pull-back by $\textnormal{sp}: \mathcal{M}_{Y,D}\to \mathcal{M}_{Y,D}^{\textnormal{sp}}$ to have natural isomorphism independent of choices on the smooth intersection $D_1\cap D_2$:
$$
\textnormal{sp}^*\big(\pi^*\circ\rho_{1,2}^*(\mathbb{P}_{e_1+e_2,1,2})\big)\cong \textnormal{sp}^*\big(\pi_2^*\circ\rho_{1,2}^*(\mathbb{P}_{e_1+e_2,1,2})\big)\,.
$$
\end{proof}
\begin{remark}
\label{remark joyce suggestion}
If the need arises to show independence of compactification, one can relate two compactifications $Y_1 \leftarrow \tilde{Y}\rightarrow Y_2$ by a common one obtained as a blow up of the closure of $X\hookrightarrow Y_1\times Y_2$. Then one could use \cite[Thm. 1.2]{RYY} comparing Hodge cohomologies for locally free sheaves under blow up in hopes of showing independence of the isomorphism $\mathfrak{I}$.
\end{remark}
Let us discuss another straight-forward consequence of the framework used in Theorem \ref{maintheorem}. For $(X,\Omega)$ a quasi-projective Calabi--Yau fourfold, let $M$ be a moduli scheme of stable pairs $\mathcal{O}_X\to F$ where $F$ is compactly supported (see  \autocite{CMT2, CKM2, CT2, JoyceSong, todacurve}) or ideal sheaves of proper subvarieties. To make sense out of Serre duality, generalizing the approach in Kool--Thomas \autocite[§3]{KT1} and Maulik--Pandharipande--Thomas \cite[§3.2]{MPT}, we choose a compactification $Y$ as in Theorem \ref{maintheorem}.

\begin{definition}
\label{definition orientation stablepairs}
 Let $\mathcal{E} = (\mathcal{O}_{X\times M}\to \mathcal{F})\to X\times M$ be the universal perfect complex. Using the inclusion $i_X:X\to Y$ we obtain the universal complex $\bar{\mathcal{E}} = \big(\mathcal{O}_{Y\times M}\to (i_X\times\textnormal{id}_M)_*(\mathcal{F})\big)\to Y\times M$. We have the following isomorphism, where $(-)_0$ denotes the trace-less part:
\begin{align*}
    i^\omega_M:\,\,& \textnormal{det}\big(\underline{\textnormal{Hom}}_M(\mathcal{E},\mathcal{E})_0\big)\cong \textnormal{det}\big(\underline{\textnormal{Hom}}_M(\bar{\mathcal{E}},\bar{\mathcal{E}})_0\big)\cong \textnormal{det}^*\big(\underline{\textnormal{Hom}}_M(\bar{\mathcal{E}},\bar{\mathcal{E}}\otimes L_{\underline{k}})_0\big)\\&\stackrel{\kappa}{\cong}\textnormal{det}^*\big(\underline{\textnormal{Hom}}_M(\bar{\mathcal{E}},\bar{\mathcal{E}})_0\big)\cong \textnormal{det}^*\big(\underline{\textnormal{Hom}}_M(\mathcal{E},\mathcal{E})_0\big)\,,
\end{align*}
where $\kappa$ is constructed using the isomorphisms 
\begin{equation}
\label{pairsstep}
\underline{\textnormal{Hom}}_M(\bar{\mathcal{E}},\bar{\mathcal{E}}\otimes L_{\underline{a}})_0\cong\underline{\textnormal{Hom}}_M(\bar{\mathcal{E}},\bar{\mathcal{E}}\otimes L_{\underline{a}-e_i})_0\end{equation}
 in each step determined by $\bowtie$ as in Definition \ref{definitionbowtie}. 
The orientation bundle $O^\omega_M\to M$ is defined as the square root $\Z_2$-bundle associated with $i^\omega_M$.
\end{definition}
Let $\bar{\mathcal{M}}$ be a moduli stack of stable pairs or ideal sheaves on $Y$  of subvarieties with proper support in $X$ with the projection $\pi_{\mathbb{G}_m}: \bar{\mathcal{M}}\to M$ which is a $[*/\mathbb{G}_m]$ torsor. We have an inclusion $\eta: \bar{\mathcal{M}}\to \mathcal{M}_{Y,D}$ given on $\textnormal{Spec}(A)$-points by mapping $[E]\mapsto ([\mathcal{O}_{Y}, \bar{E}])$, where $\bar{E}= i_{X\,*}(E)$.

The following result leads to the construction of virtual fundamental classes when $M$ is compact (assuming one believes the existence of shifted symplectic structures on pairs as in Preygel \cite{preygel} or Bussi \cite{bussi}) and preferred choices of orientations up to a global sign at fixed points when using localization for a fixed K-theory class  $\llbracket\mathcal{O}_X\rrbracket+\alpha$ for $\alpha\in K^0_{\textnormal{cs}}(X)$ and a choice of compactification $Y$. The latter statement requires again that $\tilde{Y}$ satisfies Assumption \ref{ass:astcs} so that we get a canonical isomorphism \eqref{isomorphism DCO Ovartheta}.

\begin{theorem}
\label{theorem stable pairs}
Let $(X,\Omega)$ be a quasi-projective Calabi--Yau fourfold and let $Y$ be its compactification as in Theorem \ref{maintheorem}. Let $O^\omega_M\to M$ be the orientation bundle from Definition \ref{definition orientation stablepairs} for $M$ a moduli scheme of stable pairs or ideal sheaves of proper subschemes of $X$. There is a canonical isomorphism of $\Z_2$-bundles
$$
\pi_{\mathbb{G}_m}^*(O^\omega_M)\cong \eta^*(O^{\bowtie})\,.
$$
In particular, $O^\omega_M\to M$ is trivializable if $X$ satisfies Assumption \ref{ass:astcs}.
\end{theorem}
\begin{proof}
The universal perfect complex $\mathcal{E}_{\bar{\mathcal{M}}}$ on $\bar{\mathcal{M}}$ is given by $(\textnormal{id}_{Y}\times \pi_{\mathbb{G}_m})^*(\bar{\mathcal{E}})$. We have: 
\begin{align*}
\gamma:\textnormal{det}\big(\underline{\textnormal{Hom}}_{\bar{\mathcal{M}}}(\mathcal{E}_{\bar{\mathcal{M}}}, \mathcal{E}_{\bar{\mathcal{M}}})_0\big)\cong \textnormal{det}\big(\underline{\textnormal{Hom}}_{\bar{\mathcal{M}}}(\mathcal{E}_{\bar{\mathcal{M}}}, \mathcal{E}_{\bar{\mathcal{M}}})\big)\textnormal{det}^*\big(\underline{\textnormal{Hom}}_{\bar{\mathcal{M}}}(\mathcal{O},\mathcal{O})\big)
\cong \eta^*(\mathcal{L}_{Y,D})\,,
\end{align*}
such that the following diagram of isomorphism commutes:
\begin{equation*}
    \begin{tikzcd}[column sep= huge]
    \eta^*(\mathcal{L}_{Y,D})\arrow[r,"\sim","\eta^*(\vartheta^{\bowtie})"']&\eta^*(\mathcal{L}_{Y,D})^*\arrow[d,"\sim","\gamma^{-*}"']\\
    \arrow[u,"\sim", "\gamma"']\textnormal{det}\big(\underline{\textnormal{Hom}}_{\bar{\mathcal{M}}}(\mathcal{E}_{\bar{\mathcal{M}}}, \mathcal{E}_{\bar{\mathcal{M}}})_0\big)\arrow[r,"\sim" ,"\pi^*_{\mathbb{G}_m}(i^\omega_M)"']&\textnormal{det}^*\big(\underline{\textnormal{Hom}}_{\bar{\mathcal{M}}}(\mathcal{E}_{\bar{\mathcal{M}}}, \mathcal{E}_{\bar{\mathcal{M}}})_0\big)^*
    \end{tikzcd}
\end{equation*}
which follows from the commutativity of
\begin{equation*}
\adjustbox{scale=0.8, center}{%
\begin{tikzcd}
   \arrow[d] \underline{\textnormal{Hom}}_{\bar{\mathcal{M}}}\big(\mathcal{E}_{\bar{\mathcal{M}}},\mathcal{E}_{\bar{\mathcal{M}}}\otimes L_{\underline{a}-e_i}\big)_0\arrow[r]&\arrow[d]\underline{\textnormal{Hom}}_{\bar{\mathcal{M}}}\big(\mathcal{E}_{\bar{\mathcal{M}}},\mathcal{E}_{\bar{\mathcal{M}}}\otimes L_{\underline{a}-e_i}\big)\arrow[r,"\textnormal{tr}"]&\arrow[d]\underline{\textnormal{Hom}}_{\bar{\mathcal{M}}}\big(\mathcal{O},\mathcal{O}\otimes L_{\underline{a}-e_i}\big)\\
   \arrow[d]\underline{\textnormal{Hom}}_{\bar{\mathcal{M}}}\big(\mathcal{E}_{\bar{\mathcal{M}}},\mathcal{E}_{\bar{\mathcal{M}}}\otimes L_{\underline{a}}\big)_0\arrow[r]&\arrow[d]\underline{\textnormal{Hom}}_{\bar{\mathcal{M}}}\big(\mathcal{E}_{\bar{\mathcal{M}}},\mathcal{E}_{\bar{\mathcal{M}}}\otimes L_{\underline{a}}\big)\arrow[r,"\textnormal{tr}"]&\underline{\textnormal{Hom}}_{\bar{\mathcal{M}}}\big(\mathcal{O},\mathcal{O}\otimes L_{\underline{a}}\big)\arrow[d]\\
0\arrow[r]&\underline{\textnormal{Hom}}_{\bar{\mathcal{M}}}\big(\mathcal{E}_{\bar{\mathcal{M}}},\mathcal{E}_{\bar{\mathcal{M}}}\otimes L_{\underline{a},i}\big)\arrow[r,"\textnormal{tr}"]&\underline{\textnormal{Hom}}_{\bar{\mathcal{M}}}\big(\mathcal{O},\mathcal{O}\otimes L_{\underline{a},i}\big)
    \end{tikzcd}}
\end{equation*}
in each step \eqref{pairsstep}.
As a result, the $\Z_2$-bundles associated to these are canonically isomorphic. To prove that $O^{\omega}_M$ is orientable, one additionaly needs the arguments in §\ref{sec:restriction}.
\end{proof}

The reader interested in computations is welcome to skip the next two sections. We discuss sign comparisons under direct sums of perfect complexes in §\ref{bigsection signs}.
\section{Background and some new methods}
\label{bigsection technical proof}
In this section, we review and develop further the necessary language for working with orientations. This includes developing an excision principle for complex determinant line bundles generalizing the work of Upmeier \cite{Markus}, Donaldson \cite{Donaldson}, \cite{DK} and Atiyah--Singer \cite{AS1}.

 \subsection{Topological stacks}

\label{backtopsta}
The definition of a topological stack follows at first the standard definition of stacks over the standard site of topological spaces. It can be found together with all basic results in Noohi \autocite{Noohi1} and Metzler \autocite{Metzler}, the homotopy theory of topological stacks is developed by Noohi in \autocite{Noohi2}. For each groupoid of topological spaces $[G\rightrightarrows X]$, one defines a prestack $\lfloor X/G\rfloor$, such that the objects of the groupoid $\lfloor X/G\rfloor(W)$ correspond to the continuous maps $W\to X$ for any $W\in \textnormal{Ob}(\textbf{Top})$. The morphisms between $\alpha:W\to X$ and $\beta: W\to X$ correspond to  $\lambda:W\to G$ which are mapped respectively to $\alpha$ and $\beta$ under the two maps $G\rightrightarrows X$. One also defines $[X/G]$  as the stack associated to this prestack. The following result makes working with topological stacks much easier. 
\begin{proposition}[Noohi {\autocite[p.26]{Noohi1}}]
Every topological stack $\mathcal{X}$ has the form of an associated stack $[X/G]$ for some topological groupoid $[G\rightrightarrows X]$. The canonical map $X\to[X/G]$ gives a chart of $\mathcal{X}$. Conversely $[X/G]$ associated to any groupoid is a topological stack. 
\end{proposition}

\begin{remark}The definition of a topological stack given in \autocite{Noohi1} is more complicated and depends on the choice of a class of morphisms called local fibrations (LF). Instead, we are using Noohi's definition of topological stacks from \autocite{Noohi2} which corresponds to pretopological stacks in \autocite{Noohi1}.  
\end{remark}

In \autocite{Noohi2}, Noohi proposes a homotopy theory for a class of topological stacks called \textit{hoparacompact}. 
Let $\textbf{tSt}_{\textbf{hp}}$ denote the 2-category of hoparacompact topological stacks. A\textit{ classifying space} of $\mathcal{X}$ in $\normalfont\textbf{tSt}_{\textbf{hp}}$ is a topological space $X = \mathcal{X}^{cla}$ with a representable map $\pi^{\textnormal{cla}}: X\to\mathcal{X}$ such that for any $T\to \mathcal{X}$, where $T$ is a topological space, its base change $T\times_\mathcal{X}X\to T$ is a weak homotopy equivalence.

\sloppy In \autocite[§8.1]{Noohi2}, Noohi provides a functorial construction of the classifying space $\mathcal{X}^{\textnormal{cla}}$ for every hoparacompact topological stack $\mathcal{X}$, such that the resulting space is paracompact. In fact,  \autocite[Corollary 8.9]{Noohi2} states that the functor 
$$(-)^{\textnormal{cla}}:\textnormal{Ho}(\textbf{tSt}_{\textbf{hp}})\longrightarrow \textnormal{Ho}(\textbf{pTop})$$
is an equivalence of categories, where \textbf{pTop} denotes the category of paracompact topological spaces. Note that as, we are interested in comparing $\Z_2$-bundles, it is enough to restrict to finite CW complexes and weak homotopy equivalences are replaced by usual ones avoiding the question of ghost maps.

\subsection{Moduli stack of connections and their $\Z_2$-graded $H$-principal $\Z_2$-bundles}
\label{Hspaces}
Let $X$ be a smooth connected manifold of dimension $n$ and $\pi: P\to X$ be a principal $G$ bundle for a connected Lie group $G$ with the Lie algebra $\mathfrak{g}$. We have the action of the gauge group $\mathcal{G}_P$ on the space of connections $\mathcal{A}_P$.  This action will be continuous and the spaces are paracompact because they are infinity CW-complexes, so we get a hoparacompact stack $\mathcal{B}_P = [\mathcal{A}_P/\mathcal{G}_P]$.

If $X$ is a compact spin Kähler fourfold, let $S_+,S_-$ denote the positive and negative spinor bundles and $\slashed{D}:S_+\to S_-$ the positive Dirac operator, then for each connection $\nabla_P\in \mathcal{A}_P$ one can define the \textit{twisted Dirac operator}
\begin{equation}
\label{twistedop}
   \slashed{D}^{\nabla_{\textnormal{ad}(P)}}:\Gamma^{\infty}(\textnormal{ad}(P)\otimes S_+)\to \Gamma^{\infty}(\textnormal{ad}(P)\otimes S_-)\,.
\end{equation}
This induces an $\mathcal{A}_P$ family of real elliptic operators and therefore gives by §\ref{section pseudo} a real line bundle $\textnormal{det}^{\slashed{D}}_P\to \mathcal{A}_P$. Because $\mathcal{G}_P$ maps the kernel of \eqref{twistedop} to the kernel and same for the cokernels, this $\R$-bundle is $\mathcal{G}_p$ equivariant and descends to an $\R$-bundle on $\mathcal{B}_P$. The orientation bundle of which we denote by $O^{\slashed{D}}_P\to \mathcal{B}_P$. One takes the unions over all isomorphism classes of $U(n)$-bundles for all $n$:
\begin{equation}
\label{Odiracbundle}
  \mathcal{B}_X = \bigsqcup_{[P]}\mathcal{B}_P\,,\qquad O^{\slashed{D}} = \bigsqcup_{[P]}O^{\slashed{D}}_P\,. 
\end{equation}
These are the \textit{orientation bundles} of Joyce--Tanaka--Upmeier \autocite{JTU} and  Cao--Gross--Joyce \autocite{CGJ}. 
For the proof of Theorem \ref{maintheorem}, we will rely on the properties of special principal $\mathbb{Z}_2$-bundles under homotopy-theoretic group completion of H-spaces. For background on H-spaces, see Hatcher \autocite[§3.C]{Hatcher}, May--Ponto \cite[§9.2]{MayPonto} and Cao-Gross--Joyce \autocite[§3.1]{CGJ}.
Recall that an \textit{(admissible) H-space }is a triple $(X,e_X,\mu_X)$ of a topological space $X$, its point $e_x\in X$ and a continuous map $\mu_X: X\times X\to X$ is called an \textit{H-space}, if it induces a commutative monoid in $\textnormal{Ho}(\textbf{Top})$.  An admissible H-space $X$ is \textit{group-like} if $\pi_0(X)$ is a group.  Note that there are many ways how to include higher homotopies into the theory of H-spaces. For $A^n$-spaces see Stasheff \autocite{StasheffI} and \autocite{StasheffII}. For $E^\infty$-spaces see May \autocite{May}, for $\Gamma$-spaces see Segal \autocite{Segal}. While $E^\infty$-spaces and $\Gamma$-spaces are roughly the same, $A^\infty$ spaces do not require commutativity. All our spaces fit into these frameworks which by  \autocite[Example 2.19]{JTU} give us additional control over the $\mathbb{Z}_2$-bundles on them. One also defines H-maps as the obvious maps in the category of H-spaces. We use the notion of \textit{homotopy-theoretic group completions} from May \autocite[§1]{May}. One has the following universality result for homotopy theoretic group completion, that we will use throughout. 

\begin{proposition}[Caruso--Cohen--May--Taylor {\autocite[Proposition 1.2]{CCMT}}]
\label{proposition homotopy theoretic group completion}
Let $f:X\to Y$ be a homotopy-theoretic group-completion. If $\pi_0(X)$ contains a countable cofinal sequence, then for each weak H-map $g: X\to Z$, where $Z$ is group-like, there exists a weak H-map $g': Y\to Z$ unique up to weak homotopy equivalence, such that $g'\circ f$ is weakly homotopy equivalent to $g$.
\end{proposition}

Note that weak H-maps correspond to relaxing the commutativity to hold only up to weak homotopy equivalences. We will again not differentiate between the two. Let us now merge the definition of $\Z_2$-graded commutativity with the notion of H-principal $\Z_2$-bundles of Cao--Gross--Joyce \cite{CGJ}.
\begin{definition}
\label{definition weakstrong Hprincipal}
Let $(X,e_X,\mu_X)$ be an H-space. A $\mathbb{Z}_2$-bundle $O\to X$ together with a continuous map $\textnormal{deg}(O):X\to \mathbb{Z}_2$  satisfying $$\textnormal{deg}(O)\circ\mu(x,y) = \textnormal{deg}(O)(x) + \textnormal{deg}(O)(y)$$
is a $\mathbb{Z}_2$-\textit{graded} $\mathbb{Z}_2$\textit{-bundle}. If $O_1,O_2$ are $\mathbb{Z}_2$-graded then the isomorphism
$$
O_1\otimes _{\mathbb{Z}_2}O_2\cong O_2\otimes _{\mathbb{Z}_2}O_1
$$
differs by the sign $(-1)^{\textnormal{deg}(O_1)\textnormal{deg}(O_2)}$ from the naive one. Moreover, $O_1\otimes_{\mathbb{Z}_2} O_2$ has grading $\textnormal{deg}(O_1) + \textnormal{deg}(O_2)$. A pullback of a $\mathbb{Z}_2$-graded $\mathbb{Z}_2$-bundle, is naturally $\mathbb{Z}_2$-graded. An isomorphism of $\mathbb{Z}_2$-graded $\mathbb{Z}_2$-bundles has to preserve the grading.
A \textit{weak H-principal $\Z_2$-graded $\mathbb{Z}_2$-bundle} on $X$ is a $\Z_2$-graded $\mathbb{Z}_2$-bundle $P\to X$, such that there exists an isomorphism $p$ of $\Z_2$-graded $\mathbb{Z}_2$-bundles on $X\times X$ 
$$
p:P\boxtimes_{\mathbb{Z}_2}P\to \mu_X^*(P)\,.
$$
A \textit{$\Z_2$-graded strong H-principal $\mathbb{Z}_2$-bundle} on $X$ is a pair $(Q,q)$ of a trivializable $\Z_2$ graded $\mathbb{Z}_2$-bundle $Q\to X$ and an isomorphism of $\Z_2$-graded $\mathbb{Z}_2$ bundles on $X\times X$
$$
q:Q\boxtimes_{\mathbb{Z}_2}Q\to \mu_X^*(Q)\,,
$$
such that under the homotopy $h: \mu_X\circ(\textnormal{id}_X\times\mu_X)\simeq \mu_X\circ(\mu_X\times\textnormal{id}_X)$ the following two isomorphisms of the $\mathbb{Z}_2$-bundles on $X\times X\times X$ are identified

$$
(\textnormal{id}_X\times \mu_X)^*(q)\circ (\textnormal{id}\times q): Q\boxtimes_ {\mathbb{Z}_2}Q\boxtimes_{\mathbb{Z}_2}Q\to\big(\mu_X\circ(\textnormal{id}_X\times\mu_X)\big)^*Q
$$
and 
$$
( \mu_X\times \textnormal{id}_X)^*(q)\circ (q\times \textnormal{id}): Q\boxtimes_ {\mathbb{Z}_2}Q\boxtimes_{\mathbb{Z}_2}Q\to\big(\mu_X\circ(\mu_X\times \textnormal{id}_X)\big)^*Q\,.
$$
 The isomorphism $i:(P,p)\to (Q,q)$ of $\Z_2$-graded strong H-principal $\mathbb{Z}_2$-bundles has to solve
$
\mu_X^*i\circ p = q\circ(i\boxtimes i)\,.
$
\end{definition}
 Including the $\Z_2$-gradedness, we obtain a minor modification of Cao--Gross--Joyce {\cite[Proposition 3.5]{CGJ}}.

\begin{proposition}
\label{exun}
Let $f: X\to Y$ be a homotopy-theoretic group completion of H-spaces, then for 
\begin{enumerate}[label=(\roman*)]
\item a $\Z_2$-graded weak H-principal $\mathbb{Z}_2$-bundle $P\to X$, there exists a unique $\Z_2$-graded weak H-principal $\mathbb{Z}_2$-bundle $P'\to Y$ such that $f^*(P')$ is isomorphic to $P$. 
\item a $\Z_2$-graded strong H-principal $\mathbb{Z}_2$-bundle $(Q,q)$ on $X$, there exists a unique $\Z_2$-graded strong H-principal $\mathbb{Z}_2$-bundle $(Q',q')$ on $Y$ unique up to a canonical isomorphism, such that $(f^*Q',(f\times f)^* q')$ is isomorphic to $(Q,q)$. 
\end{enumerate}
\end{proposition}
\begin{proof}
Without the $\Z_2$-graded condition the result is stated in Cao--Gross--Joyce {\cite[Proposition 3.5]{CGJ}}. Then as $\textnormal{deg}(P)$ respectively $\textnormal{deg}(Q)$ can be viewed as additive maps $\pi_0(X)\to \Z_2$ and $\pi_0(Y)$ is a group-completion, there exist unique extensions of the grading.
\end{proof}
We often suppress the maps $\mu_X$ and $e_X$ for an H-space $X$, we also write $Q$ instead of $(Q,q)$ for a strong H-principal $\Z_2$-bundle when $q$ is understood. 
\begin{lemma}
\label{lemma Z2graded}
Let $O_1,O_2\to X$ be $\mathbb{Z}_2$-graded strong (resp. weak) H-principal $\mathbb{Z}_2$-bundles. Then $O_1\otimes_{\mathbb{Z}_2} O_2$ is a $\mathbb{Z}_2$-graded strong (resp. weak) H-principal $\mathbb{Z}_2$-bundle. 
\end{lemma}
\begin{proof}
Let $q_i:O_i\boxtimes O_i\to \mu_X^*(O_i)$ be the isomorphisms from Definition \ref{definition weakstrong Hprincipal}. Then we define 
\begin{align*}
    q: &(O_1\otimes_{\mathbb{Z}_2}O_2)\boxtimes_{\mathbb{Z}_2}(O_1\otimes_{\mathbb{Z}_2}O_2)\stackrel{\textnormal{Def \ref{definition weakstrong Hprincipal}}}{\cong} (O_1\boxtimes_{\mathbb{Z}_2}O_1)\otimes_{\mathbb{Z}_2}(O_2\boxtimes_{\mathbb{Z}_2} O_{2})\\ &\stackrel{p_1\otimes p_2}{\cong}\mu_X^*(O_1)\otimes_{\mathbb{Z}_2}\mu_X^*(O_2)\cong \mu^*(O_1\otimes_{\mathbb{Z}_2}O_2)\,.
\end{align*}
Notice that we get an extra sign $(-1)^{\textnormal{deg}(\pi_1^*(O_2))\textnormal{deg}(\pi_2^*(O_1))}$. To check associativity \ref{definition weakstrong Hprincipal}, we need commutativity of
\begin{equation*}
\begin{tikzcd}[row sep=2.8cm]
   \arrow[d,"(-1)^{\textnormal{deg}(\pi_1^*(O_2)\textnormal{deg}(\pi_2^*(O_1))}"'] (O_1\otimes O_2)\boxtimes   (O_1\otimes O_2)\boxtimes   (O_1\otimes O_2)\arrow[r,"\begin{subarray} a (-1)^{\textnormal{deg}(\pi_2^*(O_2))\textnormal{deg}(\pi_3^*(O_1))}\\\,\\\,\end{subarray}"]&(O_1\otimes O_2)\boxtimes \mu_X^*(O_1\otimes O_2)\arrow[d,"\begin{subarray}a(-1)^{\textnormal{deg}\big(\mu_X^*(\pi_2^*(O_1)\boxtimes \pi_3*(O_1)\big)
    \textnormal{deg}\big(\pi_1^*(O_2)\big)}\\  \,\\\,\end{subarray}"]\\
    \mu_X^*(O_1\otimes O_2)\boxtimes_{\mathbb{Z}_2} (O_1\otimes O_2)\arrow[r,"\begin{subarray}a (-1)^{\begin{subarray}a \textnormal{deg}\big(\mu_X^*(\pi_1^*(O_2)\\
   \otimes\pi_2^*(O_2)\big)\textnormal{deg}\big(\pi_2^*(O_1)\big)\\\,\\\end{subarray}}\\\,\\\end{subarray}"]&\begin{subarray}a (\mu_X\times \textnormal{id}_X)^*\circ \mu^*_X(O_1\otimes O_2)\cong\\
   \circ(\textnormal{id}_X\times \mu_X)^*\circ \mu_X^*(O_1\otimes O_2)\,.
   \end{subarray}
\end{tikzcd}
\end{equation*}
Without the extra signs, it would be commutative because $O_i$ are strong H-principal. To check the signs note that going  down and right, resp. right and down we get
\begin{align*}
&(-1)^{\textnormal{deg}(\pi_2^*(O_2)\textnormal{deg}(\pi_3^*)(O_1)+(\textnormal{deg}\pi_2^*(O_1)\textnormal{deg}+(\pi_3^*)(O_1))\textnormal{deg}\big(\pi_1^*(O_1)\big)} \\&
= (-1)^{\textnormal{deg}(\textnormal{deg}(\pi_2^*(O_2)\textnormal{deg}(\pi_3^*)(O_1)+(\textnormal{deg}\pi_1^*(O_2)\textnormal{deg}+(\pi_2^*)(O_1))\textnormal{deg}\big(\pi_1^*(O_2)\big)}\,.
\end{align*}
\end{proof}
With the $\Z_2$-grading we need to distinguish between duals of strong H-principal $\Z_2$-bundles. 
\begin{definition}
\label{definitiondual}
 Let $(O,p)$ be a strong H-principal $\mathbb{Z}_2$-graded $\mathbb{Z}_2$-bundle. Its \textit{dual} $(O^*,p^*)$  will be defined to be a strong H-principal $\mathbb{Z}_2$-graded $\mathbb{Z}_2$-bundle, such that as $\mathbb{Z}_2$-bundles $O^* = O$ and  the isomorphism
$$
p^*: O^*\boxtimes_{\mathbb{Z}_2}O^*\xrightarrow{\sim}\mu^*_X(O^*)\,,
$$
is given by $p^*=(-1)^{\textnormal{deg}(\pi_1^*(O))\textnormal{deg}(\pi_2^*(O))}p$, where $\pi_1,\pi_2$ are the projections $X\times X\to X$.

When $O$ is a $\Z_2$-graded weak H-principal $\Z_2$-bundle, then $O^*$ will just mean $O$ canonically.
\end{definition}

\begin{example}
\label{z2gradedbundles}
An example of an H-space is the topological space $(\mathcal{B}_X)^{\textnormal{cla}}$, where the multiplication $\mu_{\mathcal{B}_X}: \mathcal{B}_X\times \mathcal{B}_X\to \mathcal{B}_X$ is given by mapping
$
([\nabla_P],[\nabla_Q]\mapsto [\nabla_P\oplus \nabla_Q]),
$
and we take $(\mu_{\mathcal{B}_X})^{\textnormal{cla}}: (\mathcal{B}_X)^{\textnormal{cla}}\times (\mathcal{B}_X)^{\cla}\to (\mathcal{B}_X)^{\cla}$. 
It is $\mathbb{Z}_2$-graded (see  \autocite{Markus, Zinger} for the corresponding grading of real determinant line bundles) in the following sense: Let $[\nabla_P]\in \mathcal{B}_X$, then
\begin{equation}
\label{degree of OD}
    \textnormal{deg}(O^{\slashed{D}_+})([\nabla_P]) = \chi^{\slashed{D}}(E,E)\,,
\end{equation}
where $E$ is the $\C^n$ vector bundle associated to the $U(n)$-bundle $P$ and $\chi^{\slashed{D}}(E,E)=\textnormal{ind}\big(\slashed{D}^{\nabla_{\textnormal{End}(E)}}\big)$ is the complex index from Definition \ref{def grading real}.
\end{example}
\begin{example}
\label{algebraicbundles}
Let $X$ be a quasi-projective Calabi--Yau 4-fold. Let $\mu_{\mathcal{M}_X}:\mathcal{M}_X\times\mathcal{M}_X\to \mathcal{M}_X$ be the map corresponding to taking sums of perfect complexes. Then a natural isomorphism
$$
\phi^{\omega}: O^\omega\boxtimes _{\Z_2}O^{\omega}\to \mu_{\mathcal{M}_X}^*(O^\omega)
$$
was constructed in \autocite[Definition 3.12]{CGJ} for projective $X$. This also works when $X$ is quasi-projective by restricting the isomorphism $\Lambda_{\underline{0}}\xrightarrow{\sim}\Lambda^*_{\underline{k}}$ for any compactification $Y$ along $\xi_Y$ from \eqref{xiY}. Therefore if $O^{\omega}$ is trivializable, then the pair
$
(O^{\omega})^{\textnormal{top}}\to (\mathcal{M}_X)^{\textnormal{top}}
$
together with $(\phi^{\omega})^{\textnormal{top}}$ gives a strong H-principal $\Z_2$-bundle.

We also construct in Proposition \ref{proposition Ovartheta Hprincipal} the isomorphism
$
\phi^{\vartheta_{\bowtie}}: O^{\bowtie}\boxtimes _{\Z_2}O^{\bowtie}\to \mu_{\mathcal{M}_{Y,D}}^*(O^{\bowtie}),
$
where 
$$\mu_{\mathcal{M}_{Y,D}}: \mathcal{M}_{Y,D}\times \mathcal{M}_{Y,D}\to \mathcal{M}_{Y,D}$$
corresponds to summation on both components of $\mathcal{M}_{Y,D}$.
\end{example}

We will need the following formulation of \autocite[Theorem 1.11]{CGJ}:
\begin{theorem}[Cao--Gross--Joyce {\cite[Theorem 1.11]{CGJ}, \cite[Theorem 12.6 (b)]{JU}}]
\label{Cao}
Let $X$ be a compact spin manifold of dimension 8, then the $\mathbb{Z}_2$-bundle $O^{\slashed{D}_+}\to \mathcal{B}_X$ is $\Z_2$-graded weak H-principal $\Z_2$-bundles. If $X$ satisfies Assumption \ref{ass:astcs}, it is strong.
\end{theorem}

\subsection{Complex excision}
\label{section pseudo}
Pseudo-differential operators over $\mathbb{R}^n$ are explained in Hörmander \autocite{Hor}. For background on pseudo-differential operators on manifolds, we recommend Lawson--Michelson \autocite[§3.3]{LawMic}, Atiyah--Singer \autocite[§5]{AS1}, Donaldson--Kronheimer \autocite[7.1.1]{DK}, and Upmeier \autocite[Appendix A]{Markus}. We will not review the definition due to its highly analytic nature, as we do not use it explicitly. The excision principle for differential operators was initiated by Seeley \cite{Seeley} and used by Atiyah--Singer \cite{AS1}.
Its refinement to excision for $\Z_2$-bundles was  applied by  Donaldson \cite{Donaldson}, Donaldson--Kronheimer \cite{DK} and categorified by Upmeier \cite{Markus}. We use these ideas and extend them to complex determinant line bundles. 

 From now on, we will be assuming that all real bundles come with a choice of a metric and all  complex vector bundles  with a choice of a hermitian metric. Note that the spaces of metrics are convex and therefore contractible. When we use convex, we automatically mean non-empty.

 Let $X$ be a manifold, $E,F\to X$ complex vector bundles, $P: \Gamma_{\textnormal{cs}}^\infty(E)\to \Gamma^{\infty}(F)$ pseudo-differential operator of degree $m$, then its symbol $\sigma(P): \pi^*(E)\to \pi^*(F)$, where $\pi:T^*X\to X$ is the projection map, is a homogeneous of degree $m$ on each fiber of $T^*X$ linear homomorphism. One says that $P$ is \textit{elliptic}, when its symbol $\sigma(P)$ is an isomorphism outside of the zero section $X\subset T^*X$.

 We will be working with continuous families of symbols and pseudo-differential operators as defined in \autocite[p. 491]{AS1} or as in Upmeier \autocite[Appendix]{Markus}. For a topological space $M$, we denote the corresponding set of elliptic pseudo-differential $M$-families by $\Psi_m(E,F;M)$ and the elliptic symbol $M$-families by $S_m(E,F;M)$ with the map
\begin{equation}
\label{symbol map}
\sigma: \Psi_m(E,F;M)\to S_m(E,F;M)\,.
\end{equation}
It is compatible with respect to addition, scalar multiplication, composition and taking duals (see \cite[§5]{AS1} for details).
It is standard to restrict to degree 0 operators and symbols using 
\begin{equation*}
    \begin{tikzcd}
    P\dar[maps to,"\Psi_0"]\rar["\sigma"]&\sigma(P)\dar[mapsto,"S_0"]\\
    (1+PP^*)^{-\frac{1}{2}}P\rar["\sigma"] &(\sigma(P)\sigma(P)^*)^{-\frac{1}{2}}\sigma(P)
    \end{tikzcd}\,.
\end{equation*}
If $X$ is compact then each $P\in \Psi_m(E,F;M)$ gives an $M$-family of Fredholm operators between Hilbert spaces containing $\Gamma^{\infty}_{\textnormal{cs}}(E)$ and $\Gamma^{\infty}(E)$ such that $\textnormal{ker}(P)$ and $\textnormal{coker}(P)$ lie in $\Gamma^{\infty}(E_0)$ and $\Gamma^{\infty}(E_1)$ respectively.  

Let $P$ be a continuous $Y$-family of Fredholm operators $P_y: H_0\to H_1$ for each $y$, where $H_i$ are Hilbert spaces. Determinant line bundle $\textnormal{det}(P)\to Y$ of $P$ is defined in Zinger \autocite{Zinger} using stabilization (in this case one only needs the $H_i$ to be Banach spaces) and in Upmeier \autocite[Definition 3.4]{Markus}, Freed \autocite{Freed} or Quillen \autocite{Quillen}. We will use the conventions from \autocite[Definition 3.4]{Markus}. 

\begin{definition}[Phillips \autocite{phillips}]
\label{definition mu}
Let $P:H\to H$ be a self adjoint Fredholm operator on the Hilbert space $H$. The \textit{essential spectrum} $\textnormal{spec}_{\textnormal{ess}}(P)$ is the set $\lambda\in \mathbb{R}$, such that $P-\lambda \textnormal{Id}$ is not Fredholm. We denote by $\textnormal{spec}(P)$ the spectrum of $P$. For each $\mu >0$, such that $\pm\mu\notin \textnormal{spec}(P)$ and $(-\mu,\mu)\cap \textnormal{spec}_{\textnormal{ess}}(P)=\emptyset$, one defines $V_{(-\mu,\mu)}(P)\subset H$  as the subspace of eigenspaces of $P$ for eigenvalues $-\mu<\lambda<\mu$. If $P$ is positive semi-definite, we will also write $V_{[0,\mu)}(P)$\,. If $P$ is skew adjoint, we will also denote the set of its eigenvalues by $\textnormal{spec}(P)$ (note that $\textnormal{spec}(P) = i\textnormal{spec}(-iP)$).
\end{definition}

For a $Y$ family of self adjoint Fredholm operators, one can choose $\mathfrak{U} \subset Y$ sufficiently small and $\mu$ from Definition \ref{definition mu}, such that $V_{(-\mu,\mu)}(P)$ becomes a vector bundle on $\mathfrak{U}$. This can be used to define topology on the union of determinant lines
$$
\textnormal{det}(P_y) = \textnormal{det}(P_y)\otimes \textnormal{det}(P_y^*)^*\,,y\in Y
$$
as in \autocite[Definition 3.4]{Markus}. 

\begin{definition}
\label{def grading real}
The bundle $\textnormal{det}(P)$ is $\mathbb{Z}_2$-graded with degree $\textnormal{ind}(P)$, where $\textnormal{ind}(P) = \dim(\textnormal{Ker}(P_y))-\dim(\textnormal{Ker}(P_y^*))=\textnormal{ind}(P_y)$.  If we have two $Y$-families $P_1$ and $P_2$, then the isomorphism 
\begin{equation}
    \label{sum isomorphism}
    \textnormal{det}(P_1)\otimes \textnormal{det}(P_2)\cong \textnormal{det}(P_2)\otimes\textnormal{det}(P_1)
\end{equation}
differs from the naive one by the sign $(-1)^{\textnormal{ind}(P_1)\textnormal{ind}(P_2)}$. 
\end{definition} 

We have the ``inverse" of \eqref{symbol map}
\begin{equation*}
    \begin{tikzcd}
 S_0(E,F,M) \ni   p\longmapsto P\in S_0\big(E,F,M\times \sigma^{-1}(p)\big)
    \end{tikzcd}
\end{equation*}
Which we use to abuse the notation
\begin{equation*}
    \begin{tikzcd}
    \textnormal{det}(p)\arrow[d, ""{name = A}]&\textnormal{det}(P_0)\arrow[d, ""{name = B}]&\textnormal{det}(P)\arrow[d]\\
    Y& *\times Y\arrow[r,"i_{P_0}\times \textnormal{id}"]&\sigma^{-1}(p)\times Y
    \arrow[equal,shorten >=2mm, from=A, to=B]
    \end{tikzcd}
\end{equation*}
Here $\textnormal{det}(P_0) = (i_{P_0}\times\textnormal{id})^*\textnormal{det}(P)$. Note that as $\sigma^{-1}(p)$ is convex (\cite[Theorem 4.6]{Markus}), for two different choices $P_0,P_1\in \sigma^{-1}(P)$ we have natural isomorphisms $\textnormal{det}(P_0)\cong \textnormal{det}(P_1)$. Therefore
\begin{lemma}
The complex line bundle $\textnormal{det}(p)$ is well-defined up to natural choices of isomorphisms.
\end{lemma}
The following lemma is meant for book-keeping purposes.
\begin{lemma} 
\label{proposition functoriality ds adj}
Let $p_{i} \in S_m(E_i,F_i;M)$ for $i=1,2$ and $q\in S_m(E,F,Y\times I)$.
\begin{enumerate}[label=(\roman*)]
\item (\textit{Functoriality.}) If $\mu_E: E_1\to E_2$, $\mu_F: F_1\to F_2$ are isomorphisms such that 
\begin{equation}
\label{functoriality}
\begin{tikzcd}
  \arrow[d,"\pi^*(\mu_E)"]  \pi^*(E_1)\arrow[r,"p_1"]& \pi^*(F_1)\arrow[d,"\pi^*(\mu_F)"]\\
    \pi^*(E_2)\pi^*\arrow[r,"p_2"]& \pi^*(F_2)
    \end{tikzcd}
\end{equation}
commutes, then there is a natural isomorphism $\textnormal{det}(p_1)\to \textnormal{det}(p_2)$\,.
\item (\textit{Direct sums}.) There is a natural isomorphism 
\begin{equation}
\label{direct sums}
  \textnormal{det}(p_1\oplus p_2)\longrightarrow \textnormal{det}(p_1)\textnormal{det}(p_2)\,.  
\end{equation}
\item (\textit{Adjoints.}) There is a natural isomorphism 
\begin{equation}
\label{dual}
 \textnormal{det}((p_1)^*)\longrightarrow \textnormal{det}^*(p_1)\,.  
\end{equation}
\item (Triviality.) If $p_1 = \pi^*(\mu)$ for some isomorphism $\mu: E_1\to F_1$, then there is a natural isomorphism
\begin{equation}
\label{triviality}
    \textnormal{det}(p^+)\longrightarrow \C\,.
\end{equation}
\item (Transport.) There is a natural isomorphism $\textnormal{det}(q)|_{Y\times \{0\}}\cong \textnormal{det}(q)|_{Y\times \{1\}}$. such that for 
  $q_i\in S_m(E_i,F_i,Y\times I)$ we have the commutative diagram
   $$
   \begin{tikzcd}
     \arrow[d,"(ii)"] \textnormal{det}(q_1\oplus q_2)|_{Y\times \{0\}}\arrow[r,"(v)"]& \textnormal{det}(q_1\oplus q_2)|_{Y\times \{1\}}\arrow[d,"(ii)"]\\
     \textnormal{det}(q_1)|_{Y\times \{0\}}\otimes \textnormal{det}(q_2)|_{Y\times \{0\}}\arrow[r,"(v)\otimes (v)"]& \textnormal{det}(q_1)|_{Y\times \{1\}}\otimes \textnormal{det}(q_2)|_{Y\times \{1\}}
   \end{tikzcd}
   $$
   \end{enumerate}
\end{lemma}
\begin{proof}
For 
\begin{enumerate}[label=(\roman*)]
    \item make a natural choice of a pair $(P_1,P_2)\in\sigma^{-1}(p_1)\times \sigma^{-1}(p_2)$ commuting with $\mu_E, \mu_F$ and apply \autocite[Proposition 3.5 (i)]{Markus}.
    \item make a natural choice of any $P_1\times P_2\in \sigma^{-1}(p_1)\times \sigma^{-1}(p_2)$ in loc cit.
    \item make a natural choice $P_1\in \sigma^{-1}(p_1)$ in loc cit.
    \item make the choice $P_1=\mu$ in loc cit.
    \item make a contractible choice of $Q_0\in \sigma^{-1}(q)$, then we have natural isomorphism $\tau$ such that $\tau_t:\textnormal{det}(Q_0)|_{Y\times \{0\}}\cong \textnormal{det}(Q_0)|_{Y\times \{t\}}$ for all $t\in I$ and $\tau_{0}=\textnormal{id}$, then consider the one for $t=1$. The commutativity of the diagram follows immediately from the definition. 
\end{enumerate}  
\end{proof}
The following definition is the main reason, why we introduced the above concepts.

\begin{definition}
\label{definition excision I family}
Let $E_i$, $F_i$ be vector bundles on compact manifold $X$ and $p_i\in S_0(E_i,F_i; Y)$. Let $U,V\subset X$ be open, $U\cup V = X$ and $\mu_E:E_1|_U\to E_2|_U$ , $\mu_F:F_1|_U\to F_2|_U$ isomorphisms, such that
\begin{equation}
\label{commutativity}
    \begin{tikzcd}
    \pi^*(E_1|_U)\arrow[d,"\pi^*(\mu_E)"]\arrow[r,"p_1|_{T^*U}"]&\arrow[d,"\pi^*(\mu_F)"]\pi^*(F_1|_U)\\
    \pi^*(E_2|_U)\arrow[r,"p_2|_{T^*U}"]&\pi^*(F_2|_U)
    \end{tikzcd}
\end{equation}
commutes. 
Choose a function $\chi\in C^\infty_{\textnormal{cpt}}(V,[0,1])$ with $\chi|_{X\backslash U} = 1$. Then we obtain that:
\begin{equation}
\label{excision I family}
   t\in I\longmapsto  (p_1,p_2,\mu_E,\mu_F)_t^\chi = 
\begin{pmatrix}
(1-t+t\chi)p_1& t(1-\chi)\pi^*\mu_F^*\\
t(1-\chi)\pi^*\mu_E& -(1-t+t\chi)(p_2)^*
\end{pmatrix} 
\end{equation}
is elliptic.
\end{definition}
The following result might appear deceptively obvious, but the usual $I^2$-family argument does not go through.

\begin{lemma}
\label{proposition global excision isomorphism}
    Let $X$ be compact, $E_i,F_i$, complex vector bundles on $X$ and  $p_i\in S_0(E_i,F_i;Y)$ with isomorphism 
  $
  \mu_E:E_1 \to E_2$, $\mu_F:F_1\to F_2\,,$
satisfying \eqref{commutativity} on $X$ then we have the commutativity up to contractible choices
\begin{equation*}
\begin{tikzcd}[column sep=huge, row sep=huge]
   \arrow[d,"\textnormal{Prop. \ref{proposition functoriality ds adj} (iii) + (i)}" ] \textnormal{det}(p_1)\textnormal{det}(p^*_2)\arrow[r," \textnormal{Prop. \ref{proposition functoriality ds adj} (ii)}"]&\textnormal{det}\big((p_1,p_2,\mu_E,\mu_F)^0_0\big)\arrow[d,"\textnormal{Prop. \ref{proposition functoriality ds adj} (v)}"]\\
   \mathbb{C}&\arrow[l,"\textnormal{Prop. \ref{proposition functoriality ds adj} (iv)}"] \textnormal{det}\big((p_1,p_2,\mu_E,\mu_F)^0_1\big)
\end{tikzcd}
\end{equation*}
\end{lemma}
\begin{proof}

Choose $(P_1,P_2)\in\sigma^{-1}(p_1)\times \sigma^{-1}(p_2)$ commuting with $\mu_E,\mu_F$  and construct 
$$
\Psi_t = \begin{pmatrix}
(1-t)P_1& t\mu_F^*\\
t\mu_E&-(1-t)P_2^*
\end{pmatrix}\in\sigma^{-1}\big((p_1,p_2,\mu_E,\mu_F)^0_t\big)\,.
$$
By composing $\Psi_t$ with $\SmallMatrix{0&\mu_E^{-1}\\(\mu_F^*)^{-1}&0}
$ we obtain
$$\tilde{\Psi}_t = \begin{pmatrix}
t\,\textnormal{id}&-(1-t)P^*\\
(1-t)P&t\,\textnormal{id}
\end{pmatrix}: E_1\oplus F_2\longrightarrow E_1\oplus F_2
\,.$$
Let $\nu\in \mathbb{R}^{>0}$ and $\mathfrak{U}\subset Y$ be chosen sufficiently small as in Upmeier \autocite[Definition 3.4]{Markus}, such that 
$
V_{[0,\nu)}(\tilde{\Psi}^*_0\tilde{\Psi}_0)
$
is a vector bundle.

Notice that $\tilde{\Psi}^*_t\tilde{\Psi}_t = \tilde{\Psi}_t\tilde{\Psi}^*_t$. Moreover, by spectral theorem each non-zero eigenvalue $\lambda^2\in(0,\nu)$ of $\tilde{\Psi}_0^*\tilde{\Psi}_0$ has multiplicity $2k$ for some positive integer $k$ and then $\tilde{\Psi}_0$ has eigenvalues $i\lambda$, $-i\lambda$ each of multiplicity $k$ in its set of eigenvalues $\textnormal{spec}(\tilde{\Psi}_0)$. The eigenvectors of $\tilde{\Psi}_t^*\tilde{\Psi}_t$ remain the same, but corresponding eigenvalues are $\lambda^2(1-t)^2 + t^2$. We therefore define 
$\nu(t) = \nu(1-t)^2 + t^2$
 and we have a natural isomorphism 
 \begin{equation}
 \label{identity on Vmu along t}
     V_{[0,\nu)}(\tilde{\Psi}^*_0\tilde{\Psi}_0)\cong V_{[0,\nu(t))}(\tilde{\Psi}^*_t\Psi_t)
 \end{equation}
 given by the identity for all $t\in I$ (here one extends to $t=1$ by considering the same finite set of eigenvectors which now have eigenvalue 1), which gives a continuous isomorphism of vector bundles on $\mathfrak{U}\times I$ and restricts to identity for $t=0$. The isomorphisms of determinant line bundles is then given by 
\begin{align*}
  \alpha_\nu(t):\,&\textnormal{det}(\Psi_0)\stackrel{\textnormal{\autocite[Def. 3.4]{Markus}}}{\cong} \textnormal{det}(V_{[0,\nu)}(\Psi_0^*\Psi_0))\textnormal{det}^*(V_{[0,\nu)}(\Psi_0\Psi^*_0))\\&\stackrel{\textnormal{\eqref{identity on Vmu along t}}}{\cong} \textnormal{det}(V_{[0,\nu(t))}(\Psi_t^*\Psi_t))\textnormal{det}^*(V_{[0,\nu(t))}(\Psi_t\Psi^*_t))\stackrel{\textnormal{\autocite[Def. 3.4]{Markus}}}{\cong} \textnormal{det}(\Psi_t)\,.  
\end{align*}
We see that this is a representative of the transport Prop. \ref{proposition functoriality ds adj} (v) because it restricts to identity at $t=1$. To see that this isomorphism is independent of $\nu$, we can restrict to a single point $y\in Y$. Let $\nu'>\nu >0$, then for $\Psi_0(y)$ choose its diagonalization when restricted to $V_{[0,\nu')}(\Psi_0^*\Psi_0)$. From looking at \autocite[Definition 3.4]{Markus} it is then easy to see that 
$$
\alpha_{\nu'}(t) = \prod_{\begin{subarray}\ \mu\in \textnormal{spec}(\Psi_0)\\
\nu<|\mu|^2<\nu'\end{subarray}}\frac{(1-t) + \mu^{-1}t}{[(1-t)^2+|\mu|^{-2}t^2]^{\frac{1}{2}}}\alpha_\nu(t)\,.
$$
As each $\mu=i\lambda$ comes with its conjugate of the same multiplicity, the factor is equal to one. Let $\alpha': \textnormal{det}(\Psi_0)\cong \textnormal{det}(P)\textnormal{det}(P^*)\cong \mathbb{C}$ be isomorphism combining \eqref{functoriality},\eqref{direct sums} and \eqref{dual},  then it can be checked in the same way that
$$
\alpha_\nu(1) = \prod_{\begin{subarray} \ \mu\in\textnormal{Spec}(\Psi_0)\\
0<|\mu|<\nu\end{subarray}}\frac{|\mu|^2}{\mu}\alpha'\,,
$$
where the factor again becomes one. By covering $Y$ by such sets $\mathfrak{U}_i$ and choosing appropriate $\nu_i$, we can glue the isomorphisms on $\mathfrak{U}_i\times I$, because they coincide on the overlaps $(\mathfrak{U}_i\cap \mathfrak{U}_j)\times I$. Composing 
$\alpha(t): \textnormal{det}(\tilde{\Psi}_0)\to \textnormal{det}(\tilde{\Psi}_t)$
 with 
 $\SmallMatrix{0&\mu_E^{-1}\\(\mu_F^*)^{-1}&0}
$, we obtain Prop. \ref{proposition functoriality ds adj} (v) and the commutativity of the diagram.
\end{proof}
\begin{remark}
\label{excision}
Note that when $p_i\in S_{0}(E_i,F_i;Y)$ have a real structure and $\mu_E,\mu_F$ preserve it, then there exists a natural $\Z_2$-bundle $\textnormal{or}(p_1,p_2,\mu_E,\mu_F)^{\chi}\subset \textnormal{det}(p_1,p_2,\mu_E,\mu_F)^{\chi}$ as in Donaldson--Kronheimer \cite[§7.1.1]{DK} or Upmeier \cite{Markus} The transport isomorphism of Proposition \ref{proposition functoriality ds adj} (v) for the $Y\times I$ family along $\textnormal{or}(p_1,p_2,\mu_E,\mu_F)$ is canonical, because it is the standard transport along fibers $\Z_2$.
\end{remark}
Our main object of study are going to be twisted Dirac operators and Dolbeault operators. Let $X$ be a manifold, $P$ a $U(n)$-principal bundle, $V_n$ a representation of $U(n)$ and $E$ the associated vector bundle, then for a given connection $\nabla_P$ on $P$ and its associated connection $\nabla_E$, the twisted operator $D^{\nabla_{E}}$ has the degree 0 symbol $$S_0\big(\sigma(D)\big)\otimes \textnormal{id}_{\pi^*\big(E\big)}=:\sigma_{E}(D)\,.$$ 
If $\Phi:V\to W$ is an isomorphism of vector bundles, we will also write $\Phi=\textnormal{id}\otimes\Phi:E\otimes V\to E\otimes W$.\\

Let us now formulate the excision isomorphism for complex operators in the form we will need in \ref{section comon resolution}. This generalizes \cite[Thm. 2.10]{Markus} to complex determinant line bundles. Moreover, for real operators it is slightly more general then \cite[Thm. 2.13]{Markus} in that, we do not require a framing of bundles, but isomorphisms in \cite[Thm. 2.13(b)]{Markus}. This would already follow from \cite[Thm. 2.10]{Markus}, but we obtain it as a consequence of Remark \ref{excision}. Note that we also do not require the isomorphisms below to be unitary, as this is not necessary for the operators in \eqref{excision I family} to be elliptic. 

\begin{definition}
\label{definition excision I2 family}
    Let $X_i$ be compact, $E_i, F_i$,  vector bundles on $X_i$ for $i=1,2$ and $D_{i}: \Gamma^{\infty}(E_i)\to \Gamma^{\infty}(F_i)$ complex/real elliptic differential operators. Moreover, let $S_i,T_i\subset X_i$ open, such that $S_i\cup T_i=X_i$ and $I_S:S_1\to S_2$ an isomorphism. We then denote by
\begin{equation}
\label{excision data}
   \begin{tikzcd}
   \arrow[d,dashrightarrow,"\xi_{V}"']\sigma_{V_1}(D_1)\arrow[r,dashrightarrow,"\Phi_1"]&\sigma_{W_1}(D_1)\arrow[d,dashrightarrow,"\xi_W"']\\
   \sigma_{V_2}(D_2)\arrow[r,dashrightarrow,"\Phi_2"]&\sigma_{W_2}(D_2)
   \end{tikzcd} 
\end{equation}
    the collection of isomorphisms $\Phi_i:V_i|_{T_i}\xrightarrow{\sim} W_i|_{T_i}$,  $\xi_{V}:I_V^*(V_2|_{S_2})\xrightarrow{\sim} V_1|_{S_1}$, $\xi_W:I^*_S(W_2|_{S_2})\xrightarrow{\sim} W_1|_{S_1}$ satisfying $\xi_W\circ \Phi_1=I^*_S(\Phi_2)\circ \xi_V$ for families of vector bundles $V_i,W_i\to X_i$.
  \end{definition}
  \begin{lemma}
  \label{upmeier}
For the data given by \eqref{excision data} and a compact subsets $K_i$, s.t. $X_i\backslash K_i\subset T_i$ are identified by $I_S$,we have natural isomorphisms in families
$$
\Xi(D_i,\xi_{V/W},\Phi_i):\textnormal{det}\big(\sigma_{V_1}(D_1)\big)\textnormal{det}^*\big(\sigma_{W_1}(D_1)\big)\xrightarrow{\sim}\textnormal{det}\big(\sigma_{V_2}(D_2)\big)\textnormal{det}^*\big(\sigma_{W_2}(D_2)\big)\,,
$$
such that for another set of data
$$
  \begin{tikzcd}
   \arrow[d,dashrightarrow,"\xi_{\mathcal{V}'}"']\sigma_{V'_1}(D_1)\arrow[r,dashrightarrow,"\Phi'_1"]&\sigma_{W'_1}(D_1)\arrow[d,dashrightarrow,"\xi_{W'}"]\\
   \sigma_{V'_2}(D_2)\arrow[r,dashrightarrow,"\Phi'_2"']&\sigma_{W'_2}(D_2)
   \end{tikzcd} 
$$
for the same $S_i,T_i,K$ the diagram is commutative up to natural isotopies
\begin{equation*}
    \begin{tikzcd}[column sep=large, every label/.append
style={font=\tiny}]
    \arrow[d, "\begin{subarray}a \textnormal{\eqref{direct sums}}\\
\textnormal{\eqref{dual}}\end{subarray}"]\textnormal{det}\big(\sigma_{V_1\oplus V_1'}(D_1)\big)\textnormal{det}^*\big(\sigma_{W_1\oplus W_1'}(D_1)\big)
\arrow[r,"\begin{subarray}a {\Xi(D_i,\xi_{V\oplus V'/W\oplus W'},\Phi\oplus\Phi')}\\\,
\\\,\\
\,
\end{subarray}"]&  \arrow[d,"\begin{subarray}a \textnormal{\eqref{direct sums}}\\
\textnormal{\eqref{dual}}\end{subarray}"]
\textnormal{det}\big(\sigma_{V_2\oplus V_2'}(D_2)\big)\textnormal{det}^*\big(\sigma_{W_2\oplus W_2'}(D_2)\big)\\
\begin{subarray}a \textnormal{det}\big(\sigma_{V_1}(D_1)\big)\textnormal{det}\big(\sigma_{V_1'}(D_1)\big)\\
\textnormal{det}^*\big(\sigma_{W'_1}(D_1)\big)\textnormal{det}^*\big(\sigma_{W_1}(D_1)\big)\end{subarray}\arrow[r,"\begin{subarray}a{\Xi(D_i,\xi_{V/W},\Phi)}\\{\Xi(D_i,\xi_{V'/W'},\Phi')}\end{subarray}"]&\begin{subarray}a \textnormal{det}\big(\sigma_{V_2}(D_2)\big)\textnormal{det}\big(\sigma_{V_2'}(D_2)\big)\\\textnormal{det}^*\big(\sigma_{W'_2}(D_2)\big)\textnormal{det}^*\big(\sigma_{W_2}(D_2)\big)\end{subarray}
    \end{tikzcd}
\end{equation*}
Moreover, if $S_i=X_i$, then $\Xi(D_i,\xi_{V/W},\Phi_i) = (\eqref{functoriality})^{-1}\circ(\eqref{functoriality})$.  
  \end{lemma}
  \begin{proof}
  The following is standard, and we simply lift it to complex determinant line bundles. Making a contractible choice of $\chi_i\in C^{\infty}_{\textnormal{cs}}(S_i)$, $\chi_i|_{K_i}=1$ identified under $I_V$, the composition of the following isomorphisms gives $\Xi(D_i,\xi_{V/W},\Phi_i)$: 
\begin{align*}
  \textnormal{det}\big(\sigma_{V_1})\otimes \textnormal{det}(\sigma_{W_1}\big)^* &
 \stackrel{\ref{proposition functoriality ds adj}(ii),(i)}{\cong}\textnormal{det}\big((\sigma_{V_1}(D_1),\sigma_{W_1}(D_1),\Phi_1,\Phi_1)^{\chi_1}_0\big)\\
&\stackrel{\ref{proposition functoriality ds adj} (v)}{\cong}
\textnormal{det}\big((\sigma_{V_1}(D_1),\sigma_{W_1}(D_1),\Phi_1,\Phi_1)^{\chi_1}_1\big)\\
&\stackrel{*}{\cong}\textnormal{det}\big((\sigma_{V_2}(D_2),\sigma_{W_2}(D_2),\Phi_2,\Phi_2)^{\chi_2}_1\big)\\
&\stackrel{\ref{proposition functoriality ds adj}(v)}{\cong}\textnormal{det}\big((\sigma_{V_2}(D_2),\sigma_{W_2}(D_2),\Phi_2,\Phi_2)^{\chi_2}_0\big)\\
&\cong \textnormal{det}\big(\sigma_{V_2}(D_2))\textnormal{det}^*(\sigma_{W_2(D_1)}\big)\,,
\end{align*}
where for the step $*$, we are making a contractible choice of $P_i\in \sigma^{-1}\big((\sigma_{V_i}(D_i),\sigma_{W_i}(D_i),\Phi_i,\Phi_i)^{\chi}_1\big)$ supported representatives in $S_i$ of the two symbols on both sides as in Upmeier \cite[Thm. A.6]{Markus} identified by the isomorphism $\xi_{V'},\xi_{W'}$ and using that  
\begin{align*}
   \textnormal{ker}(P_i)\in \Gamma^{\infty}_{\textnormal{cs}}\big(S_i, (E_i\otimes V_i)\oplus (F_i\otimes W_i)\big)\,,\\
  \textnormal{coker}(P_i)\in \Gamma^{\infty}_{\textnormal{cs}}\big(S_i,(F_i\otimes V_i)\oplus (E_i\otimes W_i)\big) \,.
\end{align*}  
The second statement follows from the compatibility under direct sums in \ref{proposition functoriality ds adj}(v). The final statement is just Lemma \ref{proposition global excision isomorphism}.
  \end{proof}
\section{Proof of Theorem \ref{maintheorem}}
\label{sec:maintheoremproof}
 We construct here a double $\tilde{Y}$ for our manifold $X$, such that the ``compactly supported''  orientation on $X$ can be identified with the one on $\tilde{Y}$. We use homotopy theoretic group completion to reduce the problem to trivializing the orientation $\Z_2$-bundles on the moduli space of pairs of vector bundles generated by global sections identified on the normal crossing divisor. Then we express the isomorphism $\vartheta_{\bowtie}$ from Definition \ref{definitionbowtie} using purely vector bundles in §\ref{section comon resolution}. We then construct the isotopy between the two different real structures to obtain an isomorphism of $\Z_2$-bundles by hand. The final result of this section is contained in Proposition \ref{proposition comparing excision} and Proposition \ref{proposition main}.

\subsection{Relative framing on the double}
\label{section relative framing}
Here we construct the double of a non-compact $X$, such that it can be used in §\ref{sec:tubular-neighborhood} to define orientations back on moduli spaces over $X$.

\begin{definition}
\label{definition double}
Let $X$ be a non-compact spin manifold $\textnormal{dim}_{\mathbb{R}}(X) = n$. Choose a compact $U \subset X$ which is a manifold with boundary $\partial U$. Suppose that there is normal vector field $\nu$ on $\partial U$. Let $V$ be the tubular neighborhood of $\partial U$ in $X$, then it is diffeomorphic to $(-1,1)\times \partial U$ and is a collar. We define $\tilde{Y}:=U\cup_{\partial U}(-U)$, where $-U$ denotes a copy of $U$ with negative orientation. Then $\tilde{Y}$ admits a natural spin-structure which restricts to the original one on $U$ (see for example \cite[p. 193]{KC}). Since we do not need it here explicitly, we do not give its description. Lastly, let $T\subset U$ be open such that $\partial U\subset T$, then define $\tilde{T} =(\bar{T}\cap U)\cup (-U)$  (see Figure \ref{figure 1}).
\begin{figure}
\label{figure 1}
\includegraphics[width=8cm]{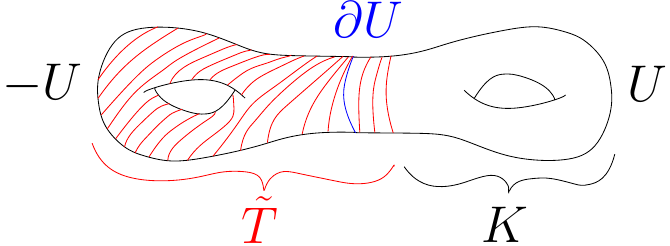}
\centering
\caption{Spin manifold $\tilde{Y}$ and the subset $\tilde{T}\subset \tilde{Y}$. The set $K$ is compact because $U$ was.}
\end{figure}
\end{definition}
\begin{remark}
   We will need a particular case of $U$ and $T$ which will be defined in Definition \ref{def:UandT-defined}.
\end{remark}
Let $P,Q\to \tilde{Y}$, be two $U(n)$, bundles, such that there exists an isomorphism  $P|_{\tilde{T}}\cong Q|_{\tilde{T}}$.
We define now the moduli stack of pairs of connections on principal bundles identified on $\tilde{T}$. 

\begin{definition}
\label{definition relative framing orienation}
Consider the space $\mathcal{A}_P\times \mathcal{A}_Q\times \mathcal{G}_{P,Q,\tilde{T}}$, where $\mathcal{G}_{P,Q,\tilde{T}}$ is the set of smooth isomorphisms $\tilde{\phi}: P|_{\tilde{T}}\to Q|_{\tilde{T}}$. Let $\mathcal{G}_P\times \mathcal{G}_Q$ be the product of gauge groups. We have a natural action 
\begin{align*}
  (\mathcal{G}_P\times\mathcal{G}_Q)\times (\mathcal{A}_P\times \mathcal{A}_Q\times\mathcal{G}_{P,Q,\tilde{T}} )\to  \mathcal{A}_P\times \mathcal{A}_Q\times \mathcal{G}_{P,Q,\tilde{T}}  \\
  (\gamma_P,\gamma_Q,\nabla_P,\nabla_Q,\tilde{\phi})\mapsto   (\gamma_P(\nabla_P),\gamma_Q(\nabla_Q),\gamma_Q\circ\tilde{\phi}\circ(\gamma_P)^{-1})\,.
\end{align*}
We denote the quotient stack by $\mathcal{B}_{P,Q,\tilde{T}} = [\mathcal{A}_P\times \mathcal{A}_Q\times \mathcal{G}_{P,Q,\tilde{T}} / \mathcal{G}_P\times \mathcal{G}_Q]$\,. Let us define the union
$$
\mathcal{B}_{\tilde{Y},\tilde{T}}= \bigcup_{\begin{subarray}
a[P],[Q]:\\
[P|_{\tilde{T}}]=[Q|_{\tilde{T}}]
\end{subarray}} \mathcal{B}_{P,Q,\tilde{T}}\,,
$$
where we chose representatives $P,Q$ for the isomorphism classes. 
\end{definition}
There exist natural maps $\mathcal{B}_{\tilde{Y}}\xleftarrow{p_1}\mathcal{B}_{\tilde{Y},\tilde{T}}\xrightarrow{p_2} \mathcal{B}_{\tilde{Y}}$ induced by $\mathcal{A}_P\times \mathcal{A}_Q\times \mathcal{G}_{P,Q,\tilde{T}}\to \mathcal{A}_Q$ and $\mathcal{A}_P\times \mathcal{A}_Q\times \mathcal{G}_{P,Q,\tilde{T}}\to \mathcal{A}_P$. Let $O^{\slashed{D}_+}\to \mathcal{B}_{\tilde{Y}}$ be the $\Z_2$-graded H-principal $\mathbb{Z}_2$-bundles from \eqref{Odiracbundle}, then we define 
\begin{equation}
\label{DO bundle}
    D_O(\tilde{Y}) = p_1^*(O^{\slashed{D}_+})\boxtimes_{\mathbb{Z}_2} p_2^*(O^{\slashed{D}_+})^*\,.
\end{equation}
Let us now construct an explicit representative $(\mathcal{B}_{\tilde{Y},\tilde{T}})^{\textnormal{cla}}$.

\begin{definition}
\label{representatinofproduct}
Let $P$ and $Q$ be $U(n)$-bundles on $\tilde{Y}$ isomorphic on $\tilde{T}$. Consider the following two quotient stacks
\begin{align*}
   \mathcal{P}_Q=[\mathcal{A}_P\times \mathcal{A}_Q\times \mathcal{G}_{P,Q,\tilde{T}}\times P /\mathcal{G}_P\times\mathcal{G}_Q]\,,\\
\mathcal{Q}_P=[\mathcal{A}_P\times \mathcal{A}_Q\times \mathcal{G}_{P,Q,\tilde{T}}\times Q /\mathcal{G}_P\times\mathcal{G}_Q]\,, 
\end{align*}
which are $U(n)$-bundles on $\tilde{Y}\times\mathcal{B}_{P,Q,\tilde{T}}$. We have a natural isomorphism $\tau_{P,Q}:\mathcal{P}_Q|_{\bar{T}\times\mathcal{B}_{P,Q,\tilde{T}}} \to \mathcal{Q}_{P}|_{\bar{T}\times\mathcal{B}_{P,Q,\tilde{T}}} $ given by
$
[\nabla_P,\nabla_Q,\tilde{\phi}, p]\mapsto [\nabla_P,\nabla_Q,\tilde{\phi},\tilde{\phi}(p)].
$
After taking appropriate unions, we obtain  bundles $\mathcal{P}_1,\mathcal{P}_2\to \tilde{Y}\times\mathcal{B}_{\tilde{Y},\tilde{T}}$ with an isomorphism  $\mathcal{P}_1|_{\tilde{T}\times\mathcal{B}_{\tilde{Y},\tilde{T}}}\cong \mathcal{P}_2|_{\tilde{T}\times\mathcal{B}_{\tilde{Y},\tilde{T}}}$. Pulling $\mathcal{P}_i$ back to $\tilde{Y}\times(\mathcal{B}_{\tilde{Y},\tilde{T}})^{\textnormal{cla}}$, we obtain $\mathcal{P}_i^{\textnormal{cla}}$ fiber bundles, which are  $U(n)$-bundles on each connected components for some $n\geq 0$. Together with the isomorphism $\tau^{\textnormal{cla}}$, these induce two maps 
$$
\mathfrak{p}_1,\mathfrak{p}_2:\tilde{Y}\times \mathcal{B}_{\tilde{Y},\tilde{T}}\longrightarrow \bigsqcup_{n\geq 0}BU(n)\,,
$$
with a unique (up to contractible choices) homotopy $H_{\mathfrak{p}}: \tilde{T}\times \mathcal{B}_{\tilde{Y},\tilde{T}}\times I\to \bigsqcup_{n\geq 0}BU(n)$ between $\mathfrak{p}_1$ and $\mathfrak{p}_2$ restricted to $\tilde{T}\times \mathcal{B}_{\tilde{Y},\tilde{T}}$. We obtain the following homotopy commutative diagram
\begin{equation}
\label{fibrantdiagram}
    \begin{tikzcd}
  (\mathcal{B}_{\tilde{Y}}\times_{\mathcal{B}_{\tilde{T}}}\mathcal{B}_{\tilde{Y}})^{\textnormal{cla}}\arrow[r]\arrow[d]&  \textnormal{Map}_{\mathcal{C}^0}(\tilde{Y},\bigsqcup_{n\geq 0} BU(n))\arrow[d]\\
  \textnormal{Map}_{\mathcal{C}^0}(\tilde{Y},\bigsqcup_{n\geq 0} BU(n))\arrow[r]&  \textnormal{Map}_{\mathcal{C}^0}(\tilde{T},\bigsqcup_{n\geq 0} BU(n))\,.
    \end{tikzcd}
\end{equation}
This induces a map $\mathcal{B}_{\tilde{Y},\tilde{T}}\to V_{\tilde{Y}}\times^{h}_{\mathcal{V}_{\tilde{T}}}\mathcal{V}_{\tilde{Y}}$, where we use the notation 
$$\mathcal{V}_Z = \textnormal{Map}_{\mathcal{C}^0}(Z,\bigsqcup_{n\geq 0} BU(n))$$
for each topological space $Z$.  If $\bar{T}\hookrightarrow X$ is a neighborhood deformation retract pair then so is $\tilde{T}\hookrightarrow{Y}$. It is then a cofibration in \textbf{Top} and the left vertical and lower horizontal arrow of \eqref{fibrantdiagram} are fibrations in \textbf{Top}. This implies that the natural map 
$$\mathcal{V}_{\tilde{Y},\tilde{T}}=\mathcal{V}_{\tilde{Y}}\times_{\mathcal{V}_{\tilde{T}}}\mathcal{V}_{\tilde{Y}}\longrightarrow \mathcal{V}_{\tilde{Y}}\times^{h}_{\mathcal{V}_{\tilde{T}}}\mathcal{V}_{\tilde{Y}}$$
is a homotopy equivalence. By homotopy inverting, we construct
$
\mathfrak{R}:(\mathcal{B}_{\tilde{Y},\tilde{T}})^{\textnormal{cla}}\to \mathcal{V}_{\tilde{Y},\tilde{T}}\,.
$
It can be shown by following the arguments of Atiyah--Jones \autocite{AJ}, Singer \autocite{Sing}, Donaldson \autocite[Prop 5.1.4]{DK} and Atiyah–-Bott \autocite{AB} that this is a weak homotopy equivalence. We therefore have the natural $\Z_2$-bundle $(D_O(\tilde{Y}))^{\textnormal{cla}}\to \mathcal{V}_{\tilde{Y},\tilde{T}}$.
\end{definition}

We summarize some obvious statements about the above constructions. There is a natural map $u_n: BU(n)\to BU\times \Z$, such that $\pi_2\circ u_n = n$. This induces maps $\mathcal{V}_Z\to \mathcal{C}_Z$ which are homotopy theoretic group completions for any $Z$. In particular we have a natural map $\tilde{\Omega}: \mathcal{V}_{\tilde{Y},\tilde{T}}\to \mathcal{C}_{\tilde{Y},\tilde{T}}:=\mathcal{C}_{\tilde{Y}}\times_{\mathcal{C}_{\tilde{T}}}\mathcal{C}_{\tilde{Y}}$. 

\begin{lemma}
\label{lemma everything is H in gauge}
The spaces $(\mathcal{B}_{\tilde{Y},\tilde{T}})^{\textnormal{cla}}$, $\mathcal{V}_{\tilde{Y},\tilde{T}}$ are H-spaces. The maps $(p_1)^{\textnormal{cla}},(p_2)^{\textnormal{cla}},\mathfrak{R}$ are H-maps. In particular, $(D_O(\tilde{Y}))^{\textnormal{cla}}\to (\mathcal{B}_{\tilde{Y},\tilde{T}})^{\textnormal{cla}}\simeq \mathcal{V}_{\tilde{Y},\tilde{T}}$ is a $\mathbb{Z}_2$-graded weak H-principal $\mathbb{Z}_2$-bundle, so Proposition \ref{exun}(i) determines a unique $\Z_2$-graded weak H-principal $\mathbb{Z}_2$-bundle $D^{\mathcal{C}}_O(\tilde{Y})\to \mathcal{C}_{\tilde{Y},\tilde{T}}$  up to isomorphisms, such that there is an isomorphism
$$
\mathfrak{q}^*(D^{\mathcal{C}}_O(\tilde{Y}))\cong D_O(\tilde{Y})\,,
$$
where $\mathfrak{q}: \mathcal{V}_{\tilde{Y},\tilde{T}} \to \mathcal{C}_{\tilde{Y},\tilde{T}}$ is the homotopy theoretic group completion. If Assumption \ref{ass:astcs} is satisfied by $\tilde{Y}$, then all the H-principal bundles become strong and the isomorphisms canonical.
\end{lemma}
\begin{proof}
The last statement follows using Proposition \ref{exun} (ii).
\end{proof}

\subsection{Moduli space of vector bundles generated by global sections}
\label{section moduli space of vector bundles}
We use the definitions of moduli spaces of vector bundles generated by global section of Friedlander--Walker \cite{FriWal} used by Cao-Gross--Joyce \autocite[Definition 3.18]{CGJ}. For definition of Ind-schemes see for example Gaitsgory--Rozenblyum \autocite{Gait}, for general treatment of indization of categories see Kashiwara-Shapira \autocite[§6]{kashiwara}. For $Z$ a scheme over $\mathbb{C}$, this moduli space is defined as the mapping Ind-scheme :
$$
\mathcal{T}_Z =\textnormal{Map}_{\textbf{IndSch}}(Z,\textnormal{Gr}(\mathbb{C}^{\infty}))\,,
$$
where $\textbf{IndSch}$ is the category of Ind-schemes over $\mathbb{C}$ and we view $\textnormal{Gr}(\mathbb{C}^\infty)$ as an object in this category. 

\begin{definition}
\label{definition Delta}
Induced by the embedding of schemes $D\hookrightarrow Y$, we obtain a map $\rho^{\textnormal{vb}}_{D}:\mathcal{T}_{Y}\to\mathcal{T}_{D}$. We can construct the fiber-product in Ind-schemes $\mathcal{T}_{Y,D}\,.$
There is a natural map $\Omega_Y:\mathcal{T}_{Y}\to \mathcal{M}_{Y}$ given by composing with the natural $\textnormal{Gr}(\mathbb{C}^\infty)\to \textnormal{Perf}_{\C}$. Together with the map $\Omega^{\textnormal{ag}}_{D}:\mathcal{T}_{D}\to \mathcal{M}^{D}$ constructed in the same way, we obtain a homotopy commutative diagram of higher stacks:
$$
\begin{tikzcd}[column sep=1.5cm]
\mathcal{T}_{Y}\arrow[d,"\Omega^{\textnormal{ag}}_Y"]\arrow[r,"\rho^{\textnormal{vb}}_{D}"]&\mathcal{T}_{D}\arrow[d,"\Omega^{\textnormal{ag}}_{D}"]&\arrow[l, "\rho^{\textnormal{vb}}_{D}"]\mathcal{T}_{Y}\arrow[d,"\Omega^{\textnormal{ag}}_Y"]\\
\mathcal{M}_{Y}\arrow[r,"\rho_{D}"]&\mathcal{M}^{D}&\arrow[l, "\rho_{D}"]\mathcal{M}_{Y}\,,
\end{tikzcd}
$$
which induces 
$\Omega^{\textnormal{ag}}:\mathcal{T}_{Y,D}= \mathcal{T}_{Y}\times_{\mathcal{T}_D}\mathcal{T}_Y\to \mathcal{M}_{Y,D}\,.$
\end{definition}

As $(-)^{\textnormal{top}}$ commutes with homotopy colimits by Blanc \cite[Prop. 3.7]{Blanc} for an Ind-scheme $\mathcal{S}$ considered as a higher stack represented by the sequence of closed embeddings of finite type schemes
$S_0\to S_1\to S_2\to \ldots $, its topological realization $(\mathcal{S})^{\textnormal{top}}$ is the co-limit in $\textbf{Top}$ of the sequence 
$
S_0^{\textnormal{an}}\to S_1^{\textnormal{an}}\to S_2^{\textnormal{an}}\to \ldots\,, 
$
because the maps are closed embeddings of CW-complexes and thus cofibrations.
Using that filtered co-limits commute with finite limits, we can express $\mathcal{T}_{Y,D}$ as the filtered co-limit of 
$
\mathcal{T}^p_{Y}\times_{\mathcal{T}^p_{D}}\mathcal{T}^p_{Y}\,,
$
where $\mathcal{T}^p_Z = \textnormal{Map}_{\textbf{Sch}}(Z,\textnormal{Gr}(\mathbb{C}^p))$ for any scheme $Z$. 
From this, it also follows that

$$(\mathcal{T}_{Y,D})^{\textnormal{top}}=\varinjlim_{p\to \infty}(\mathcal{T}^p_{Y})^\textnormal{an}\times_{(\mathcal{T}^p_{D})^{\textnormal{an}}}(\mathcal{T}^p_{Y})^{\textnormal{an}}= \mathcal{T}^{\textnormal{an}}_{Y}\times_{\mathcal{T}^{\textnormal{an}}_{D}}\mathcal{T}^{\textnormal{an}}_{Y}\,.$$ 

We have therefore constructed a map 
\begin{equation}
\label{deltatop}
  \Omega^{\textnormal{top}}: \mathcal{T}^{\textnormal{an}}_{Y}\times_{\mathcal{T}^{\textnormal{an}}_{D}}\mathcal{T}^{\textnormal{an}}_{Y} \to (\mathcal{M}_{Y}\times_{\mathcal{M}^{D}}\mathcal{M}_{Y})^{\textnormal{top}}\,.  
\end{equation}
The following is a non-trivial modification of \cite[Prop. 3.22]{CGJ}, \cite[Prop. 4.5]{gross} to the case of $\mathcal{M}_{Y,D}$. We use in the proof the language of spectra (see  Strickland \autocite{Strickland} and Lewis--May \autocite{equivariantspectra}). We only use that the infinite loop space functor $\Omega^\infty: \textbf{Sp}\to\textnormal{Top}$ preserves homotopy equivalences, where $\textbf{Sp}$ is the category of topological spectra. 
\begin{proposition}
\label{lemma groupcompletion}
The map $\Omega^{\textnormal{top}}: \mathcal{T}^{\textnormal{an}}_{Y}\times_{\mathcal{T}^{\textnormal{an}}_{D}}\mathcal{T}^{\textnormal{an}}_{Y} \to (\mathcal{M}_{Y,D})^{\textnormal{top}}$ is a homotopy theoretic group completion of H-spaces. 
\end{proposition}
\begin{proof}
Let us recall that in a symmetric closed monoidal category $\mathcal{C}$ with the internal hom functor $\textnormal{Map}_{\mathcal{C}}(-,-)$ the contravariant functor $C\mapsto \textnormal{Map}_{\mathcal{C}}(C,D)$ maps co-limits to limits. Thus push-outs are mapped to pullbacks because the homotopy category of higher stacks is symmetric closed monoidal as shown by Toën--Vezzosi in \autocite[Theorem 1.0.4]{TVHAG1}.

 The following diagram
\begin{equation}
    \begin{tikzcd}
     D\arrow[r,"i_{D}"]\arrow[d,"i_{D}"]&Y\\
    Y&
    \end{tikzcd}
\end{equation}
has a push-out $Y\cup _{D}Y$ in the category of schemes over $\mathbb{C}$  using that $i_{D}$ is a closed embedding and Schwede \autocite[Corollary 3.7]{Schwede}. Moreover,  the result of Ferrand \autocite[§6.3]{Ferrand} tells us that $Y\cup _{D}Y$ is projective.

We conclude that there are natural isomorphisms
\begin{align*}
   \mathcal{M}_{Y,D}\cong \textnormal{Map}_{\textbf{HSt}}(Y\cup_{D}Y,\textnormal{Perf}_{\mathbb{C}})=\mathcal{M}^{Y\cup_{D}Y}\,,\\
\mathcal{T}_{Y,D}\cong \textnormal{Map}_{\textbf{IndSch}}(Y\cup_{D}Y,\textnormal{Gr}^\infty(\mathbb{C}))=\mathcal{T}_{Y\cup_{D}Y}\,.
\end{align*}
In fact, under these isomorphisms, the map $\Omega$ from Definition \ref{definition Delta} corresponds to the natural map $\Omega_{Y\cup_{D}Y}:\mathcal{T}_{Y\cup_{D}Y}\to \mathcal{M}^{Y\cup_{D}Y}$.

For a quasi-projective variety $Z$ over $\mathbb{C}$, Friedlander--Walker define in \autocite[Definition 2.9]{FriWal} the space $\mathcal{K}^{\textnormal{semi}}(Z)$ as the infinity loop space  $\Omega^{\infty}\mathcal{T}^{\textnormal{an}}_Z$, where they use that $\mathcal{T}^{\textnormal{an}}$ is an $E_\infty$-space. Therefore there is a map $\mathcal{T}^{\textnormal{an}}_Z\to \mathcal{K}^{\textnormal{semi}}(Z)$, which is a homotopy theoretic group completion by \autocite[6.4]{Lima} and \autocite[§2]{LawLim}. For a dg-category $\mathcal{D}$ over $\mathbb{C}$, Blanc \autocite[Definition 4.1]{Blanc} defines the connective semi-topological K-theory $\tilde{\textbf{K}}^{\textnormal{st}}(\mathcal{D})$ in the category $\textbf{Sp}$. Moreover, in \autocite[Theorem 4.21]{Blanc}, he constructs an equivalence between the $\tilde{\textbf{K}}^{\textnormal{st}}(\mathcal{D})$ and the spectrum of the topological realization of the higher moduli stack of perfect modules of $\mathcal{D}$\footnote{
This moduli stack is denoted in Blanc \autocite{Blanc} by $\mathcal{M}^{\mathcal{D}}$. Unlike the moduli stacks in Toën--Vaquié, it classifies only perfect dg-modules over $\mathcal{D}$ and not the pseudo-perfect ones. For the case $\mathcal{D} = L_{\textnormal{pe}}(Y)$ it therefore coincides with the mapping stack $\mathcal{M}^Y$. When $Y$ is projective and smooth, we already know that $\mathcal{M}^ Y$ and $\mathcal{M}_Y$ are equivalent. }. This induces a homotopy equivalence 
$
\Omega^\infty \tilde{\textbf{K}}^{\textnormal{st}}(\textnormal{Perf}(Z))\to (\mathcal{M}^Z)^{\textnormal{top}}
$
of H-spaces. In \autocite[Theorem 2.3]{AntiHell}, Antieau--Heller prove existence of a natural homotopy equivalence between the H-spaces $\Omega^\infty\tilde{\textbf{K}}^{\textnormal{st}}(\textnormal{Perf}(Z))$ and $\mathcal{K}^{\textnormal{semi}}(Z)$. The composition   $$\mathcal{T}^{\textnormal{an}}_Z\longrightarrow \mathcal{K}^{\textnormal{semi}}(Z)\longrightarrow \Omega^\infty \tilde{\textbf{K}}^{\textnormal{st}}(\textnormal{Perf}(Z))\longrightarrow (\mathcal{M}^Z)^{\textnormal{top}}$$ for $Z=Y\cup_{D}Y$ is homotopy equivalent to $\Omega^{\textnormal{top}}_{Y\cup_{D}Y}$. We have thus shown that $\Omega^{\textnormal{top}}$ is a homotopy theoretic group-completion.
\end{proof}

We now make $(O^{\bowtie})^{\textnormal{top}}\to (\mathcal{M}_{Y,D})^{\textnormal{top}}$ into a weak H-principal $\Z_2$-bundle with respect to the binary operation 
\begin{equation}
\label{muM}
 (\mu_{\mathcal{M}_{Y,D}})^{\textnormal{top}}: (\mathcal{M}_{Y,D})^{\textnormal{top}}\times (\mathcal{M}_{Y,D})^{\textnormal{top}} \to (\mathcal{M}_{Y,D})^{\textnormal{top}}\,,   
\end{equation}
which is determined by
\begin{equation}
\label{diagram}
\begin{tikzcd}
&\mathcal{M}_{Y}\times\mathcal{M}_{Y}\arrow[d,"\rho_{D}\times\rho_{D}"]\arrow[r,"\mu_{\mathcal{M}_{Y}}"]&\mathcal{M}_{Y}\arrow[d,"\rho_{D}"]\\
\mathcal{M}_{Y,D}\times \mathcal{M}_{Y,D}\arrow[ur,"\pi_{1,3}"]\arrow[dr,"\pi_{2,4}"']&\mathcal{M}^{D}\times\mathcal{M}^{D}\arrow[r,"\mu_{\mathcal{M}^{D}}"]&\mathcal{M}^{D}\\
&\mathcal{M}_{Y}\times\mathcal{M}_{Y}\arrow[u,"\rho_{D}\times\rho_{D}"']\arrow[r,"\mu_{\mathcal{M}_{Y}}"]&\mathcal{M}_{Y}\arrow[u,"\rho_{D}"']\,,
\end{tikzcd}
\end{equation}
It can be checked to be commutative and associative in $\textbf{Ho(HSta)}_{\mathbb{C}}$. In fact, as $\mathcal{M}_{Y,D}$ is a homotopy fiber product of $\Gamma$-objects, it is itself one in $\textbf{HSta}_{\mathbb{C}}$ (see Bousfield--Friedlander \autocite[§3]{Bousfield} for definition of $\Gamma$-objects in model categories and Blanc \autocite[p. 45]{Blanc} for the construction in this case). Let us set some notation. For any $\underline{a}$, $\underline{b}$, we have the isomorphisms
$$
\pi_{1,3}^*(\Sigma_{\underline{a}})\otimes\pi^*_{2,4}(\Sigma_{\underline{a}})^*\cong \pi_{1,3}^*(\Sigma_{\underline{b}})\otimes\pi_{2,4}^*(\Sigma_{\underline{b}})^*
$$
by similar construction as in \eqref{taubowtie}. In particular, for fixed $\bowtie$ we obtain the isomorphism
\begin{equation}
\label{sigmabowtie}
\sigma_{\bowtie}:\pi_{1,3}^*(\Sigma_{\underline{k}})\otimes\pi_{2,4}^*(\Sigma_{\underline{k}})^*\overset{\sim}{\longrightarrow} \pi_{1,3}^*(\Sigma_{\underline{0}})\otimes\pi_{2,4}^*(\Sigma_{\underline{0}})^*\,.    
\end{equation}

\begin{proposition}
 \label{proposition Ovartheta Hprincipal} 
Let $O^{\bowtie}\to \mathcal{M}_{Y,D}$ be the $\Z_2$-bundle from Definition \ref{definitionbowtie}, then there exists an isomorphism 
$
\phi^{\bowtie}: O^{\bowtie}\boxtimes _{\Z_2}O^{\bowtie}\to \mu_{\mathcal{M}}^*(O^{\bowtie})
$ depending on $\bowtie$ but independent of $\mathfrak{ord}$.
Moreover, we have
$$
(\textnormal{id}_{\mathcal{M}_{Y,D}}\times \mu_{\mathcal{M}})^*(\phi^{\bowtie})(\textnormal{id}\times \phi^{\bowtie}) = ( \mu_{\mathcal{M}}\times \textnormal{id}_{\mathcal{M}_{Y,D}})^*(\phi^{\bowtie})\circ (\phi^{\bowtie}\times \textnormal{id})\,.
$$
\end{proposition}
\begin{proof}
First note, that we have the isomorphism
$$\mu^*_{\mathcal{M}_Y}(\Lambda_{\underline{0}})\cong \pi_{1}^*\Lambda_{\underline{0}}\otimes \pi_{1,2}^*\Sigma_{\underline{0}}\otimes \pi_{1,2}^*\circ \sigma^*\Sigma_{\underline{0}}\otimes\pi_2^*\Lambda_{\underline{0}}\,.$$
Using this together with \eqref{diagram}  we obtain the following commutative diagram
\begin{equation*}
    \begin{tikzcd}
    &\pi_1^*\Lambda_{\underline{0}}\otimes \pi_3^*\Lambda_{\underline{0}}\otimes \pi_4^*\Lambda^*_{\underline{0}}\otimes \pi_2^*\Lambda^*_{\underline{0}}\arrow[d]\\
      &\begin{subarray} a \pi_1^*\Lambda_{\underline{0}}\otimes\pi_{1,3}^*\Sigma_{\underline{0}}\otimes \pi_{1,3}^*\Sigma^*_{\underline{k}}\otimes\pi_{3}^*\Lambda_{\underline{0}}\\
   \otimes\pi_4^*\Lambda^*_{\underline{0}}\otimes\pi_{2,4}^*\Sigma^*_{\underline{0}}\otimes \pi_{2,4}^*\Sigma_{\underline{k}} \otimes\pi_2^*\Lambda^*_{\underline{0}}
   \end{subarray}\arrow{d}\\
    \arrow{dd}\mu^*(\pi_1^*\Lambda_{\underline{0}}\otimes \pi_{2}^*\Lambda^*_{\underline{0}})\arrow[r]&\begin{subarray}a\pi^*_{1}\Lambda_{\underline{0}}\otimes \pi_{1,3}^*\Sigma_{\underline{0}} \otimes\pi_{1,3}^*\circ \sigma^*\Sigma_{\underline{0}}\otimes \pi_3^*(\Lambda_{\underline{0}})\\\otimes \pi_4^*\Lambda^*_{\underline{0}}\otimes\pi_{2,4}^*\circ \sigma^*\Sigma^*_{\underline{0}}\otimes \pi^*_{2,4}\Sigma_{\underline{0}}\otimes \pi_2^*\Lambda^*_{\underline{0}}\end{subarray}\arrow[d]\\
    &\begin{subarray}a  \pi^*_{1}\Lambda^*_{\underline{k}}\otimes \pi_{1,3}^*\circ\sigma^*\Sigma^*_{\underline{k}} \otimes\pi_{1,3}^* \Sigma^*_{\underline{k}}\otimes \pi_3^*(\Lambda^*_{\underline{k}})\\\otimes \pi_4^*\Lambda_{\underline{k}}\otimes\pi_{2,4}^*\Sigma_{\underline{k}}\otimes \pi^*_{2,4}\circ \sigma^*\Sigma^*_{\underline{k}}\otimes \pi_2^*\Lambda_{\underline{k}}\end{subarray}
  \arrow[d]\\
   \mu^*(\pi_1^*\Lambda^*_{\underline{0}}\otimes \pi^*_2\Lambda_{\underline{0}})\arrow[r]& \begin{subarray} a \pi_1^*\Lambda^*_{\underline{0}}\otimes\pi_{1,3}^*\Sigma^*_{\underline{0}}\otimes\pi_{1,3}^*\circ\sigma^*\Sigma^*_{\underline{0}}\otimes\pi_{3}^*\Lambda^*_{\underline{0}}\\
   \otimes\pi_4^*\Lambda_{\underline{0}}\otimes\pi_{2,4}^*\Sigma_{\underline{0}}\otimes\pi_{2,4}^*\circ\sigma^*\Sigma_{\underline{0}}\otimes\pi_2^*\Lambda_{\underline{0}}
   \end{subarray}\arrow{d}\\
   &\begin{subarray} a \pi_1^*\Lambda^*_{\underline{0}}\otimes\pi_{1,3}^*\Sigma^*_{\underline{0}}\otimes\pi_{1,3}^*\Sigma_{\underline{k}}\otimes\pi_{3}^*\Lambda^*_{\underline{0}}\\
   \otimes\pi_4^*\Lambda_{\underline{0}}
   \otimes\pi_{2,4}^*\Sigma_{\underline{0}}
   \otimes \pi_{2,4}^* \Sigma^*_{\underline{k}}\otimes \otimes\pi_2\Lambda^*_{\underline{0}}
   \end{subarray}\arrow{d}\\
  & \pi_1^*\Lambda^*_{\underline{0}}\otimes \pi_3^*\Lambda^*_{\underline{0}}\otimes \pi_4^*\Lambda_{\underline{0}}\otimes \pi_2^*\Lambda_{\underline{0}}
    \end{tikzcd}
\end{equation*}
Where the left vertical arrow is $\mu^*(\vartheta_{\bowtie})$ and  the composition of all arrows on the right is $\pi_{1,2}^*(\vartheta_{\bowtie})\otimes \pi_{3,4}^*(\vartheta_{\bowtie})$ by generalization of the arguments in Cao--Gross--Joyce  \cite[p. 43]{CGJ}.  To construct arrows on the right, we use multiple times Serre duality and \eqref{sigmabowtie}.  This is what induces the isomorphism $\phi^{\bowtie}$. Note that we need to permute $\pi_2^*\Lambda_{\underline{0}}$ through $\pi^*_{3}\Lambda_{\underline{0}}$ and $\pi^*_{4}\Lambda^*_{\underline{0}}$ on both ends, giving the extra sign \begin{equation}\label{strongHsigns}
(-1)^{\textnormal{deg}(\pi_2^*\Lambda_{\underline{0}})\big(\textnormal{deg}(\pi_3^*\Lambda_{\underline{0}})+\textnormal{deg}(\pi^*_4\Lambda_{\underline{0}})\big)}
\end{equation} for the isomorphism of $\Z_2$-bundles. Checking the associativity of the isomorphism combines  the ideas of the proof of associativity in Lemma \ref{lemma Z2graded} and the ones used in the diagram above. The independence of $\mathfrak{ord}$ follows by the same arguments as used in 3. of the proof of Theorem \ref{maintheorem}.
\end{proof}

Using the notation from Definition \ref{representatinofproduct}, we have an obvious map 
\begin{equation}
    \label{eq:Lambda}\Lambda:\mathcal{T}^{\textnormal{an}}_{Y}\times_{\mathcal{T}^{\textnormal{an}}_{D}}\mathcal{T}^{\textnormal{an}}_{Y}\longrightarrow \mathcal{V}_{Y,D}
\end{equation}
 which corresponds to the inclusion of holomorphic maps into the continuous maps to $\textnormal{Gr}(\mathbb{C}^{\infty})^{\textnormal{an}}$\,. This map is continuous  (see Friedlander--Walker \autocite{FriWal2}). 

\subsection{Tubular neighborhood of $D$}
\label{sec:tubular-neighborhood}
We now produce the data $U,T$ and  the normal vector field $\nu$ from Definition \ref{definition double}.
\begin{definition}
\label{def:UandT-defined}
    As in \cite[p. 229]{Peters-Steenbrink}, there exists a smooth triangulation of $Y^{\an}$ such that  $D$ is a finite subcomplex. This gives us a regular neighbourhood $\tilde{N}$ of $D$ as in \cite[Chapter 3]{RourkeSanderson}. By \cite[Theorem (1a')]{Hirsch}, one can use a homeomorphism of $Y$ onto itself to replace it by a smooth regular neighborhood $\bar{N}$ which is codimension 0 submanifold of $Y$ with a smooth boundary $\partial \bar{N}$. Note that regular neighborhoods automatically strongly deformation retract onto $D$. By \cite[Theorem 1.(b) and (c)]{Hirsch}, two different choices of $\bar{N}$ would be identified by a diffeomorphism $Y\to Y$ which fixes $D$ pointwise and is isotopic to $\operatorname{id}_{Y}$.\footnote{This implies uniqueness for our purposes.} By \cite[Lemma 11]{Lerario-Meroni-Zuddas}, there is a smooth function $u:\bar{N}\to[0,1]$ which has no critical points in $(0,1)$ and $\bar{N}_{\epsilon}:= u^{-1}\big([0,\epsilon]\big)$ are smooth regular neighborhoods for any $\epsilon \in (0,1]$. Thus we can set
$$
U:=Y\backslash N_{\frac{1}{2}}  \qquad \textnormal{and}\qquad T:= U\cap N_{\frac{3}{4}}
$$
where $N_{\epsilon} $ is the interior of $\bar{N}_{\epsilon}$. By regularity, we can use $\operatorname{grad}(u)|_{\partial U}$ to obtain the normal field $\nu$ needed in Definition \ref{definition double}.
\end{definition}
Set $\ov{S}:=\ov{N}_{\frac{3}{4}}$. The above data induces a homotopy auto-equivalence of $Y$
$$\tilde{H}:Y\times I\to I$$
such that $\tilde{H}( -,t)|_D =\operatorname{id}_D$, $\tilde{H}(\ov{S},t)\subset \ov{S}$, $\tilde{H}(\ov{T},1)= D$ and $\tilde{H}(-,t)|_{Y\backslash \ov{N}}=\operatorname{id}_{Y\backslash \ov{N}}$. The pullback along $\tilde{H}(1,-)$ induces the homotopy equivalences 
\begin{align*}
\label{upsilon}
  \Upsilon: \mathcal{V}_{Y,D}\longrightarrow \mathcal{V}_{Y,\bar{S}}\,,\qquad
  \Upsilon_{\mathcal{C}}:  \mathcal{C}_{Y,D}\longrightarrow \mathcal{C}_{Y,\bar{S}}\,,
  \numberthis
\end{align*}
which we use from now on to identify the spaces.

We now use the data from Definition \ref{def:UandT-defined} to construct $\tilde{Y}$ from \ref{definition double}. Define
 \begin{equation}
 \label{Sigma_{tilde{Y}}}
     G_{\tilde{Y}}: \mathcal{V}_{Y,\bar{S}}\longrightarrow \mathcal{V}_{\tilde{Y},\tilde{T}}\,, \qquad G^{\mathcal{C}}_{\tilde{Y}}:\mathcal{C}_{Y,\bar{S}}\longrightarrow \mathcal{C}_{\tilde{Y},\tilde{T}}\,, 
 \end{equation}
 by $G_{\tilde{Y}}(m_1,m_2) = (\tilde{m}_1,\tilde{m}_2)$ for each $(m_1,m_2)\in \mathcal{V}_{Y}\times_{\mathcal{V}_{\bar{S}}}\mathcal{V}_{Y}$ such that
\begin{align*}
\label{gluing equation}
   \tilde{m}_1|_{U} = m_1|_U\,,\quad \tilde{m}_1|_{-U} = m_1|_U\,,\qquad 
\tilde{m}_2|_{U} = m_2|_U\,,\quad \tilde{m}_2|_{-U} = m_1|_U\,.
\numberthis
\end{align*}
Which gives us $\Z_2$-graded weak H-principal $\Z_2$-bundles (applying Proposition \ref{exun}(i) in the second case):
 \begin{equation}
 \label{DO}
D_O:=\Upsilon^*\circ G_{\tilde{Y}}^*(D_O(\tilde{Y}))\,,\quad  D_O^{\mathcal{C}}\longrightarrow  \mathcal{C}_{Y,D}\,.
\end{equation}
We now fix choices of compact tubular neighborhorhoods $\bar{S}_i\supset D_i$ such that $\bar{S}_i\subset \bar{S}$.
\begin{lemma}
\label{lemma independent}
\label{lemma extend}
Let $E,F,\phi$ be smooth vector bundles and $\phi:E|_{D}\to F|_{D}$ be an isomorphism which is smooth on $D$. Then there exists a contractible choice of smooth isomorphism $\bar{\Phi}: E|_{\bar{S}}\to F|_{\bar{S}}$, $\Phi_i: E|_{\bar{S}_i}\to F|_{\bar{S}_i}$ such that $\bar{\Phi}|_D = \phi$ and $\Phi_i|_{D_i}=\phi|_{D_i}$. Moreover $\Phi_i$ can be deformed into $\bar{\Phi}|_{\bar{S}_i}$ along isomorphism restricting to $\phi|_{D_i}$. Then, the map \eqref{Sigma_{tilde{Y}}} corresponds to
$$
[E,F,\phi]\longmapsto [E,F,\bar{\Phi}]\,.
$$
The resulting $\Z_2$-graded weak H-principal $\Z_2$-bundles $D_O$ is independent of the choices made up to canonical isomorphisms. By Proposition \ref{exun}, this also holds for $D_O^{\mathcal{C}}$ whenever $\tilde{Y}$ satisfies Assumption \ref{ass:astcs}. 
\end{lemma}
\begin{proof}
Since $\ov{S}$ is a strong deformation retract of $D$, the existence of $\bar{\Phi}$ is standard. The same holds for $\Phi_i$. They can be deformed into $\bar{\Phi}|_{\ov{S}_i}$ because both restrict to $\phi|_{D_i}$ on $D_i$ and there is a contractible choice of such isomorphisms on $\ov{S}_i$.

 Different choices of sizes of $\ov{S}_i$ neighborhoods correspond to choices of some small $\epsilon_i>0$. As explained in Definition \ref{def:UandT-defined}, the construction of the $\Z_2$-bundle $D_O\to \mathcal{V}_{Y,D}$ only depends on contractible choices which induces canonical isomorphisms between two such constructions. during the construction of $\tilde{Y}$ in Definition \ref{definition double}.  We could also use excision as in \cite[Thm. 2.10]{Markus} to prove directly that the $\Z_2$-bundle for two different constructions of $\tilde{T}$ are canonically isomorphic.
\end{proof}
 
From now on, we will not distinguish between Ind-schemes and their analytifications. Note that by Proposition \ref{proposition homotopy theoretic group completion}, we have the commutative diagram
\begin{equation}
\label{htgc}
    \begin{tikzcd}[column sep=1.5cm]
    \arrow[d,"(\Omega^{\textnormal{ag}})^{\textnormal{top}}"']\mathcal{T}_{Y,D}\arrow[r,"\Lambda"]&\mathcal{V}_{Y,D}\arrow[d,"\Omega"']\arrow[r,"\Upsilon"]&\mathcal{V}_{Y,\bar{S}}\arrow[d]\arrow[r,"G_{\tilde{Y}}"]&\mathcal{V}_{\tilde{Y},\tilde{T}}\arrow[d,"\tilde{\Omega}"]\\
    (\mathcal{M}_{Y,D})^{\textnormal{top}}\arrow[r,"\Gamma"]&\mathcal{C}_{Y,D}\arrow[r,"\Upsilon_{\mathcal{C}}"]&\mathcal{C}_{Y,\bar{S}}\arrow[r,"G^{\mathcal{C}}_{\tilde{Y}}"]&\mathcal{C}_{\tilde{Y},\tilde{T}}
    \end{tikzcd}\,.
\end{equation}
The map $\Gamma$ was expressed explicitly in Definition \ref{definition Gamma} and all the vertical arrows are homotopy theoretic group completions.

\subsection{Comparing excisions}
\label{section comon resolution}
 We begin by defining  a new set of differential geometric line bundles on $\mathcal{T}_{Y,D}\times \mathcal{T}_{Y,D}$. Let $D=\bar{\partial}+\bar{\partial}^*:\Gamma^{\infty}(\mathcal{A}^{0,\textnormal{even}})\to \Gamma^{\infty}(\mathcal{A}^{0,\textnormal{odd}})$, then we define
$$
\hat{\Sigma}^{\textnormal{dg}}_{\underline{a},P,Q}\longrightarrow \mathcal{A}_P\times \mathcal{A}_Q
$$
given by a complex line $\textnormal{det}(D^{\nabla_{\textnormal{Hom}(E,F\otimes L_{\underline{a}})}})$
 at each point $(\nabla_P,\nabla_Q)$, where $E,F$ are the associated complex vector bundles to $P,Q$ and $\nabla_{\textnormal{Hom}(E,F)}$ is the induced connection on $E^*\otimes F$. This descends to a line bundle on $\hat{\Sigma}^{\textnormal{dg}}_{\underline{a},P,Q}\to \mathcal{B}_P\times \mathcal{B}_Q$. Taking the union over isomorphism classes $[P]$, $[Q]$, we obtain a line bundle
 $$
 \hat{\Sigma}^{\textnormal{dg}}_{\underline{a}}\longrightarrow \mathcal{B}_Y\times \mathcal{B}_Y\,.
 $$
 Using the natural map $\Lambda\times \Lambda: \mathcal{T}_{Y}\times \mathcal{T}_Y\to \mathcal{V}_Y\times \mathcal{V}_Y\simeq (\mathcal{B}_Y)^{\textnormal{cla}}\times (\mathcal{B}_Y)^{\textnormal{cla}}$, we can pull back these bundles to obtain \begin{equation}\Sigma^{\textnormal{dg}}_{\underline{a}}\longrightarrow \mathcal{T}_{Y}\times \mathcal{T}_Y\quad \textnormal{and}\quad \Lambda^{\textnormal{dg}}_{\underline{a}}=\Delta^*(\Sigma^{\textnormal{dg}}_{\underline{a}}).\end{equation}
 \begin{lemma}
 \label{lemma dg orientations}
After a choice of $\bowtie$ there exist natural isomorphisms
 \begin{align*}
   \sigma^{\textnormal{dg}}_{\bowtie}:&\pi_{1,3}^*(\Sigma^{\textnormal{dg}}_{\underline{k}})\otimes\pi_{2,4}^*(\Sigma^{\textnormal{dg}}_{\underline{k}})^*\stackrel{\sim}{\longrightarrow} \pi_{1,3}^*(\Sigma^{\textnormal{dg}}_{\underline{0}})\otimes\pi_{2,4}^*(\Sigma^{\textnormal{dg}}_{\underline{0}})^*\,,\\
   \tau^{\textnormal{dg}}_{\bowtie}:&\pi_{1}^*(\Lambda^{\textnormal{dg}}_{\underline{k}})\otimes\pi_{2}^*(\Lambda^{\textnormal{dg}}_{\underline{k}})^*\stackrel{\sim}{\longrightarrow} \pi_{1}^*(\Lambda^{\textnormal{dg}}_{\underline{0}})\otimes\pi_{2}^*(\Lambda^{\textnormal{dg}}_{\underline{0}})^*\,.
 \end{align*}
 Moreover, we have the isomorphisms 
 \begin{equation}
     \#_{\Sigma^{\textnormal{dg}}_{\underline{a}}}:\Sigma^{\textnormal{dg}}_{\underline{a}}\xlongrightarrow{\sim}\sigma^*(\Sigma^{\textnormal{dg}}_{\underline{k}-\underline{a}})^*\,,\quad \#_{\underline{a}}: \Lambda^{\textnormal{dg}}_{\underline{a}}\xlongrightarrow{\sim} (\Lambda^{\textnormal{dg}}_{\underline{k}-\underline{a}})^*\,.
 \end{equation}
\begin{proof}
The following construction works in families due to the work done in §\ref{section pseudo}, so we restrict ourselves to a point $(p,q) = ([E_1,F_1,\phi_1],[E_2,F_2,\phi_2])$, where $\phi_{1/2,i}:E_{1/2}|_{D_i}\to F_{1/2}|_{D_i}$ are isomorphism. We also set the notation \begin{equation}V_{\underline{a}} = \textnormal{End}(E_1,F_1\otimes L_{\underline{a}})\quad \textnormal{and}\quad W_{\underline{a}} = \textnormal{End}(E_2,F_2\otimes L_{\underline{a}})\,.
\end{equation}
Using $\bar{\Phi}_{\underline{a}}$ to denote the isomorphism $V_{\underline{a}}|_T\to W_{\underline{a}}|_T$ constructed in Lemma \ref{lemma independent}. We then have the data
$$
   \begin{tikzcd}
   \arrow[d,dashrightarrow,"\Omega^{-1}"']\sigma_{V_{\underline{k}}}\arrow[r,dashrightarrow,"i\Phi_{\underline{k}}"]&\sigma_{W_{\underline{k}}}\arrow[d,dashrightarrow,"\Omega^{-1}"]\\
   \sigma_{V_{\underline{0}}}\arrow[r,dashrightarrow,"i\Phi_{\underline{0}}"']&\sigma_{W_{\underline{0}}}
   \end{tikzcd} 
$$
using that $\Omega$ is invertible outside of $D$. We therefore obtain the isomorphism $\sigma_{\bowtie}$ from  Lemma \ref{upmeier}.  Then we have the standard definition of the Hodge star
$
\star:\Gamma^\infty(\mathcal{A}^{p,q})\to \Gamma^\infty(\mathcal{A}^{4-p,4-q})\,,
$
given by 
$
\alpha \wedge *\beta =  \langle\alpha,\beta\rangle\Omega\wedge \bar{\Omega}, 
$
where $\langle-,-\rangle$ is the hermitian metric on forms.  We define then the anti-linear maps
\begin{align*}
\label{antilinear maps}
\#_1: \mathcal{A}^{0,\textnormal{even}}\to\mathcal{A}^{0,\textnormal{even}} \otimes K_X\,,\qquad
\#_2:\mathcal{A}^{0,\textnormal{odd}}\to \mathcal{A}^{0,\textnormal{odd}}\otimes K_X\,\\
\#^{\textnormal{op}}_1: \mathcal{A}^{0,\textnormal{even}}\otimes K_X\to\mathcal{A}^{0,\textnormal{even}} \,,\qquad
\#^{\textnormal{op}}_2:\mathcal{A}^{0,\textnormal{odd}}\otimes K_X\to \mathcal{A}^{0,\textnormal{odd}}\,,
\numberthis{}
\end{align*}
by
$
  \#_1|_{\mathcal{A}^{0,2q}} = (-1)^q\star\,, \#_2|_{\mathcal{A}^{0,2q+1}} =(-1)^{q+1}\star\,,
   \#^{\textnormal{op}}_1|_{\mathcal{A}^{4,2q}} = (-1)^q\star\,, \#^{\textnormal{op}}_2|_{\mathcal{A}^{4,2q+1}} =(-1)^{q+1}\star\,.
$
These solve $\#_i^{\textnormal{op}}\circ \#_i = \textnormal{id}$ and $\#_i\circ \#_i^{\textnormal{op}} = \textnormal{id}$. We have the commutativity relations $D_{K_X}\circ\#_1=\#_2\circ D$ and $\#_2^{\textnormal{op}}\circ D_{K_X} = D\circ \#_1^{\textnormal{op}}$ and obtain the isomorphisms $   \det(\sigma_{V_{\underline{a}}})\cong \overline{\det(\sigma_{V_{\underline{k}-{\underline{a}}}})}\cong{\det(\sigma_{V_{\underline{k}-{\underline{a}}}})}^*\,, $
where the second isomorphism on both lines uses the hermitian metrics on forms which descend to a hermitian metric on the determinant. 
\end{proof}

 \end{lemma}
\begin{definition}Let $\mathcal{U}^{\textnormal{vb}}\to Y\times \mathcal{T}_{Y}$ be the universal vector bundle generated by global sections. We define
\begin{alignat*}{2}
    \mathcal{E}\textnormal{xt}^{\textnormal{vb}}_{\underline{a}}& = \pi_{1,2\,*}(\pi_{1,3}^*\,\mathcal{U}_{\textnormal{vb}}^*\otimes \pi_{1,3}^*\,\mathcal{U}_{\textnormal{vb}}\otimes \pi_1^*(L_{\underline{a}}))\,,\qquad  & \mathbb{P}^{\textnormal{vb}}_{\underline{a}} &=
  \Delta^*\mathcal{E}\textnormal{xt}^{\textnormal{vb}}_{\underline{a}}\\
  \Sigma_{\underline{a}}^{\textnormal{vb}}&= \textnormal{det}(\mathcal{E}\textnormal{xt}^{\textnormal{vb}}_{\underline{a}})\,,\qquad& \Lambda^{\textnormal{vb}}_{\underline{a}} &= \Delta^*\big(\textnormal{det}(\Sigma^{\textnormal{vb}}_{\underline{a}})\big)\
\end{alignat*}
We then have the isomorphisms 
 \begin{align*}
   \sigma^{\textnormal{vb}}_{\bowtie} &= (\Omega^{\textnormal{ag}})^*(\sigma_{\bowtie}) :\pi_{1,3}^*(\Sigma^{\textnormal{vb}}_{\underline{k}})\otimes\pi_{2,4}^*(\Sigma^{\textnormal{vb}}_{\underline{k}})^*\overset{\sim}{\longrightarrow} \pi_{1,3}^*(\Sigma^{\textnormal{vb}}_{\underline{0}})\otimes\pi_{2,4}^*(\Sigma^{\textnormal{vb}}_{\underline{0}})^*\,,\\
 \tau^{\textnormal{vb}}_{\bowtie} &=\Delta^*(\sigma^{\textnormal{vb}}_{\bowtie}) :\pi_{1}^*(\Lambda^{\textnormal{vb}}_{\underline{k}})\otimes\pi_{2}^*(\Lambda^{\textnormal{vb}}_{\underline{k}})^*\overset{\sim}{\longrightarrow} \pi_{1}^*(\Lambda^{\textnormal{vb}}_{\underline{0}})\otimes\pi_{2}^*(\Lambda^{\textnormal{vb}}_{\underline{0}})^{*}\,.
 \end{align*}
In particular, we have the $\Z_2$-bundle $O^{\bowtie}_{\textnormal{vb}}\to \mathcal{T}_{Y,D}$, such that naturally $O^{\bowtie}_{\textnormal{vb}}\cong (\Omega^{\textnormal{ag}})^*O^{\bowtie}$.
\end{definition} 

The following proposition is the result of trying to develop a more general framework of relating compactly supported coherent sheaves to compactly supported pseudo-differential operators using cohesive modules of Block \cite{block} and Yu \cite{Yu}. 
\begin{proposition}
\label{proposition comparing excision}
There are natural isomorphisms $\kappa^{\textnormal{a,d}}_{\underline{a}}:\Sigma_{\underline{a}}^{\textnormal{vb}}\cong \Sigma^{\textnormal{dg}}_{\underline{a}}$ such that the the following diagram commutes up to natural isotopies:
\begin{equation}
    \label{comparing excision}
  \begin{tikzcd}
\arrow[d,"\pi_{1,3}^*(\kappa^{\textnormal{a,d}}_{\underline{k}})\otimes \pi_{2,4}^*(\kappa^{\textnormal{a,d}}_{\underline{k}})^{-*}"']\pi_{1,3}^*(\Sigma_{\underline{k}}^{\textnormal{vb}})\otimes \pi_{2,4}^*(\Sigma^{\textnormal{vb}}_{\underline{k}})^*\arrow[r,"\sigma^{\textnormal{vb}}_{\bowtie}"]&\pi_{1,3}^*(\Sigma^{\textnormal{vb}}_{\underline{0}})\otimes \pi_{2,4}^*(\Sigma^{\textnormal{vb}}_{\underline{0}})^*\arrow[d,"\pi_{1,3}^*(\kappa^{\textnormal{a,d}}_{\underline{0}})\otimes \pi_{2,4}^*(\kappa^{\textnormal{a,d}}_{\underline{0}})^{-*}"]\\
\pi_{1,3}^*(\Sigma^{\textnormal{dg}}_{\underline{k}})\otimes \pi_{2,4}^*(\Sigma^{\textnormal{dg}}_{\underline{k}})^{*}\arrow[r,"\sigma^{dg}_{\bowtie}"]&\pi_{1,3}^*(\Sigma^{\textnormal{dg}}_{\underline{0}})\otimes \pi_{2,4}^*(\Sigma^{\textnormal{dg}}_{\underline{0}})^*\,,
\\
\arrow[d]\pi_{1}^*(\Lambda_{\underline{k}}^{\textnormal{vb}})\otimes \pi_{2}^*(\Lambda^{\textnormal{vb}}_{\underline{k}})^*\arrow[r,"\tau^{\textnormal{vb}}_{\bowtie}"]&\pi_{1}^*(\Lambda^{\textnormal{vb}}_{\underline{0}})\otimes \pi_{2}^*(\Lambda^{\textnormal{vb}}_{\underline{0}})^*\arrow[d]\\
\pi_{1}^*(\Lambda^{\textnormal{dg}}_{\underline{k}})\otimes \pi_{2}^*(\Lambda^{\textnormal{dg}}_{\underline{k}})^{*}\arrow[r,"\tau^{dg}_{\bowtie}"]&\pi_{1}^*(\Lambda^{\textnormal{dg}}_{\underline{0}})\otimes \pi_{2}^*(\Lambda^{\textnormal{dg}}_{\underline{0}})^*\,.
\end{tikzcd}  
\end{equation}
\end{proposition}
 \begin{proof}
We examine up close the definitions of each object involved and show that up to natural isotopies in families the diagram commutes. We begin therefore with the definition of $\tau^{a,d}_{\underline{a}}$. We restrict again to a point $(p,q) = ([E_1,F_1,\phi_1],[E_2,F_2,\phi_2])$ as it can be shown by using the arguments of §\ref{section pseudo}, \cite[Prop. 3.25]{CGJ} that our methods work in families. We also use
 \begin{align*}
     V_{\underline{a},i}:= V_{\underline{a}}|_{D_i}\,,\qquad  W_{\underline{a},i}:= W_{\underline{a}}|_{D_i}\qquad  \phi_{\underline{a},i}:V_{\underline{a},i}\stackrel{\sim}{\longrightarrow}W_{\underline{a},i}\\
     \Phi_{\underline{a},i}:V_{\underline{a}}|_{S_i}\stackrel{\sim}{\longrightarrow}W_{\underline{a}}|_{T_i}\,,\qquad  \Phi_{\underline{a}}:V_{\underline{a}}|_{T}\stackrel{\sim}{\longrightarrow}W_{\underline{a}}|_{T}
 \end{align*}
 
 Let $E\to Y$ be a holomorphic vector bundle, then $R\Gamma^\bullet(E) = \Gamma(E\otimes \mathcal{A}^{0,\bullet})$. where the differential is given by $\bar{\partial}_E = \bar{\partial}^{\nabla_E}$. Here $\nabla_E$ is the corresponding Chern connection. Let 
	$D_E = \bar{\partial}_E+\bar{\partial}_E^*:\Gamma(E\otimes \mathcal{A}^{0,\textnormal{even}})\to \Gamma(E\otimes \mathcal{A}^{0,\textnormal{odd}})\,,$ then Hodge theory gives us the natural isomorphisms
	$
	\textnormal{det}(D_E)\cong \textnormal{det}\big(R\Gamma^\bullet(E)\big)
	$
	after making a contractible choice of metric on $E$. Continuing to use the notation 
	from Lemma \ref{lemma dg orientations}, we obtain the isomorphisms
\begin{align*}
\kappa^{\textnormal{a,d}}_{\underline{a}}|_p:\Sigma^{\textnormal{vb}}_{\underline{a}}|_{p} \cong \textnormal{det}\big(R\Gamma^\bullet(V_{\underline{a}})\big)\cong \textnormal{det}\big( D_{V_{\underline{a}}}\big)\cong \textnormal{det}(D^{\nabla_{V_{\underline{a}}}})\cong \Sigma^{\textnormal{ag}}_{\underline{a}}|_{p}\,
\end{align*}
generalizing those of  Cao--Gross--Joyce  \cite[Prop. 3.25]{CGJ}, Cao--Leung \cite[Thm. 2.2]{CaoLeungor}, Joyce--Upmeier \cite[p. 38]{JoyceMarkus}.
Recall now that at $(p,q)$ the isomorphism 	$\pi_{1,3}^*(\Sigma_{\underline{a}})\otimes \pi^*_{2,4}(\Sigma^*_{\underline{a}})\cong \pi_{1,3}^*(\Sigma_{\underline{a}-e_i})\otimes\pi_{2,4}^*(\Sigma^*_{\underline{a}-e_i})$ is given by 
\begin{align*}
\label{restriction}
&\textnormal{det}\big(R\Gamma^\bullet(V_{\underline{a}})\big)\otimes \textnormal{det}\big(R\Gamma^\bullet(W_{\underline{a}})\big)^*
\cong \textnormal{det}\big(R\Gamma^\bullet(V_{\underline{a}-e_i})\big)\otimes \textnormal{det}\big(R\Gamma^\bullet(V_{\underline{a},i})\big) \\
&\otimes \textnormal{det}\big(R\Gamma^\bullet(W_{\underline{a},i})\big)^*\otimes \textnormal{det}\big(R\Gamma^\bullet(W_{\underline{a}-e_i})\big)^* \cong  \textnormal{det}\big(R\Gamma^\bullet(V_{\underline{a}-e_i})\big)\otimes \textnormal{det}\big(R\Gamma^\bullet(W_{\underline{a}-e_i})\big)^*\,,
\numberthis{}
\end{align*}
where we are using the short exact sequences 
$
0\to V_{\underline{a}-e_i}\to V_{\underline{a}}\to V_{\underline{a}|_{D_i}}\to 0$ and  $0\to W_{\underline{a}-e_i}\to W_{\underline{a}}\to W_{\underline{a}|_{D_i}}\to 0\,.
$ 
We have the exact complex 
\begin{equation}
\label{both important maps}
    \begin{tikzcd}[every label/.append
style={font=\tiny},column sep=large]
           V_{\underline{a}-e_i}\oplus W_{\underline{a}-e_i}
           \arrow[r,labels={inner sep=7},"{(f_{V_{\underline{a}}},f_{W_{\underline{a}}})}"]&K_{\underline{a},i} \arrow[r,"\begin{pmatrix}p_{V_{\underline{a}}}\\p_{W_{\underline{a}}}
           \end{pmatrix}"]
           & V_{\underline{a}}\oplus W_{\underline{a}}\arrow[r,labels={inner sep=7},"{(r_{V_{\underline{a}}},-\phi_{\underline{a},i}\circ r_{W_{\underline{a}}})}"]& V_{\underline{a},i}\to 0\,,
    \end{tikzcd}
\end{equation}
Here $\rho_{V_{\underline{a}}/W_{\underline{a}}}$ are restrictions, $f_{V_{\underline{a}}/W_{\underline{a}}}$ the factors of inclusion and $p_{V_{\underline{a}}/W_{\underline{a}}}\circ f_{V_{\underline{a}}/W_{\underline{a}}}=s_i$ for the section $s_i:\mathcal{O}\xrightarrow{s_i}\mathcal{O}(D_i)$. Moreover, $K_{\underline{a},i}:=\textnormal{ker}\big(r_{V_{\underline{a}}}-\phi_{\underline{a},i}\circ r_{W_{\underline{a}}}\big)$ is locally free, because $\mathcal{V}_{\underline{a}}|_{D_i}$ has homological dimension 1. This holds also for the corresponding family on $Y\times \mathcal{T}_{Y,D}\times \mathcal{T}_{Y,D}$ by the same argument. The following is a simple consequence of the construction. 
\begin{lemma}
\label{propositiondiagramofresolutions}
We have the quasi-isomorphisms 
\begin{equation*}
    \begin{tikzcd}[column sep=large]
           \arrow[d,"\pi_{V_{\underline{a}-e_i}}"]V_{\underline{a}}\oplus W_{\underline{a}}\arrow[r,"{(f_{V_{\underline{a}}},f_{W_{\underline{a}}})}"]&K_{\underline{a},i}\arrow[d,"p_{V_{\underline{a}}}"]\\
           V_{\underline{a}-e_i}\arrow[r,"s_i"]&V_{\underline{a}}
    \end{tikzcd}\,,\quad   \begin{tikzcd}[column sep=large]
           \arrow[d,"\pi_{W_{\underline{a}-e_i}}"]V_{\underline{a}-e_i}\oplus W_{\underline{a}-e_i}\arrow[r,"{(f_{V_{\underline{a}}},f_{W_{\underline{a}}})}"]&K_{\underline{a},i}\arrow[d,"p_{W_{\underline{a}}}"]\\
           W_{\underline{a}-e_i}\arrow[r,"s_i"]&W_{\underline{a}}
    \end{tikzcd}\,.
\end{equation*}
Using $C_{\underline{a}}$ to denote both upper cones and $C_{V_{\underline{a}}}$, $C_{W_{\underline{a}}}$ to denote the lower cones, this gives a commutative diagram
$
\begin{tikzcd}
\arrow[d]C_{\underline{a}}\arrow[r]& \arrow[d] C_{V_{\underline{a}}}\\
C_{W_{\underline{a}}}\arrow[r]&W_{\underline{a},i}
\end{tikzcd}
$ of quasi-isomorphisms.
	\end{lemma}
Therefore \eqref{restriction} becomes 
	\begin{align*}
   &\textnormal{det}\big(\textnormal{R}\Gamma^\bullet(V_{\underline{a}}) \big)\otimes\textnormal{det}^*\big(R\Gamma^\bullet(V_{\underline{a}-e_i})\big)\otimes \textnormal{det}(R\Gamma^\bullet(W_{\underline{a}-e_i}))\otimes \textnormal{det}^*(R\Gamma^\bullet(W_{\underline{a}})\big)\\
  &\cong \textnormal{det}\big(R\Gamma^\bullet\big(C_{\underline{a}}\big)\big)\otimes \textnormal{det}^*\big(R\Gamma^\bullet\big(C_{\underline{a}}\big)\Big)\cong \mathbb{C}\,.
\end{align*}
Using compatibility with respect to different filtrations discussed in \autocite[p. 22]{knumum} and that a dual of an evaluation is a coevaluation in the monoidal category of line bundles together with checking the correct signs one can show that this is expressed as 
\begin{align*}
\label{algebraic isomorphism}
   &\textnormal{det}(R\Gamma^\bullet(V_{\underline{a}}))\otimes\textnormal{det}^*(R\Gamma^\bullet(V_{\underline{a}-e_i}))\otimes\textnormal{det}(R\Gamma^\bullet(W_{\underline{a}-e_i})\otimes\textnormal{det}^*\big(R\Gamma^\bullet\big(W_{\underline{a}}\big)\big)\\
   \cong 
   &\begin{subarray}
   \ \textnormal{det}\big(R\Gamma^\bullet\big(V_{\underline{a}}\big)\big)\otimes\textnormal{det}^*\big(R\Gamma^\bullet(V_{\underline{a}-e_i})\big)\otimes\textnormal{det}^*\big(R\Gamma^\bullet(K_{\underline{a},i})\big)\otimes\textnormal{det}\big(R\Gamma^\bullet(V_{\underline{a}-e_i}\oplus W_{\underline{a}-e_i})\big)
   \\
  \textnormal{det}^*\big(R\Gamma^\bullet(V_{\underline{a}-e_i}\oplus W_{\underline{a}-e_i})\big)\otimes \textnormal{det}\big(R\Gamma^\bullet(K_{\underline{a},i})\big)\otimes \textnormal{det}\big(R\Gamma^\bullet(W_{\underline{a}-e_i})\big)\otimes\textnormal{det}^*\big(R\Gamma^\bullet(W_{\underline{a}})\big)\end{subarray}\cong
 \mathbb{C}\,,
 \numberthis
\end{align*}
where the last isomorphism is the consequence of the following short exact sequences:
\begin{align*}
\label{global exact sequences}
    0\longrightarrow V_{\underline{a}-e_i}\oplus W_{\underline{a}-e_i}\longrightarrow K_{\underline{a},i}\oplus V_{\underline{a}-e_i}\longrightarrow V_{\underline{a}}\longrightarrow 0\,,\\
    0\longrightarrow  V_{\underline{a}-e_i}\oplus W_{\underline{a}-e_i}\longrightarrow K_{\underline{a},i}\oplus W_{\underline{a}-e_i}\longrightarrow W_{\underline{a}}\longrightarrow 0\,.
    \numberthis
\end{align*}
Let $\bar{\partial}_1,\bar{\partial}_2,\bar{\partial}_3$ be the holomorphic structures on each of the three terms of the first sequence, then choosing its splitting we can define $\bar{\partial}^{t}_2=\begin{pmatrix}
\bar{\partial}_1 &t\bar{\partial}_{(1,1)}\\
0&\bar{\partial}_3
\end{pmatrix}$ and deforming to $t=0$ the sequence splits. This gives us the following diagram commuting up to isotopy:
$$
\begin{tikzcd}
\arrow[d]\textnormal{det}\big(R\Gamma^\bullet(V_{\underline{a}-e_i}\oplus W_{\underline{a}-e_i})\big)\textnormal{det}(R\Gamma^\bullet(V_{\underline{a}}))\arrow[r,labels={inner sep=7},"\textnormal{\cite[Corollary 2]{knumum}}"]&\textnormal{det}\big(R\Gamma^\bullet(K_{\underline{a},i}\oplus V_{\underline{a}-e_i})\big)\arrow[d]\\
\textnormal{det}\big(D_{V_{\underline{a}-e_i}\oplus W_{\underline{a}-e_i}\oplus V_{\underline{a}}}\big)\arrow[r,"*"]&\textnormal{det}\big(D_{K_{\underline{a},i}\oplus V_{\underline{a}-e_i}}\big)\,,
\end{tikzcd}
$$
(and a similar one for the second sequence) where $*$ corresponds to the isomorphism \eqref{functoriality} for
\begin{equation}
\label{total isom}
    \begin{pmatrix}
f_{V_{\underline{a}}}&f_{W_{\underline{a}}}&p_{V_{\underline{a}}}^*\\
0&\textnormal{id}&s_i^*
\end{pmatrix}: V_{\underline{a}-e_i}\oplus W_{\underline{a}-e_i}\oplus V_{\underline{a}}\longrightarrow K_{\underline{a},i}\oplus V_{\underline{a}-e_i}\,.	
\end{equation}
Recall the tubular neighborhoods $\ov{S}_i$ from Lemma \ref{lemma extend} and denote their interiors by $S_i$. In Lemma \ref{lemma phi_K} below, we show that there exists a natural isomorphism $\Phi_{K_{\underline{a},i}}:K_{\underline{a},i}|_{S_i}\to K_{\underline{a},i}|_{S_i}$, such that all the diagrams below satisfy the conditions in Definition \ref{definition excision I2 family}. We therefore get
\begin{equation*}
 \begin{tikzcd}[ column sep=4cm, row sep=2cm, ampersand replacement=\&]
   \arrow[d,dashrightarrow,"{\begin{psmallmatrix}
f_{V_{\underline{a}}}&f_{W_{\underline{a}}}&p_{V_{\underline{a}}}^*\\
0&\textnormal{id}& s_i^*
\end{psmallmatrix}}"']\sigma_{V_{\underline{a}-e_i}\oplus W_{\underline{a}-e_i}\oplus V_{\underline{a}}}\arrow[r,dashrightarrow,"{\begin{psmallmatrix}
   0&i\Phi^{-1}_{\underline{a}-e_i,i}&0\\
   i\Phi_{\underline{a}-e_i,i}&0&0\\
   0&0&i\Phi_{\underline{a},i}
   \end{psmallmatrix}}"]\&\sigma_{V_{\underline{a}-e_i}\oplus W_{\underline{a}-e_i}\oplus W_{\underline{a}}}\arrow[d,dashrightarrow,"{\begin{psmallmatrix}
f_{V_{\underline{a}}}& f_{W_{\underline{a}}}& p_{V_{\underline{a}}}^*\\
0&\textnormal{id}& s_i^*
\end{psmallmatrix}}"]\\
   \sigma_{K_{\underline{a},i}\oplus V_{\underline{a}-e_i}}\arrow[r,dashrightarrow,"{\begin{psmallmatrix}i\Phi_{K_{\underline{a},i}}&0\\
   0&i\Phi_{\underline{a}-e_i}
   \end{psmallmatrix}}"']\&\sigma_{K_{\underline{a},i}\oplus W_{\underline{a}-e_i}}
   \end{tikzcd} 
\end{equation*}
We may restrict \eqref{total isom} to $X\backslash (1-\epsilon)\bar{S}_i$ because we already cover $S_i$ by the other isomorphisms in the diagram. We choose the compact set $K_i=(1-\epsilon/2)S_i$. Then deform $t\mapsto{\begin{psmallmatrix}
f_{V_{\underline{a}}}&f_{W_{\underline{a}}}&tp_{V_{\underline{a}}}^*\\
0&t\textnormal{id}& s_i^*
\end{psmallmatrix}}$ as these are now isomorphisms in $X\backslash (1-\epsilon)S_i$. Moreover, rotating
$$
t\mapsto \begin{pmatrix}
   \textnormal{sin}(t)\,\textnormal{id}&i\textnormal{cos}(t)\,\Phi^{-1}_{\underline{a}-e_i,i}&0\\
  i\textnormal{cos}(t) \Phi_{\underline{a}-e_i,i}&\textnormal{sin}(t)\,\textnormal{id}&0\\
   0&0&\Phi_{\underline{a},i}
   \end{pmatrix}, \begin{pmatrix}i\textnormal{cos}(t)\,\Phi_{K_{\underline{a},i}}+\textnormal{sin}(t)\,\textnormal{id}&0\\
   0&\Phi_{\underline{a}-e_i}
   \end{pmatrix}
$$
we obtain the separate two diagrams
\begin{equation*}
 \begin{tikzcd}[ column sep=2cm, row sep=1cm, ampersand replacement=\&]
   \arrow[d,dashrightarrow,"{\begin{psmallmatrix}
f_{V_{\underline{a}}}&f_{W_{\underline{a}}}
\end{psmallmatrix}}"']\sigma_{V_{\underline{a}-e_i}\oplus W_{\underline{a}-e_i}}\arrow[r,dashrightarrow,"{\begin{psmallmatrix}
   \textnormal{id}&0\\
   0&\textnormal{id}\\
 \end{psmallmatrix}}"]\&\sigma_{V_{\underline{a}-e_i}\oplus W_{\underline{a}-e_i}}\arrow[d,dashrightarrow,"{\begin{psmallmatrix}
f_{V_{\underline{a}}}& f_{W_{\underline{a}}}
\end{psmallmatrix}}"]\\
   \sigma_{K_{\underline{a},i}}\arrow[r,dashrightarrow,"\textnormal{id}"']\&\sigma_{K_{\underline{a},i}}
   \end{tikzcd} \,,
    \begin{tikzcd}[ column sep=2cm, row sep=1cm, ampersand replacement=\&]
   \arrow[d,dashrightarrow,"s_i^*"']\sigma_{V_{\underline{a}-e_i}}\arrow[r,dashrightarrow,"\Phi_{\underline{a}-e_i,i}"]\&\sigma_{W_{\underline{a}-e_i}}\arrow[d,dashrightarrow,"s_i^*"]\\
   \sigma_{V_{\underline{a}}}\arrow[r,dashrightarrow,"\Phi_{\underline{a},i}"']\&\sigma_{W_{\underline{a}}}
   \end{tikzcd} \end{equation*}
In the left diagram we can extend the identities to all of $Y$, so by Lemma \ref{upmeier}, we showed that \eqref{algebraic isomorphism} coincides with the adjoint of $\Xi$ from Lemma \ref{upmeier} for the right diagram. Using 
this for each step $\underline{k}=\sum_i k_i e_i$, deforming the isomorphisms $\Phi_{\underline{a},i}$ on $S_i$ into $\Phi_{\underline{a}}$ on $T$ and taking $K=\bigcap_{i=1}^NK_i$, we obtain \eqref{comparing excision} for the data $\bowtie$ because $\prod (s_{i,k})^*\prod (t_{j,k})^{-*}=\Omega^*$.  The second diagram in \eqref{comparing excision} is obtained by pulling back along $\Delta: T_{Y,D}\to T_{Y,D}\times T_{Y,D}$.
	\end{proof}
	\begin{remark}
	We could replace in \eqref{comparing excision} the labels $\underline{k}$, $\underline{0}$ by arbitrary $\underline{a}$, $\underline{b}$.
	\end{remark}
	\begin{lemma}
\label{lemma phi_K}	
There exists a natural isomorphism $\Phi_{K_{\underline{a},i}}:K_{\underline{a},i}|_{S_i}\to K_{\underline{a},i}|_{S_i}$, such that $i\textnormal{cos}(t)\Phi_{K_{\underline{a}},i}+\textnormal{sin}(t)\textnormal{id}_{K_{\underline{a},i}}$ are invertible for all $t$ and such that all diagrams used in the proof of Proposition \ref{proposition comparing excision} satisfy the condition of Definition \ref{definition excision I2 family}.
	\end{lemma}
	\begin{proof}
	Using the octahedral axiom, one can show that there are short exact sequences $0\to V_{\underline{a}-e_i}\xrightarrow{f_{V_{\underline{a}}}}K_{\underline{a},i}\xrightarrow{p_{W_{\underline{a}}}}W_{\underline{a}}\to 0$ and $0\to W_{\underline{a}-e_i}\xrightarrow{f_{W_{\underline{a}}}}K_{\underline{a},i}\xrightarrow{p_{V_{\underline{a}}}}V_{\underline{a}}\to 0$. We obtain the following commutative diagram
	$$
\begin{tikzcd}
0\arrow[r]&\arrow[d,"-\phi_{\underline{a}-e_i,i}"] V_{\underline{a}-e_i}|_{D_i}\arrow[r]& K_{\underline{a},i}|_{D_i}\arrow[r]\arrow[d,"\textnormal{id}"]& W_{\underline{a}}|_{D_i}\arrow[d,"\phi_{\underline{a},i}"]\arrow[r]& 0\\
0\arrow[r]& W_{\underline{a}-e_i}|_{D_i}\arrow[r]& K_{\underline{a},i}|_{D_i}\arrow[r]& V_{\underline{a}}|_{D_i}\arrow[r]& 0
\end{tikzcd}
	$$
	as can be seen from the following diagram:
	$$
	\begin{tikzcd}
    0\arrow[r]&V_{\underline{a}-e_i,i}\arrow[d,"-{\phi_{\underline{a}}-e_i,i}"]\arrow[r]&K_{\underline{a},i}|_{D_i}\arrow[r]\arrow[d,"\textnormal{id}"]&\arrow[d,"\textnormal{id}"]V_{\underline{a},i}\oplus W_{\underline{a},i}\arrow[r]&\arrow[d]V_{\underline{a},i}\arrow[d,"-{\phi_{\underline{a},i}}"]\arrow[r]&0\,,
\\
    0\arrow[r]&W_{\underline{a}-e_i,i}\arrow[r]&K_{\underline{a},i}|_{D_i}\arrow[r]&V_{\underline{a},i}\oplus W_{\underline{a},i}\arrow[r]&W_{\underline{a},i}\arrow[r]&0
	\end{tikzcd}
	$$
	induced by restricting
	$$
	\begin{tikzcd}
    0\arrow[r]&K_{\underline{a},i}|_{D_i}\arrow[r]\arrow[d,"\textnormal{id}"]&\arrow[d,"\textnormal{id}"]V_{\underline{a},i}\oplus W_{\underline{a},i}\arrow[r]&\arrow[d]V_{\underline{a},i}\arrow[d,"-{\phi_{\underline{a},i}}"]\arrow[r]&0
\\
    0\arrow[r]&K_{\underline{a},i}|_{D_i}\arrow[r]&V_{\underline{a},i}\oplus W_{\underline{a},i}\arrow[r]&W_{\underline{a},i}\arrow[r]&0
	\end{tikzcd}
	$$
to the divisor. Choosing a splitting of the first exact sequence in $S_i$. we obtain $(f_{V_{\underline{a}}},f_{W_{\underline{a}}})=\begin{psmallmatrix}\textnormal{id}&-\Phi'_{\underline{a}-e_i,i}\\
0&s_i\end{psmallmatrix}$, where $\Phi'_{\underline{a}-e_i,i}|_{D_i} = \phi_{\underline{a}-e_i,i}$, so we can take $\Phi_{\underline{a}-e_i,i} =\Phi'_{\underline{a}-e_i,i}$. This induces also the splitting of the second sequence in $S_i$ and we can define the isomorphism by
\begin{equation*}
\begin{tikzcd}[column sep=2cm, ampersand replacement=\&]
\Phi_{K_{\underline{a},i}}: K_{\underline{a},i}|_{S_i}\cong W_{\underline{a}-e_i}|_{S_i}\oplus V_{\underline{a}}|_{S_i}\arrow[r, "{\begin{psmallmatrix}\Phi_{\underline{a}-e_i}&0\\
0&\Phi_{\underline{a}}\end{psmallmatrix}}"]\& V_{\underline{a}-e_i}|_{S_i}\oplus W_{\underline{a}}\cong K_{\underline{a},i}|_{S_i}\,.
\end{tikzcd}
\end{equation*}
In this splitting, we then have the invertible isomorphisms
$
i\textnormal{cos}(t)\Phi_{K_{\underline{a},i}}+\textnormal{sin}(t)\textnormal{id} =\begin{psmallmatrix}ie^{it}\Phi_{\underline{a}-e_i}&0\\
s_i&ie^{-it}\Phi_{\underline{a}}\end{psmallmatrix}
$
where we used $\textnormal{id}_{K_{s_i}}=(f_{V_{\underline{a}}},f_{W_{\underline{a}}})\circ (f_{V_{\underline{a}}},f_{W_{\underline{a}}})^{-1}= \begin{psmallmatrix}-\Phi_i^{-1}&0\\
s_i& \Phi_i\end{psmallmatrix}$
and one can check directly, these satisfy the necessary commutativity for all steps needed in Proposition \ref{proposition comparing excision}.
	\end{proof}
\begin{proposition}
\label{proposition main}
There are natural isomorphisms 
\begin{equation}
\label{eq:vector-bundle-isom}
\Lambda^*(D_O)\cong (\Omega^{\textnormal{ag}})^*(O^{\bowtie})
\end{equation}
of weak $\Z_2$-graded H-principal $\Z_2$-bundles independent of $\mathfrak{ord}$. If $\tilde{Y}$ additionally satisfies Assumption \ref{ass:astcs}, this becomes an isomorphism of strong $\Z_2$-graded H-principal $\Z$-bundles. 
\end{proposition}
\begin{proof}
Let $O^{\bowtie}_{\textnormal{dg}}$ be the $\Z_2$-bundle associated to
$$\vartheta^{\textnormal{dg}}_{\bowtie} =  \big(\pi^*_1(\#_{\underline{k}})\otimes \pi_2^*(\#_{\underline{k}})^{-1}\big)\circ (\tau^{\textnormal{dg}}_{\bowtie})^{-1}:\pi^*_1(\Lambda^{\textnormal{dg}}_{\underline{0}})\otimes  \pi_2^*(\Lambda^{\textnormal{dg}}_{\underline{0}})^*\to \pi^*_1(\Lambda^{\textnormal{dg}}_{\underline{0}})^*\otimes  \pi_2^*(\Lambda^{\textnormal{dg}}_{\underline{0}})\,,$$
then it is a strong H-principal $\Z_2$-bundle by similar arguments to \ref{proposition Ovartheta Hprincipal} using Lemma \ref{lemma dg orientations} and part 2 of Lemma \ref{excision}, where we obtain similarly as in \eqref{strongHsigns} additional signs $(-1)^{\textnormal{deg}(\pi_2^*\Lambda^{\textnormal{dg}}_{\underline{0}})\big(\textnormal{deg}(\pi_3^*\Lambda^{\textnormal{dg}}_{\underline{0}})+\textnormal{deg}(\pi^*_4\Lambda^{\textnormal{dg}}_{\underline{0}})\big)}$.

We have the isomorphisms of weak H-principal $\Z_2$-bundles
$$
\Lambda^*(O^{\bowtie})\cong O^{\bowtie}_{\textnormal{dg}}\cong D_O\,,
$$
where the first isomorphism is constructed by applying Proposition \ref{proposition comparing excision} and Lemma \ref{upmeier} and is clearly strong H-principal when the bundles are orientable. For the second one, let us again restrict ourselves to the point $(p,q)$ as in the proof of Proposition \ref{comparing excision}. Note that for all $t$ we have isomorphisms:
$$
\begin{tikzcd}[column sep=huge]
  \arrow[d,dashed]\textnormal{det}(\sigma_{V_{\underline{0}}},\sigma_{W_{\underline{0}}},\bar{\Phi},\bar{\Phi})^t_{\chi} \arrow[r,"{\textnormal{Prop. \ref{proposition functoriality ds adj}(v)}}"]&\textnormal{det}(\sigma_{V_{\underline{0}}},\sigma_{W_{\underline{0}}},\bar{\Phi},\bar{\Phi})^1_{\chi}\arrow[d,"\pi_1^*(\Omega)\otimes \pi_2^*(\Omega^{-*})"] \\
  \arrow[d,"\eqref{antilinear maps}"] \textnormal{det}(\sigma_{V_{\underline{0}}},\sigma_{W_{\underline{0}}},\bar{\Phi},\bar{\Phi})^t_{\chi} \arrow[r,"{\textnormal{Prop. \ref{proposition functoriality ds adj}(v)}}"]&\textnormal{det}(\sigma_{V_{\underline{k}}},\sigma_{W_{\underline{k}}},\bar{\Phi},\bar{\Phi})^1_{\chi}\arrow[d,"\eqref{antilinear maps}"]\\
  \textnormal{det}^*(\sigma_{V_{\underline{0}}},\sigma_{W_{\underline{0}}},\bar{\Phi},\bar{\Phi})^t_{\chi} \arrow[r,"{\textnormal{Prop. \ref{proposition functoriality ds adj}(v)}}"]&\textnormal{det}^*(\sigma_{V_{\underline{0}}},\sigma_{W_{\underline{0}}},\bar{\Phi},\bar{\Phi})^1_{\chi}
\end{tikzcd}
$$
The composition of the vertical arrows on the left for $t=0$ is precisely $\vartheta^{\textnormal{dg}}_{\bowtie}$, while the associated $\Z_2$-torsor to the real structure corresponding to vertical arrows on the right is by excision as in Lemma \ref{lemma independent} isomorphic to $D_O|_{(p,q)}$ using that $\bar{\Phi}$ are unitary and  \eqref{antilinear maps} restrict outside of $\epsilon \bar{S}$ from Lemma \ref{lemma extend} to \cite[Def. 3.24]{GJT}. As the diagram is commutative for all $t$, we obtain the required isomorphisms $D_O|_{(p,q)}\cong O^{\bowtie}_{\textnormal{dg}}$. When the bundles are trivializable, it is an isomorphisms of strong H-principal $\Z_2$-bundles by the compatibility under direct sums of Lemma \ref{excision} and using the arguments of \cite[Prop. 3.25]{CGJ} (see also proof of Lemma \ref{lemma independent}) together with paying attention to the signs above. The independence of  any choice of $\mathfrak{ord}$ can then be shown  because diagram \eqref{comparing excision} commutes and both $\sigma^{\textnormal{vb}}_{\bowtie}$ and $\sigma^{\textnormal{dg}}_{\bowtie}$ are independent (see proof of Theorem \ref{mainnctheorem}). We only sketch the idea, as the precise formulation is comparably more tedious than the proof of Proposition \ref{proposition comparing excision}: one uses excision on the diagram \eqref{3x3diag} using the common resolutions of Lemma \ref{propositiondiagramofresolutions}, where the top right corner of \eqref{3x3diag} has resolution $V_{\underline{0}}\oplus W_{\underline{0}}\to K_{e_2,2}$, bottom left $V_{\underline{0}}\oplus W_{\underline{0}}\to K_{e_1,1}$ and the bottom right one $$(V_{\underline{0}}\oplus W_{\underline{0}}\to K_{e_1,1}\oplus V_{e_2}\oplus W_{e_2}\to K_{e_1+e_2,1})\cong (V_{\underline{0}}\oplus W_{\underline{0}}\to K_{e_2,2}\oplus V_{e_1}\oplus W_{e_1}\to K_{e_1+e_2,2})\,.$$  The automorphism on $K_{e_1,1}|_{S_1}$, $K_{e_2,2}|_{S_2}$, $K_{e_1+e_2,1}|_{S_1}$, $K_{e_1+e_2,2}|_{S_i}$ are then the ones constructed in Lemma \ref{lemma phi_K} and one follows the arguments of Proposition \ref{proposition comparing excision} to remove contributions of all $K$'s.
\end{proof}
\subsection{Imposing the restriction from \cite[($\ast$)]{JU}}
\label{sec:restriction}
Applying Proposition \ref{exun}(i) to  the isomorphism \eqref{eq:vector-bundle-isom}, we obtain an isomorphism \eqref{isomorphism DCO Ovartheta} where $D^{\mathcal{C}}_{O}$ was defined in \eqref{DO}. Moreover, if $\tilde{Y}$ satisfies the condition \cite[($\ast$)]{JU} recalled in Assumption \ref{ass:astcs}, then the isomorphism is canonical and $D^{\mathcal{C}}_{O}, O^{\bowtie}$ are trivializable. To understand when the stronger statements in Proposition \ref{proposition main} and therefore Theorem \ref{maintheorem} hold, we prove the following lemma.
\begin{lemma}
\label{lem:assumptionfortildeY}
If $X$ satisfies 
\begin{equation}
\label{eq:H^3-restrictions}
H^3(X,\Z_2) = 0 = H^3_{\cs}(X,\Z_2)
\end{equation}
then Assumption \ref{ass:astcs} holds for the associated spin-manifold $\tilde{Y}$ from Definition \ref{definition double}.
\end{lemma}
\begin{proof}
    We will prove that the above conditions imply that 
    $$
    H^3(\tilde{Y},\Z_2) =0\,.
    $$ 
    Note that, by the construction in Definition \ref{def:UandT-defined}, the interior $U$ is homotopy equivalent to $X$ so the assumption \eqref{eq:H^3-restrictions} applies to it also. We use the Mayer-Vietoris long exact sequence
    $$
    \cdots \to  H^2(U,\Z_2)^{\oplus 2} \xrightarrow{z} H^2(\partial U,\Z_2)\to H^3(\tilde{Y},\Z_2)\to H^3(U,\Z_2)^{\oplus 2}\cong 0\to \cdots 
    $$
    which shows that $ H^3(\tilde{Y},\Z_2) \cong \operatorname{coker}(\alpha)$. Applying the long exact sequence of the pair $(U,\partial U)$, we get
    $$
    \cdots \to H^2(U,\Z_2)\xrightarrow{w} H^2(\partial U,\Z_2)\to H^3(U,\partial U,\Z_2)\to H^3(U,\Z_2)\cong 0\to\cdots\,.
    $$
    Since $w$ is the restriction along $\partial U\hookrightarrow U$ and $z$ is the difference of these restrictions, we see that 
    $$
    H^3(\tilde{Y},\Z_2)\cong \operatorname{coker}(w)\cong \operatorname{coker}(z)\cong H^3(U,\partial U,\Z_2)\,.
    $$
    Since $(U/\partial U, [\partial U])$ is homotopy equialent to $(Y/D, [D])$, we see that $H^3_{\cs}(X,\Z_2)=0$ implies $H^3(\tilde{Y},\Z_2) = 0$.
\end{proof}
However, we do not need the entire $\tilde{Y}$ to satisfy Assumption \ref{ass:astcs} to prove orientability of $O^{\omega}$ in Theorem \ref{mainnctheorem} and orientability of $O^{\omega}_M$ in Theorem \ref{theorem stable pairs}. The next argument was worked out after I learnt of the formulation in terms of Assumption \ref{ass:astcs} from Ivan Karpov. Due to the gluing procedure described in \eqref{gluing equation}, we see that the composition $$C^{\cs}_X\xrightarrow{\kappa^{\cs}} \mC_{Y,D}\to \mC_{\tilde{Y},\tilde{T}}$$
factors through
$$
C^{\cs}_X\to \Z\times_{\mC_{\tilde{T}}}\mC_{\tilde{Y}}
$$
where the map $\Z\to \mC_{\tilde{T}}$ is given by $n\mapsto n\cdot \llbracket \un{\C}\rrbracket $. It is clear that $\Z\times_{\mC_{\tilde{T}}}\mC_{\tilde{Y}}$ is a homotopy-theoretic group completion of $\Z_{\geq 0}\times_{\mV_{\tilde{T}}}\mV_{\tilde{Y}}$. Since $O^{\cs}$ is defined by \eqref{Ocs}, it is sufficient to show that the orientation bundle $D_O(\tilde{Y})$ from \eqref{DO bundle} is trivializable when restricted to $\Z_{\geq 0}\times_{\mV_{\tilde{T}}}\mV_{\tilde{Y}}$. This will imply that this restriction is a $\Z_2$-graded strong H-principal $\Z_2$-bundle and so is the restriction of $D^{\mathcal{C}}_O$ to $\Z\times_{\mC_{\tilde{T}}}\mC_{\tilde{Y}}$ by Proposition \ref{exun}(ii). This also clearly works for $M^{\an}$, where $M$ is the moduli space from Theorem \ref{theorem stable pairs}.

From the perspective of $B_{\tilde{Y},\tilde{T}}$, the above restriction corresponds to looking at connections which are trivialized on $\tilde{T}$. We can replace this condition by instead considering  connections trivial outside of an open subset of $U$. In the statement of \cite[Theorem 3.2 (c)]{JU}, this means that the relevant classes\footnote{The $\alpha$ and $\beta$ used there are cohomology classes which clashes with our notation slightly as we use $\alpha,\beta$ to label K-theory classes. This issue is only relevant to this paragraph.} $\alpha$ and $\beta$ lie in the images of $$H^4(U\times S^1, \partial U\times S^1,\Z)\to H^4(\tilde{Y}\times S^1,\Z)\,,\quad H^3(U,\partial U,\Z)\to H^3(\tilde{Y},\Z)\,,$$ respectively. Write $\alpha_{\cs},\beta_{\cs}$ for their lifts along these morphisms and identify $ H^3(U,\partial U,\Z)\cong H^3_{\cs}(U^{\circ},\Z)$ where $U^{\circ}$ is the interior of $U$. Then, due to naturality of all relevant operations, the integral \cite[(3.22)]{JU} can be written as 
\begin{equation}
\label{eq:integraltovanish}
\int_{U^{\circ}}  \ov{\beta_{\cs}}\cup \Sq^2_{\cs}(\ov{\beta_{\cs}})
\end{equation}
where $\ov{\beta_{\cs}}\in H^3_{\cs}(U^{\circ},\Z_2)$ is the projection to $\Z_2$ coefficients. In particular, this obstruction always vanishes when Assumption \ref{ass:astcs} is satisfied by $U^{\circ}$. By the construction in Definition \ref{def:UandT-defined}, there is a homotopy equivalence $(Y/D,[D])\sim (U/\partial U,[\partial U])$. This restates \eqref{eq:integraltovanish} in terms of Assumption \ref{ass:astcs} and therefore proves Theorem \ref{mainnctheorem} and \ref{theorem stable pairs}.	
\section{Orientation groups for non-compact Calabi--Yau fourfolds}
\label{bigsection signs}
In this final section, we describe the behavior of orientations under direct sums. We recall the notion of orientation group from \autocite{JTU} and formulate the equivalent version of \autocite[Theorem 2.27]{JTU} for the non-compact setting, where we replace K-theory with compactly supported K-theory. For background on compactly supported cohomology theories, see Spanier \autocite{Spanier}, Ranicki--Roe \autocite[§2]{RaRo}. From the algebraic point of view, see the discussion in Joyce--Song \autocite[§6.7]{JoyceSong} and Fulton \autocite[§18.1]{Fultonintersection}. Throughout the section, we always assume that Assumption \ref{ass:astcs} is satisfied by $\tilde{Y}$, which is implied by $H^3(X,\Z_2) = 0 = H^3_{\cs}(X,\Z_2)$ due to Lemma \ref{lem:assumptionfortildeY}.
\subsection{Orientation on compactly supported K-theory}

In \eqref{DO}, we define the strong H-principal $\mathbb{Z}_2$-bundle  $D^{\mathcal{C}}_O\to \mathcal{C}_{Y,D}$. We first describe its commutativity rules as in \autocite[Definition 2.22]{JTU}.
\begin{definition}
\label{DCO H principal map}
Let $\mu_{\mathcal{C}}: \mathcal{C}_{Y,D}\times \mathcal{C}_{Y,D}\to \mathcal{C}_{Y,D},\mu_{\textnormal{cs}}:\mathcal{C}^{\textnormal{cs}}\times \mathcal{C}^{\textnormal{cs}}\to \mathcal{C}^{\textnormal{cs}}$ be the binary maps and 
\begin{equation}
\label{tautau}
    \tau:D^{\mathcal{C}}_O\boxtimes D^{\mathcal{C}}_O\longrightarrow \mu_{\mathcal{C}}^*(D^\mathcal{C}_O)\,,\quad \tau^{\textnormal{cs}}:O^{\textnormal{cs}}\boxtimes O^{\textnormal{cs}}\longrightarrow \mu^*_{\textnormal{cs}}(O^{\textnormal{cs}})
\end{equation}
be the isomorphisms of $\Z_2$-bundles on $\mathcal{C}_{Y,D}\times \mathcal{C}_{Y,D}$, which make $(D^{\mathcal{C}}_O,\tau)$ into a strong H-principal $\Z_2$-bundle and $\tau^{\textnormal{cs}}$ a restriction of $\tau$.
\end{definition}
We recall the notion of Euler-form as defined in Joyce--Tanaka--Upmeier \autocite[Definition 2.20]{JTU} for real elliptic differential operators.

\begin{definition}
\label{definition differential euler form}
Let $X$ be a smooth compact manifold, $E_0,E_1$ vector bundles on $X$ and $P:E_0\to E_1$ a real or complex elliptic differential operator. Let $E,F\to X$ be complex vector bundles, the Euler form $\chi_P: K^0(X)\times K^0(X)\to\Z$ is defined by
$$
\chi_P(\llbracket E \rrbracket,\llbracket F \rrbracket) =  \textnormal{ind}_{\mathbb{C}}(\sigma(P)\otimes \textnormal{id}_{\pi^*(\textnormal{Hom}(E,F))})
$$
together with bi-additivity of $\chi_P$. We used the notation from §\ref{section pseudo} for symbols of operators. If $X$ is spin and $P=\slashed{D}_+$, we write $\chi^{\R}_X := \chi_{\slashed{D}}$. Similarly, if $X$ is a complex manifold and $D =\bar{\partial} + \bar{\partial}^*: A^{0,\textnormal{even}}\to A^{0,\textnormal{odd}}$ is the Dolbeault operator, then we use $\chi_X := \chi_D$.
\end{definition}

 Recall that we have a comparison map $c: G_0(X) \to  K^0(X)$, where $G_0(X)$ is the Grothendieck group associated to $D^b\big(\textnormal{Coh}(X)\big)$.  Let $\chi^{\textnormal{ag}}_X:K_0(X)\times K_0(X)\to \Z$ be defined by 
$
\chi^{\textnormal{ag}}_X(E,F) = \sum_{i\in \Z}(-1)^i \textnormal{dim}_\C(\textnormal{Ext}^i(E,F)),
$
then when $X$ is smooth , we have $\chi_X\circ (c\times c) = \chi^{\textnormal{ag}}_X$.

\begin{proposition}
\label{proposition sign}
Assume that $\tilde{Y}$ from Definition \ref{def:UandT-defined} satisfies Assumption \ref{ass:astcs}. Let $i_1,i_2: Y\to Y\cup_{D}Y$ be the inclusions of the two copies of $Y$
and $\delta(\alpha,\beta)\in K^0(Y\cup_{D}Y)$ denote a K-theory class, such that $i_1^*(\delta(\alpha,\beta)) = \alpha$ and $i_2^*(\delta(\alpha,\beta)) = \beta$. We have the bijection $\pi_0(\mathcal{C}_{Y,D})=K^0(Y\cup_{D}Y).$
Let $\mathcal{C}_{\delta(\alpha,\beta)}$ be the components corresponding to $\delta(\alpha,\beta)$ and $D^{\mathcal{C}}_O|_{\delta(\alpha,\beta)}$ the restriction of $D^{\mathcal{C}}_O$ to it.  Suppose a choice of trivialization $o_{\delta(\alpha,\beta)}:\Z_2\to D^{\mathcal{C}}_O|_{\delta(\alpha,\beta)}$ is given for each $\delta(\alpha,\beta)\in K^0(Y\cup_{D}Y)$, then define
$\epsilon_{\delta(\alpha_1,\beta_1),\delta(\alpha_2,\beta_2)}\in \{-1,1\}$ by
$$
\tau\big(o_{\delta(\alpha_1,\beta_1)}\boxtimes_{\Z_2}o_{\delta(\alpha_2,\beta_2)}\big) = \epsilon_{\delta(\alpha_1,\beta_1),\delta(\alpha_2,\beta_2)}o_{\delta(\alpha_1,\beta_1) + \delta(\alpha_2,\beta_2)}
$$
These signs satisfy:
\begin{align*}
\label{compatibilitysignsdg}
\epsilon_{\delta(\alpha_2,\beta_2),\delta(\alpha_1,\beta_1)}=&(-1)^{\big(\chi_{Y}(\alpha_1,\alpha_1) - \chi_{Y}(\beta_1,\beta_1)\big)\big(\chi_{Y}(\alpha_2,\alpha_2) - \chi_{Y}(\beta_2,\beta_2)\big) + \chi_{Y}(\alpha_1,\alpha_2) - \chi_{Y}(\beta_1,\beta_2)}\\
&\epsilon_{\delta(\alpha_1,\beta_1),\delta(\alpha_2,\beta_2)}\,.\\
\epsilon_{\delta(\alpha_1,\beta_1),\delta(\alpha_2,\beta_2)}&\epsilon_{\delta(\alpha_1+\alpha_2,\beta_1+\beta_2),\delta(\alpha_3,\beta_3)}
=\epsilon_{\delta(\alpha_2,\beta_2),\delta(\alpha_3,\beta_3)}
\epsilon_{\delta(\alpha_2+\alpha_3,\beta_2+\beta_3),\delta(\alpha_1,\beta_1)}\,.
\numberthis
\end{align*}
Let $(M_{\gamma(\alpha,\beta)})^{\textnormal{top}} = \Gamma^{-1}(\mathcal{C}_{\gamma(\alpha,\beta)})$, $(O^{\bowtie}_{\gamma(\alpha,\beta)})^{\textnormal{top}} =(O^{\bowtie})^{\textnormal{top}}|_{(M_{\gamma(\alpha,\beta)})^{\textnormal{top}}} $  and
$o^{\textnormal{ag}}_{\gamma(\alpha,\beta)} = \mathfrak{I}^{\bowtie}\big(\Gamma^*(o_{\gamma(\alpha,\beta)})\big)$
the trivializations of $(O^{\bowtie}_{\gamma(\alpha,\beta)})^{\textnormal{top}}$ obtained using $\mathfrak{I}^{\bowtie}$ from \eqref{isomorphism DCO Ovartheta}. Let  $\phi^{\bowtie}$ be from Proposition \ref{proposition Ovartheta Hprincipal},
then it satisfies
$$
\phi^{\bowtie}(o^{\textnormal{ag}}_{\delta(\alpha_1,\beta_1)}\boxtimes_{\Z_2} o^{\textnormal{ag}}_{\delta(\alpha_2,\beta_2)}\big) =\epsilon_{\delta(\alpha_1,\beta_1),\delta(\alpha_2,\beta_2)} o^{\textnormal{ag}}_{\delta(\alpha_1,\beta_1) + \delta(\alpha_2,\beta_2)}\,.
$$
\end{proposition}
\begin{proof}
Recall the definition of $\tilde{Y}$, $\tilde{T}$ from Definition \ref{definition double}. One can express $D^{\mathcal{C}}_O(\tilde{Y})$ as a product of $\Z_2$-graded $\Z_2$-bundles $p_1^*(O_{\mathcal{C}}^{\slashed{D}_{\tilde{Y}}})\otimes p_2^*(O_{\mathcal{C}}^{\slashed{D}_{\tilde{Y}}})^*$  obtained by Proposition \ref{exun} from $O^{\slashed{D}_{\tilde{Y}}}$ in Example \ref{z2gradedbundles}, where $\mathcal{C}_{\tilde{Y}}\xleftarrow{p_1}\mathcal{C}_{\tilde{Y}}\times_{\mathcal{C}_{D}}\mathcal{C}_{\tilde{Y}}\xrightarrow{p_2} \mathcal{C}_{\tilde{Y}}$ are the projections and $\textnormal{deg}\big(O^{\slashed{D}_{\tilde{Y}}}_{\mathcal{C}}\big)|_{\mathcal{C}_{\tilde{\alpha}}} = \chi^{\mathbb{R}}_{\tilde{Y}}(\tilde{\alpha},\tilde{\alpha})$ . Using Definition \ref{dual} and Lemma \ref{lemma Z2graded} together with Joyce--Tanaka--Upmeier \autocite[2.26]{JTU}, a simple computation shows that for each $\tilde{\delta}_i(\tilde{\alpha}_i,\tilde{\beta}_i)$,  which under inclusion $\tilde{\iota}_{1,2}: \tilde{Y}\to \tilde{Y}\cup_{\tilde{T}}\tilde{Y}$ restrict to $\tilde{\alpha}_i,\tilde{\beta}_i$ respectively, we have the formula
\begin{align*}
\label{switching and signs on Y}
  &\tilde{\tau}_{\tilde{\delta}_2(\tilde{\alpha}_2,\beta_2), \tilde{\delta}_1(\tilde{\alpha}_1,\tilde{\beta}_1)} =&
  &(-1)^{\big(\chi^{\R}_{\tilde{Y}}(\tilde{\alpha}_1,\tilde{\alpha}_1) - \chi^{\R}_{\tilde{Y}}(\tilde{\beta}_1,\tilde{\beta}_1)\big)\big(\chi^{\R}_{\tilde{Y}}(\tilde{\alpha}_2,\tilde{\alpha}_2) - \chi^{\R}_{\tilde{Y}}(\tilde{\beta}_2,\tilde{\beta}_2)\big) + \chi^{\R}_{\tilde{Y}}(\tilde{\alpha}_1,\tilde{\alpha}_2) - \chi^{\R}_{\tilde{Y}}(\tilde{\beta}_1,\tilde{\beta}_2)}\\&
  \tilde{\tau}_{\tilde{\delta}_1(\tilde{\alpha}_1,\tilde{\beta}_1), \tilde{\delta}_2(\tilde{\alpha}_2,\tilde{\beta}_2)}
  \numberthis
\end{align*}
Two points $[E^\pm,F^\pm,\phi^\pm]$ of $\mathcal{V}_{Y}\times_{\mathcal{V}_{\bar{S}}}\mathcal{V}_{Y}$  map to $[\tilde{E}^{\pm},\tilde{F}^\pm,\tilde{\phi}^\pm]\in\mathcal{V}_{\tilde{Y},\tilde{T}}$, as described in \eqref{gluing equation}. Using excision on index and Definition \ref{definition differential euler form}, it follows that 
$$
\chi^{\R}_{\tilde{Y}}(\llbracket \tilde{E}^+\rrbracket,\llbracket \tilde{E}^-\rrbracket) - \chi_{\tilde{Y}}^{\R}(\llbracket \tilde{F}^+\rrbracket,\llbracket \tilde{F}^-\rrbracket)= \chi_{Y}(\llbracket E^+\rrbracket,\llbracket E^-\rrbracket) - \chi_{Y}(\llbracket F^+\rrbracket,\llbracket F^-\rrbracket)\,.
$$
Using \eqref{htgc} and biadditivity of $\chi$, we obtain 
$$
\chi^{\R}_{\tilde{Y}}(\tilde{\alpha}_1,\tilde{\alpha}_2) - \chi^{\R}_{\tilde{Y}}(\tilde{\beta}_1,\tilde{\beta}_2)= \chi_{Y}(\alpha_1,\alpha_2) - \chi_{Y}(\beta_1,\beta_2)\,,
$$
where $\tilde{\alpha}_i,\tilde{\beta}_i$ are K-theory classes glued from $\alpha_i,\beta_i$ as in \eqref{gluing equation}, from which we obtain 
\begin{align*}
\label{switching and signs}
  \tau_{\delta(\alpha_2,\beta_2), \delta(\alpha_1,\beta_1)} =& (-1)^{\big(\chi_{Y}(\alpha_1,\alpha_1) - \chi_{Y}(\beta_1,\beta_1)\big)\big(\chi_{Y}(\alpha_2,\alpha_2) - \chi_{Y}(\beta_2,\beta_2)\big) + \chi_{Y}(\alpha_1,\alpha_2) - \chi_{Y}(\beta_1,\beta_2)}\\
  &\tau_{\delta(\alpha_1,\beta_1), \delta(\alpha_2,\beta_2)}\,.  
\end{align*}
This leads to \eqref{compatibilitysignsdg} by using that $D^{\mathcal{C}}_O$ is strong H-principal. 
To conclude the final statement of the proposition, one applies Proposition \ref{proposition main}.
\end{proof}

Recall from Definition \ref{definition beforencmain} that we have the map $\Gamma^{\textnormal{cs}} : (\mathcal{M}_X)^{\textnormal{top}}\to \mathcal{C}^{\cs}_X$. There exists a compactly supported Chern character which is an isomorphism 
\begin{equation}
\label{cschern}
\textnormal{ch}_{\textnormal{cs}}: K^*_{\textnormal{cs}}(X)\otimes_{\mathbb{Z}}\mathbb{Q}\longrightarrow H^{*}_{\textnormal{cs}}(X,\mathbb{Q})\end{equation}
of $\Z_2$-graded rings. We also have the Euler form on $H^{\textnormal{even}}_{\textnormal{cs}}(X,\mathbb{Q})$:
\begin{align*}
\label{compactly supported chi}
& \bar{\chi}: H^{\textnormal{even}}_{\textnormal{cs}}(X, \mathbb{Q})\times H^{\textnormal{even}}_{\textnormal{cs}}(X, \mathbb{Q})\longrightarrow \mathbb{Q}\\
& \bar{\chi}(a,b)=\textnormal{deg}(a^\vee\cdot b\cdot \textnormal{td}(TX))_4 \,. 
\numberthis
\end{align*}
Combining \eqref{cschern} and \eqref{compactly supported chi}, one gets $
    \bar{\chi}:K^0_{\cs}(X)\times K^0_{\cs}(X)\to \Z\,.$
    Note that, we have $\bar{\chi}(\alpha,\beta) =\chi_{Y}(\bar{\alpha},\bar{\beta})$, where for a class $\alpha\in K^0_{\cs}(X)$, $\bar{\alpha}$ denotes the class of $K^0(Y)$ extended trivially.

Recall from Example \ref{algebraicbundles} that we have the isomorphisms $\phi^\omega:O^{\omega}\boxtimes_{\Z_2}O^{\omega}\to \mu^*(O^\omega)$. We can now state the main result of comparison of signs under sums in non-compact Calabi--Yau fourfolds. 

\begin{theorem}
\label{signscomparison}
Assume that $\tilde{Y}$ from Definition \ref{def:UandT-defined} satisfies Assumption \ref{ass:astcs}. 
Let $\mathcal{C}^{\textnormal{cs}}_{\alpha}$ denote the connected component of $\mathcal{C}_X^{\textnormal{cs}}$ corresponding to $\alpha\in K^0_{\textnormal{cs}}(X) =\pi_0(\mathcal{C}^{\textnormal{cs}}_X)$\, and $O^{\cs}_\alpha = O^{\cs}|_{\mathcal{C}^{\cs}_\alpha}$. Let $(\mathcal{M}_{\alpha})^{\textnormal{top}}=(\Gamma^{\textnormal{cs}})^{-1}(\mathcal{C}^{\textnormal{cs}}_{\alpha})$. After fixing choices of trivializations $o^{\textnormal{cs}}_\alpha$ of $O^{\textnormal{cs}}_\alpha$, we define $\epsilon_{\alpha,\beta}$
$$
\phi^\omega\Big(\mathfrak{I}\big(\Gamma^*o^{\textnormal{cs}}_\alpha\big)\boxtimes \mathfrak{I}\big(\Gamma^*o^{\textnormal{cs}}_{\beta}\big)\Big) \cong \epsilon_{\alpha,\beta}\mathfrak{I}\big(\Gamma^*o^{\textnormal{cs}}_{\alpha+\beta}\big)\,.
$$
If one moreover fixes  the preferred choice of $o^{\cs}_0$, such that \begin{equation}
\label{preferredchoice}
\tau^{\textnormal{cs}}(o^{\textnormal{cs}}_0\boxtimes o^{\textnormal{cs}}_0)=o^{\textnormal{cs}}_0\,,
\end{equation}
then  $\epsilon:(\alpha,\beta)\mapsto \epsilon_{\alpha,\beta}\in \{\pm 1 \}$ is up to equivalences the unique group 2-cocycle satisfying
\begin{equation}
\label{skewsymmetry}
    \epsilon_{\alpha,\beta}=(-1)^{\bar{\chi}(\alpha,\alpha)\bar{\chi}(\beta,\beta)+\bar{\chi}(\alpha,\beta)}\epsilon_{\beta,\alpha}
\end{equation}
\end{theorem}
\begin{proof}
Recall from the proof of Theorem \ref{maintheorem} that we have the isomorphism $\zeta^*\big(O^{\bowtie}\big)\cong O^\omega$ which is by construction in Proposition \ref{proposition Ovartheta Hprincipal} strong H-principal as can be checked directly by comparing $\Z_2$-torsors at each $\text{Spec}(A)$-point $[i_*(E),0]$.

By Proposition \ref{proposition sign} after setting $\beta_1$ and $\beta_2$ equal to 0 and $\alpha_1=\bar{\alpha}$, $\alpha_2 = \bar{\beta}$ it then follows that, we have \eqref{skewsymmetry} together with
\begin{equation}
\label{cocycle condition}
      \epsilon_{\alpha,0} = \epsilon_{0,\alpha}=1\,,\qquad
    \epsilon_{\alpha,\beta}\epsilon_{\alpha+\beta,\gamma}=\epsilon_{\beta,\gamma}\epsilon_{\alpha,\beta+\gamma}
\end{equation}
which it precisely the cocycle condition. The first condition follows from \eqref{preferredchoice}.
\end{proof}

We now discuss the orientation group from Joyce--Tanaka--Upmeier \autocite[Definition 2.26]{JTU} applied to $K^0_{\cs}(X)$ instead of $K^0(X)$ in our non-compact setting. The \textit{compactly supported orientation group} is defined as 
$$
\Omega_{\textnormal{cs}}(X) = \{(\alpha,o^{\textnormal{cs}}_{\alpha}): \alpha\in K_{\textnormal{cs}}^{0}(X), o^{\cs}_\alpha \textnormal{orientation on }\mathcal{C}^{\cs}_\alpha\}\,.
$$
The multiplication is given by 
$
(\alpha, o^{\cs}_\alpha)\star(\beta,o^{\cs}_\beta) =  \big(\alpha + \beta, \tau^{\cs}_{\alpha,\beta}(o^{\cs}_\alpha\boxtimes_{\Z_2} o^{\cs}_\beta)\big)\,.
$
 The resulting group is the unique group extension $0\to \Z_2\to \Omega_{\textnormal{cs}}(X)\to K^0_{\textnormal{cs}}(X)\to 0$ for the group 2-cycle $\epsilon$ of Theorem \ref{signscomparison}. Choices of orientations induce a splitting $K^0_{\textnormal{cs}}\to \Omega_{\textnormal{cs}}(X)$ as sets.
This fixed $\epsilon:K^0_{\textnormal{cs}}(X)\times K^0_{\textnormal{cs}}(X)\to \Z_2$ in Theorem \ref{signscomparison}.
Let us describe the method used in  \autocite[Thm. 2.27]{JTU} for extending orientations. Choosing generators of $K^0_{\textnormal{cs}}(X)$, one obtains
\begin{equation}
\label{Ktheory trivilization }
 K^0_{\cs}(X) \cong \Z^r\times \prod_{k=1}^p\Z_{m_k}\times\prod_{j=1}^q\Z_{2^{p_j}}\,,
\end{equation}
where $m_k>2$ odd and $p_j>0$.  Fixing a choice of isomorphism \eqref{Ktheory trivilization }, choose orientation on each $C^{\cs}_{\alpha_i}$, $\alpha_i=
(0,\ldots, 0,1,,0\ldots,0)$,
where 1 is in position $i$. Use $\tau^{\cs}$ to obtain orientations for all $\alpha\in K^0_{\textnormal{cs}}(X)$ by adding generators going from left to right in the form $(a_1,\ldots, a_p, (b_j)^q_{j=1}, (c_k)^p_{k=1})$ and using in each step
$o^{\textnormal{cs}}_{\alpha' + g} = \tau^{\textnormal{cs}}(o^{\cs}_{\alpha'}\boxtimes o^{\cs}_g),$
where $g$ is a generator. As a result, one obtains the splitting:
\begin{equation}
\label{trivialization isomorphism}
 \textnormal{Or}(\mathfrak{o}):\Omega_\cs(X)\cong K^0_{\cs}(X)\times \{-1,1\} \cong \Z^r\times \prod_{k=1}^q\Z_{m_k}\times\prod_{j=1}^p\Z_{2^{p_j}}\times \{-1,1\}\,, 
\end{equation}
where $\mathfrak{o}$ is the set of orientation on $C^{\cs}_\alpha$ for the chosen generators $\alpha$. Let $\bar{\chi}_{ij} := \bar{\chi}(\alpha_i,\alpha_j)$. The next result replaces K-theory by compactly supported K-theory in Joyce--Tanaka--Upmeier \cite[Thm. 2.27]{JTU} and considers the $\Z_2$-bundle $O^{\textnormal{cs}}$ we constructed.
\begin{theorem}[{\cite[Thm. 2.27]{JTU}}]
\label{theorem oreintation group}
Assume that $\tilde{Y}$ from Definition \ref{def:UandT-defined} satisfies Assumption \ref{ass:astcs}. Let $Or(\mathfrak{o})$ be the isomorphism \eqref{trivialization isomorphism} for a given choice of orientations $\mathfrak{o}$ on generators corresponding to the isomorphism \eqref{Ktheory trivilization }. Let $T_2$ be the 2-torsion subgroup of $K^0_{\cs}(X)$. Then:
\begin{enumerate}[label=(\roman*)]
\item Define the map $\xi: T_2\to \mathbb{Z}_2$ as $\Xi(\gamma) = \epsilon_{\gamma,\gamma}$, Then it is a group homomorphism. 
\item Using $\textnormal{Or}(\mathfrak{o})$ from \eqref{trivialization isomorphism} to identify $\Omega^{\cs}(X)$ with $\Z^{r}\times\prod_{k=1}^p\Z_{m_k} \times\prod_{j=1}^q\Z_{2^{p_j}}\times \{-1,1\}$ the induced group structure on the latter becomes
\begin{align*}
   &\Big(a_1,\ldots, a_r, (b_j)_{j=1}^{p}, (c_k)_{k=1}^q, o\Big)\star\Big(a'_1,\ldots, a'_r, (b'_j)_{j=1}^{p} ,(c'_k)_{k=1}^q, o'\Big)\\
   &=\Big(a_1+a'_1,\ldots, a_r+a'_r, (b_j-b'_j)^p_{j=1}, (c_k + c'_k)_{k=1}^q,\\& (-1)^{\sum_{1\leq h<i\leq r}(\bar{\chi}_{hi}+\bar{\chi}_{hh}\bar{\chi}_{ii})a'_ha_i}\Xi(\gamma) o\cdot o'\Big)\,,
\end{align*}
where $\gamma = (0,\ldots, 0, (0)_{k=1}^q, (\tilde{c}_j)_{j=1}^{q})$ and $$\tilde{c}_j=\bigg\lfloor\frac{\bar{c}_j+\bar{c}'_j}{\overline{c_j+c'_j}} \bigg\rfloor 2^{p_j-1}$$ for the unique representatives $0\leq \bar{c}_j,\bar{c}'_j, \overline{c_j+c'_j}<2^{p_j}$.
\end{enumerate}
\label{theorem choice of orientation}
\end{theorem}

\printbibliography
\end{document}